\documentclass[10pt]{amsart}

\usepackage{fancyhdr}
\usepackage{bm}
\usepackage{amssymb}

\newtheorem{thm}{Theorem}[section]
\newtheorem{cor}[thm]{Corollary}
\newtheorem{lem}[thm]{Lemma}
\newtheorem{prop}[thm]{Proposition}
\theoremstyle{definition}

\newcommand{\remarkstyle}[1]{\medskip\noindent\underline{#1}\;}
\def\goth{\mathfrak}

\def\longra{\longrightarrow}

\def\a{\alpha}
\def\b{\beta}
\def\d{\delta}
\def\e{\epsilon}
\def\g{\gamma}

\def\l{\lambda}

\def\vp{\varphi}

\def\Ma{\mathcal{M}}

\def\Oa{\mathcal{O}}

\def\DD{\mathbb{D}}
\def\FF{\mathbb{F}}
\def\GG{\mathbb{G}}

\def\QQ{\mathbb{Q}}

\def\W{\mathbb{W}}
\def\Z{\mathbb{Z}}

\def\Aut{\mathrm{Aut}}
\def\Ext{\mathrm{Ext}}
\def\Gal{\mathrm{Gal}}
\def\Hom{\mathrm{Hom}}
\def\Ker{\mathrm{Ker}}

\def\Tor{\mathrm{Tor}}

\def\lim{\mathrm{lim}}

\newcommand{\f}{\mathbb{F}_{p^n}}
\newcommand{\F}{\mathbb{F}}
\newcommand{\G}{\mathbb{G}}
\newcommand{\s}{\mathbb{S}}
\newcommand{\wv}{\mathbb{W}}

\date{\today} 
\begin{document}

\title{The homotopy of the $K(2)$-local Moore spectrum %\\
at the prime $3$ revisited} 

\author{Hans-Werner Henn, Nasko Karamanov and Mark Mahowald}

\address{Institut de Recherche Math\'ematique Avanc\'ee,
C.N.R.S. - Universit\'e Louis Pasteur, F-67084 Strasbourg,
France}
\address{Ruhr-Universit\"at Bochum, Fakult\"at f\"ur Mathematik, D-44780, 
Germany} 
\address{Department of Mathematics, 
Northwestern University, Evanston, IL 60208, U.S.A.}  
\thanks{The authors would like to thank the Mittag-Leffler Institute,  
Northwestern University, Universit\'e Louis Pasteur at Strasbourg  
and the Ruhr-Universit\"at Bochum
for providing them with the opportunity to work together.}

\begin{abstract}
In this paper we use the approach introduced in \cite{GHMR}  
in order to analyze the homotopy groups of $L_{K(2)}V(0)$, the 
mod-$3$ Moore spectrum $V(0)$ localized with respect to Morava 
$K$-theory $K(2)$. These homotopy groups have already been calculated 
by Shimomura \cite{shi1}. The results are very complicated 
so that an independent verification via an alternative approach 
is of interest. In fact, we end up with a result which is more precise 
and also differs in some of its details from that of \cite{shi1}. 
An additional bonus of our approach is that it   
breaks up the result into smaller and more digestible chunks 
which are related to the $K(2)$-localization of the spectrum $TMF$ 
of topological modular forms and related spectra. Even more, the 
Adams-Novikov differentials for $L_{K(2)}V(0)$ can be 
read off from those for $TMF$.

%{\bf Indicate discrepancies with Shimomura, in particular $\b$ and $\zeta$-
%multiplications!}  
\end{abstract}

\maketitle
 
\section{Introduction}

Let $K(2)$ be the second Morava $K$-theory for the 
prime $3$. For suitable spectra $F$, e.g. if $F$ is a finite spectrum,   
the homotopy groups of the Bousfield localization 
$L_{K(2)}F$ can be calculated via the Adams-Novikov spectral sequence. 
By \cite{DH} 
this spectral sequence can be identified with the descent spectral sequence 
$$
E_2^{s,t}=H^s(\G_2,(E_2)_tF)\Longrightarrow \pi_{t-s}(L_{K(2)}F) 
$$
for the action of the (extended) 
Morava stabilizer group $\G_2$ on $E_2\wedge F$ where 
the action is via the Goerss-Hopkins-Miller 
action on the Lubin-Tate spectrum $E_2$   
(see \cite{GHMR} for a summary of the necessary background material). 
Here we just recall that the homotopy groups of $E_2$ 
are non-canonically isomorphic to 
$\W_{\FF_9}[[u_1]][u^{\pm 1}]$ where $\W_{\F_9}$ 
denotes the ring of Witt vectors of $\F_9$, where $u_1$ is of degree 
$0$ and $u$ is of degree $-2$. 
We also recall that $\G_2$ is a profinite group 
and its action on the profinite module $(E_2)_*F$ is continuous; 
group cohomology is, throughout this paper, taken   
in the continuous sense.   

The cohomological dimension of $\G_2$ is well-known to be infinite and 
therefore a finite projective resolution of the trivial profinite 
$\G_2$-module $\Z_3$ cannot exist. 
However, in \cite{GHMR} a finite resolution of the trivial module $\Z_3$ 
was constructed in terms of permutation modules. More precisely, 
the group $\G_2$ is isomorphic to the product $\G_2^1\times \Z_3$ 
of a central subgroup (isomorphic to) $\Z_3$ and a group 
$\G_2^1$ which is the kernel of a homomorphism $\G_2^1\to \Z_3$, 
also called the reduced norm. 
One of the main technical achievements of \cite{GHMR} was the construction 
of a permutation resolution of the trivial module $\Z_3$ for the 
group $\G_2^1$. This resolution is self-dual in a suitable sense 
(cf. section \ref{partial3}) and has the form 
\begin{eqnarray}\label{alg}
0\rightarrow C_3\rightarrow C_2\rightarrow C_1\rightarrow C_0\rightarrow 
Z_3\rightarrow 0    
\end{eqnarray} 
with $C_0=C_3=\Z_3[[\G_2^1/G_{24}]]$ and 
$C_1=C_2=\Z_3[[\G_2^1]]\otimes_{\Z_3[SD_{16}]}\chi$.  
Here $G_{24}$ is a certain subgroup of $\G_2^1$ of order $24$, 
isomorphic to the semidirect product $\Z/3\rtimes Q_8$ of the 
cyclic group of order $3$ with a non-trivial action of the 
quaternion group $Q_8$, and $SD_{16}$ is another subgroup, 
isomorphic to the semidihedral group of order $16$ 
(see section \ref{subgroups}). Furthermore, $\chi$ is a suitable 
one-dimensional representation of $SD_{16}$, defined over $\Z_3$, 
and if  $S$ is a profinite $\G_2^1$-set 
we denote the corresponding profinite permutation module by 
$\Z_3[[S]]$.   

For any $\Z_3[[\GG_2^1]]$-module $M$ the resolution (\ref{alg}) 
gives rise to a first quadrant cohomological 
spectral sequence  
\begin{eqnarray}\label{algss}
E_1^{s,t}=\Ext_{\Z_3[[\G_2^1]]}^t(C_s,M)\Longrightarrow 
H^{s+t}(\G_2^1,M) 
\end{eqnarray}
refered to in the sequel as the algebraic spectral sequence.  
By Shapiro's Lemma we have 
\begin{equation}\label{Shapiro}
E_1^{0,t}=E_1^{3,t}\cong H^t(G_{24},M), \ \ 
E_1^{1,t}=E_1^{2,t}\cong 
\begin{cases} 
\Hom_{\Z_3[SD_{16}]}(\chi,M) & t=0 \\
                             0 & t>0  \ . \\
\end{cases}
\end{equation}  
The bulk of our work is the calculation of 
this spectral sequence if   $M=(E_2)_*(V(0))$.
In this case the $E_1$-term is well understood and 
can be interpreted in terms of modular forms in 
characteristic $3$. In fact, it is determined by the following result 
which we include for the convenience of the reader and 
in which $v_1$ denotes the well-known $\GG_2$-invariant class 
$u_1u^{-2}\in M_4$. For the definition of the 
other classes figuring in this result the reader is referred to section 
\ref{e1-term}.     

\begin{thm} \label{shapiro} Let $M=(E_2)_*(V(0))$. 

a) There are elements $\beta\in H^2(G_{24},M_{12})$, 
$\a\in H^1(G_{24},M_4)$ and $\widetilde{\a}\in H^1(G_{24},M_{12})$, 
an invertible $G_{24}$-invariant element $\Delta\in M_{24}$, 
and an isomorphism of graded algebras 
$$
H^*(G_{24},M)\cong 
\F_3[[v_1^6\Delta^{-1}]][\Delta^{\pm 1},v_1,\beta,\alpha,\widetilde{\alpha}]
/(\alpha^2,\widetilde{\alpha}^2,v_1\alpha,v_1\widetilde{\alpha},
\alpha\widetilde{\alpha}+v_1\beta)\ . 
$$

b) The ring of $SD_{16}$-invariants of $M$ is given 
by the subalgebra $M^{SD_{16}}=\FF_3[[u_1^4]][v_1,u^{\pm 8}]$ 
and $\Hom_{\Z_3[SD_{16}]}(\chi,M)$ is a free 
$M^{SD_{16}}$-module of rank $1$ with generator $\omega^2u^4$, i.e.  
$$
\mathrm{Hom}_{\Z_3[SD_{16}]}(\chi,M)\cong 
\omega^2u^4\F_3[[u_1^4]][v_1,u^{\pm 8}]\ . \qed 
$$
\end{thm}

\remarkstyle{Remark} 
We note that $v_1^6\Delta^{-1}$ is a $G_{24}$-invariant 
class in the maximal ideal of $M_0$ and hence a formal power series 
in $v_1^6\Delta^{-1}$ converges in $M$ and is also invariant. Similarly 
with $u_1^4$. Of course, the name for $\Delta$ is chosen to emphasize 
the close relation with the theory of modular forms. For example 
we note that $M^{G_{24}}$ is isomorphic to the completion of  
$\Ma_3:=\F_3[\Delta^{\pm 1},v_1]$ with respect to the ideal generated 
by $v_1^6\Delta^{-1}$, and $\Ma_3$ is isomorphic to the ring of 
modular forms in characteristic $3$ (cf. \cite{De} and \cite{behk}). 
Similarly, $M^{SD_{16}}$ is isomorphic to the 
completion of $\F_3[v_1,u^{\pm 8}]$ with respect to the ideal 
generated by $u_1^4=v_1^4u^8$. The larger algebra $\FF_3[v_1,u^{\pm 4}]$   
is isomorphic to the ring $\Ma_3(2)$ of modular forms of level $2$ 
(in characteristic $3$) (cf. \cite{behk}). 
The relation with modular forms could be made tight  
if in \cite{GHMR} we had worked with a version of $E_2$ which uses a 
deformation of the formal group of a supersingular curve rather than 
that of the Honda formal group.   
 
As $(E_2)_*(V(0))$ is a graded module, the spectral sequence is trigraded. 
The differentials in this spectral sequence are $v_1$-linear and continuous. 
Therefore $d_1$ is completely described by continuity and 
the following formulae in which we identify the $E_1$-term via  
Theorem \ref{shapiro}.       

\begin{thm}\label{d1} There are elements 
$$
\begin{array}{llll}
\Delta_k\in E_1^{0,0,24k},&  b_{2k+1}\in E_1^{1,0,16k+8}, 
& \overline{b}_{2k+1}\in E_1^{2,0,16k+8}, 
& \overline{\Delta}_k\in E_1^{3,0,24k} 
\end{array}
$$ 
for each $k\in \Z$ satisfying   
$$
\begin{array}{llll} 
\Delta_k\equiv\Delta^k, & b_{2k+1}\equiv \omega^2u^{-4(2k+1)}, 
& \overline{b}_{2k+1}\equiv \omega^2u^{-4(2k+1)}, 
& \overline{\Delta}_k\equiv\Delta^k
\end{array}
$$  
(where the congruences are modulo the ideal 
$(v_1^6\Delta^{-1})$ resp. $(v_1^4u^8)$ and in the case of $\Delta_0$ we even 
have equality $\Delta_0=\Delta^0=1$) %{\bf (RECHECK!!)} 
such that  
$$
\begin{array}{rcl}
d_1(\Delta_k)&=&  
\begin{cases}
(-1)^{m+1}b_{2.(3m+1)+1}                  &k=2m+1 \\
(-1)^{m+1}mv_1^{4.3^n-2}b_{2.3^n(3m-1)+1} &k=2m.3^n,  m\not\equiv 0\ (3)\\ 
0                                         &k=0  \\
\end{cases}   \\
&&\\
%&&\\
d_1(b_{2k+1})&=&  
\begin{cases}
(-1)^nv_1^{6.3^{n}+2}\overline{b}_{3^{n+1}(6m+1)}
&\ \ \ \ \ \ \ k=3^{n+1}(3m+1) \\
(-1)^nv_1^{10.3^n+2}\overline{b}_{3^n(18m+11)}   
&\ \ \ \ \ \ \ k=3^{n}(9m+8)  \\ 
0                                                
&\ \ \ \ \ \ \ \textrm{else}\\
\end{cases}  \\
&&\\
\end{array}
$$
$$
\begin{array}{rcl}
&&\\
d_1(\overline{b}_{2k+1})&=&
\begin{cases}
(-1)^{m+1}v_1^{2}\overline{\Delta}_{2m}             & 2k+1=6m+1 \\
(-1)^{m+n}v_1^{4.3^n}\overline{\Delta}_{3^{n}(6m+5)}& 2k+1=3^{n}(18m+17) \\
(-1)^{m+n+1}v_1^{4.3^n}\overline{\Delta}_{3^{n}(6m+1)}& 2k+1=3^{n}(18m+5)\\
0                                         & \textrm{else}  \ .
\end{cases} 
\end{array} 
$$ 
\end{thm}

It turns out that the $d_2$-differential of this spectral sequence 
is determined by the following principles: it is non-trivial if and only 
if $v_1$-linearity and sparseness of the resulting $E_2$-term permit it, 
and in this case it is determined up to sign by these two properties. 
The remaining $d_3$-differential turns out to be trivial. 
More precisely we have the following result. 
 
\begin{prop} \label{d2+d3} \mbox{ }

a) The differential $d_2:E_2^{0,1,*} \to E_2^{2,0,*}$ is determined by
$$
\begin{array}{lcl}
d_2(\Delta_k \alpha)= 
\begin{cases}
(-1)^{m+n+1}v_1^{6.3^n+1}\overline{b}_{3^{n+1}(6m+1)}        &k=2.3^n(3m+1) \\
(-1)^{m+n}  v_1^{10.3^{n+1}+1}\overline{b}_{3^{n+1}(18m+11)} &k=2.3^n(9m+8) \\
0                                                    &\textrm{else}
\end{cases}
&&\\ 
&&\\
d_2(\Delta_k\widetilde{\alpha})=
\begin{cases}
(-1)^{m}v_1^{11}\overline{b}_{18m+11}         & k=6m+5  \\
0                                               &\textrm{else} \ .
\end{cases} 
\end{array} 
$$

b) The $d_3$-differential is trivial. 
\end{prop} 

\remarkstyle{Remark 1 on notation} Of course, the elements $\Delta_k\a$ and 
$\Delta_k\widetilde{\a}$ are only names for elements in the $E_2$-term 
which are represented in the $E_1$-term as products, but which are 
no longer products in the $E_2$-term. Similar abuse of notation will be used 
in Theorem \ref{E-infinity}, Proposition \ref{extensions}, 
Theorem \ref{H*G21} and in section 6 and 8.

Next we use that the element $\b$ of Theorem \ref{shapiro} 
lifts to an element with the same name in $H^2(\GG_2^1,M_{12})$ 
resp. in $H^2(\GG_2,M_{12})$. In fact this latter element 
detects the image of $\b_1\in \pi_{10}(S^0)$ in $\pi_{10}(L_{K(2)}V(0))$. 
The previous results yield the following $E_{\infty}$-term as a 
module over $\F_3[\beta,v_1]$.             

\begin{thm}\label{E-infinity} As an $\F_3[\beta,v_1]$-module 
the $E_{\infty}$-term of the algebraic spectral sequence (\ref{algss}) 
for $M=(E_2)_*/(3)$ is isomorphic to a direct sum of cyclic modules 
generated by the following elements and with the following 
annihilator ideals: 

a) For $E_{\infty}^{0,*,*}$  we have the following generators with respective 
annihilator ideals  
$$
\begin{array}{lll}
1=\Delta_0                 &               & (\b v_1^2) \\
\Delta_m\beta&  \ m\neq 0                  &(v_1^2) \\
%&\\
\alpha                      &   &(v_1) \\
\Delta_{2m+1}\alpha &&(v_1) \\
\Delta_{2.3^n(3m-1)}\alpha&\ m\not\equiv 0 \mod(3) & (v_1) \\
\Delta_{2m}\widetilde{\alpha} &  & (v_1)\\
\Delta_{2m+1}\widetilde{\alpha}&\ m\not\equiv 2\mod(3) & (v_1) \\
\Delta_{2.3^n(3m+1)}\alpha\b   & & (v_1) \\
\Delta_{2.3^n(3m-1)}\alpha\b&\ m\equiv 0 \mod (3)    & (v_1) \\
\Delta_{2m+1}\widetilde{\alpha}\b&\ m\equiv 2\mod(3)   & (v_1) \ .\\
\end{array}
$$

b) For $E_{\infty}^{1,*,*}$  we have the following generators 
with respective annihilator ideals 
$$
\begin{array}{lll}
b_1              &                     & (\beta) \\ 
b_{2.3^n(3m-1)+1}&\  m\not\equiv 0\mod (3)& (v_1^{4.3^n-2},\beta) \ .\\ 
\end{array}
$$

c) For $E_{\infty}^{2,*,*}$  we have the following generators with respective 
annihilator ideals 
$$
\begin{array}{lll}
\overline{b}_{3^{n+1}(6m+1)} &&(v_1^{6.3^n+1},\beta) \\ 
\overline{b}_{3^{n}(6m+5)}&\ m\equiv 1\mod (3) &(v_1^{10.3^n+1},\beta) \ .\\ 
\end{array}
$$

d) For $E_{\infty}^{3,*,*}$  we have the following generators with respective 
annihilator ideals 
$$
\begin{array}{lll}
\overline{\Delta}_{2m}               &   &(v_1^2)\\
\overline{\Delta}_{3^n(6m\pm 1)}  &&(v_1^{4.3^n},\b v_1^2) \\
\overline{\Delta}_{m}\alpha        &     &(v_1)\\
\overline{\Delta}_{m}\widetilde{\alpha}& &(v_1)\ . \qed \\
\end{array}
$$
\end{thm}

To get at $H^*(\G_2^1,(E_2)_*/(3))$ we still need to know the extensions 
between the filtration quotients. 
They are given by the following result.

\begin{prop} \label{extensions} The $\F_3[\b,v_1]$-module generators of 
the $E_{\infty}$-term of Theorem \ref{E-infinity} can be 
lifted to elements (with the same name) in 
$H^*(\mathbb{G}_2^1;(E_2)_*/(3))$  such that the 
relations defining the annihilator ideals of Theorem \ref{E-infinity}
continue to hold with the following exceptions      
$$
\begin{array}{lll} 
v_1\alpha               &=&b_1\\
v_1\Delta_{2.3^n (9m+2)}\alpha   &=&(-1)^{m+1}{b}_{2.3^{n+1}(9m+2)+1}\\  
v_1\Delta_{2.3^n (9m+5)}\alpha   &=&(-1)^{m+1}{b}_{2.3^{n+1}(9m+5)+1}\\
&\\
v_1\Delta_{6m+1}\widetilde{\alpha}  &= &(-1)^{m}{b}_{2(9m+2)+1}\\ 
v_1\Delta_{6m+3}\widetilde{\alpha}  &= &(-1)^{m+1}{b}_{2(9m+5)+1} \\
&\\
\beta \overline{b}_{3^{n+1}(6m+1)}&=&
\pm \overline{\Delta}_{3^n(6m+1)}\widetilde{\alpha}\\
\beta \overline{b}_{3^{n+1}(18m+11)}&=&
\pm \overline{\Delta}_{3^n(18m+11)}\widetilde{\alpha} \\   
\beta \overline{b}_{18m+11}&=&
\pm \overline{\Delta}_{6m+4}\alpha \ .\\
\end{array}
$$ 
\end{prop}

Apart from the last group of $\beta$-extensions 
(which are simple consequences of the calculation 
of $H^*(\GG_2,(E_2)_*/(3,u_1))$, cf. \cite{GHM}) 
one can summarize the result by saying that nontrivial 
$v_1$-extensions exist only between $E_{\infty}^{0,1,*}$ and 
$E_{\infty}^{1,0,*+4}$ and there is such an extension 
whenever sparseness permits it, and then the corresponding relation is 
unique up to sign. Unfortunately this is not clear a priori, but needs  
proof and the proof gives the exact value of the sign.    
In contrast determining the sign for the $\b$-relations would require an 
extra effort.  

The main results can now be stated as follows. 

\begin{thm}\label{H*G21} As an $\F_3[\beta,v_1]$-module 
$H^*(\G_2^1,(E_2)_*/(3))$ is isomorphic to the 
direct sum of the cyclic modules generated by the following elements 
and with the following annihilator ideals   
$$
\begin{array}{lll}
1=\Delta_0                & & (\b v_1^2) \\
\Delta_m\beta & \ m\neq 0  &(v_1^2) \\
%&\\
\alpha          &                & (\b v_1) \\
\Delta_{2m+1}\alpha &            &(v_1) \\
\Delta_{2.3^n(3m-1)}\alpha& \ m\not\equiv 0\mod (3) & (v_1^{4.3^{n+1}-1},\b v_1) \\
\Delta_{2m}\widetilde{\alpha} &  & (v_1)\\
\Delta_{2m+1}\widetilde{\alpha}& \ m\not\equiv 2\mod (3) & (v_1^{3},\b v_1) \\
\Delta_{2.3^n(3m+1)}\alpha\b & & (v_1) \\
\Delta_{2.3^n(3m-1)}\alpha\b&\ m\equiv 0\mod (3)              & (v_1) \\
\Delta_{2m+1}\widetilde{\alpha}\b& \ m\equiv 2\mod (3)        & (v_1) \\
&\\
\overline{b}_{3^{n+1}(6m+1)}  && (v_1^{6.3^n+1},\b v_1) \\ 
\overline{b}_{3^{n}(6m+5)}&\ m\equiv 1\mod(3)  & (v_1^{10.3^n+1},\b v_1) \\ 
&\\
\overline{\Delta}_{2m}             &     &(v_1^2)\\
\overline{\Delta}_{3^n(6m\pm 1)}   &     &(v_1^{4.3^n},\b v_1^2) \\
&\\
%\end{array}
%$$
%$$
%\begin{array}{ll}
\overline{\Delta}_{2m+1}\alpha && (v_1)    \\
\overline{\Delta}_{2m}\alpha & \ m\not\equiv 2\mod (3) &(v_1)\\
\overline{\Delta}_{2m}\widetilde{\alpha} &               &(v_1)\\
\overline{\Delta}_{3^n(6m+5)}\widetilde{\alpha}&\ m\not\equiv 1\mod (3)&(v_1)\ .\qed \\
\end{array}
$$
\end{thm}

We emphasize that even though this result is involved the mechanism 
which produces is quite transparent. The passage to the cohomology of 
$\GG_2$ results now from the decomposition $\GG_2\cong \Z_3\times \GG_2^1$ 
and the fact that the central factor $\Z_3$ acts trivially on $(E_2)_*/(3)$.

\begin{thm}\label{H*G2} There is a class $\zeta\in H^1(\G_2,(E_2)_0/(3))$ 
and an isomorphism of graded algebras 
$$
H^*(\G_2^1;(E_2)_*/(3))\otimes_{\Z_3}\Lambda_{\Z_3}(\zeta)
\cong H^*(\G_2,(E_2)_*/(3))\  .\qed
$$
\end{thm} 

\remarkstyle{Remark} We warn the reader that there is something 
subtle about this K\"unneth type isomorphism.   
In fact, the class $\a$ of Theorem 
\ref{shapiro} is defined via the Greek letter formalism in 
$H^1(G_{24},-)$ as the Bockstein of the class $v_1$ with respect to   
the obvious short exact sequence 
$$
0\to (E_2)_*/(3)\buildrel{3}\over\longra (E_2)_*/(9)\longra (E_2)_*/(3)\to 0 \ .
$$ 
The same formalism allows to define classes $\a(F)\in H^1(F,(E_2)_4/(3))$ 
for any closed subgroup $F$ of $\GG_2$ and these classes are well  
compatible with respect to restrictions among different subgroups. However, with 
respect to the isomorphism of Theorem \ref{H*G2} the class $\a(\GG_2)$ 
corresponds to $\a(\GG_2^1)-v_1\zeta$ (cf. Corollary \ref{Bockstein}). 
We will insist on the notation $\a(\GG_2)$ and $\a(\GG_2^1)$
in order to avoid possible confusion when we deal with $H^*(\GG_2,-)$. 
Fortunately similar notation is unnecessary for the classes 
$\widetilde{\a}$ and $\b$ (cf. Corollary \ref{Bockstein}). 
The difference between $\a(\GG_2)$ and $\a(\GG_2^1)$  
turns out to be important for studying the differentials in the 
Adams-Novikov spectral sequence for $\pi_*(L_{K(2)}V(0))$.

In fact, these differentials can be derived from those 
of the Adams-Novikov spectral sequence for $\pi_*(L_{K(2)}V(1))$ 
which have been determined in \cite{GHM}.   

\remarkstyle{Remark 2 on notation} In the following theorem we 
give the $E_{\infty}$-term of the 
Adams-Novikov spectral sequence for $\pi_*(L_{K(2)}V(0))$ as a 
subquotient of its $E_2$-term which itself 
has been described in Theorem \ref{H*G21} and Theorem \ref{H*G2} 
as a module over $\FF_3[\b,v_1]\otimes \Lambda(\zeta)$ 
with generators represented in the $E_1$-term of the algebraic 
spectral sequence (\ref{algss}) for $M=(E_2)_*V(0)$. 
As before,  generators of $E_{\infty}$ which are represented by 
products in this $E_1$-term are not necessarily products in $E_{\infty}$. 
In order to distinguish between module multiplication and the 
name of a generator we write $\b$ and $v_1$ as right hand  
factors in such a product if they are only part of the name 
of a generator, e.g. in the case of $\Delta_{6m+1}\b v_1$. 
We have also renamed (for reasons which will be explained  below) 
generators involving $\overline{\Delta}_k$ 
by $\Sigma^{48}\overline{\Delta}_{k-2}$.

\begin{thm}\label{E-infty}  As a module over $\FF_3[\b,v_1]\otimes\Lambda(\zeta)$ 
the $E_{\infty}$-term of the Adams-Novikov spectral sequence for $\pi_*(L_{K(2)}V(0))$  
is the quotient of the direct sum of 
cyclic $\FF_3[\b,v_1]\otimes\Lambda(\zeta)$-modules 
with the following generators and annihilator ideals   
$$
\begin{array}{llll}
1=\Delta_0     &         &  & (\b v_1^2,\b^3v_1,\b^6) \\
&\\
\Delta_{3m}\beta,& m\neq 0&&(v_1^2,\b^2v_1,\b^5) \\
\Delta_{6m+1}\b v_1  &    &              &(v_1,\b^2) \\
\Delta_{6m+4}\b v_1 &    &              &(v_1,\b^3) \\
%\end{array}
%$$ 
%$$
%\begin{array}{ll}
&\\
\alpha(\GG_2^1)     &       &              &(\b v_1,\b^3) \\
\Delta_{2m+1}\alpha(\GG_2^1)  &m\not\equiv 2\mod(3)&     &(v_1,\b^3) \\
\Delta_{2.3^n(3m-1)}\alpha(\GG_2^1) & m\not\equiv 0\mod(3),\ n\geq 1 &
&(v_1^{4.3^{n+1}-1},\b v_1,\b^3) \\
\Delta_{2(3m-1)}\alpha(\GG_2^1) & m\not\equiv 0\mod(3)& 
&(v_1^{11},\b v_1,\b^4) \\
%\end{array}
%$$ 
%$$
%\begin{array}{llll} 
%&\\
\Delta_{6m}\widetilde{\a} %\ m\in\Z   
&&& (v_1,\b^5)\\
b_{2(9m+2)+1}
&&& (v_1^{2},\b) \\
\Delta_{6m+3}\widetilde{\a}  
& && (v_1^{3},\b v_1,\b^5) \\
%&\\
%\end{array}
%$$ 
%$$
%\begin{array}{llll} 
&\\
\Delta_{2.3^n(3m+1)}\alpha(\GG_2^1)\b  
& n\geq 1  && (v_1,\b^2) \\ 
\Delta_{2.3^n(3m-1)}\alpha(\GG_2^1)\b &m\equiv 0\mod(3),\ n\geq 1  
&& (v_1,\b^2) \\
\Delta_{2(3m-1)}\alpha(\GG_2^1)\b &m\equiv 0\mod(3)&& (v_1,\b^3) \\
\end{array}
$$ 
$$
\begin{array}{llll}
&\\
\overline{b}_{3^{n+1}(6m+1)}v_1  
& n\geq 0 & &(v_1^{6.3^n},\b) \\ 
\overline{b}_{3^{n}(6m+5)}v_1 &m\equiv 1\mod(3),\ n\neq 1          
&& (v_1^{10.3^n},\b) \\
\overline{b}_{3(6m+5)} &m\equiv 1\mod(3)              
&& (v_1^{31},\b v_1,\b^5) \\ 
%&\\ 
%\end{array}
%$$ 
%$$
%\begin{array}{llll}
\Sigma^{48}\overline{\Delta}_{3^n(6m+1)-3} 
& n\geq 1  &&(v_1^2,\b^2v_1,\b^5)\\
\Sigma^{48}\overline{\Delta}_{3^n(6m+5)-3} &m\not\equiv 1\mod(3),
\ n\geq 1 &&(v_1^2,\b^3v_1,\b^5)\\
\Sigma^{48}\overline{\Delta}_{3^n(6m+5)-3} &m \equiv 1\mod(3), 
\ n\geq 1 &&(v_1^2,\b^2v_1,\b^5)\\
%&\\
%\end{array}
%$$ 
%$$
%\begin{array}{llll}
\Sigma^{48}\overline{\Delta}_{(6m+1)-3}v_1 
&&&(v_1,\b^2)\\
\Sigma^{48}\overline{\Delta}_{3^n(6m\pm 1)-2} v_1
&n\geq 1    &&(v_1^{4.3^n-1},\b v_1,\b^3)\\
%\end{array}
%$$
%$$
%\begin{array}{ll}
%&\\
\Sigma^{48}\overline{\Delta}_{(6m+1)-2}v_1^2 
&   &&(v_1^2,\b)\\
\Sigma^{48}\overline{\Delta}_{(6m+5)-2} 
&    &&(v_1^{4},\b v_1^2,\b^3v_1,\b^5)\\
%\end{array}
%$$
%$$
%\begin{array}{ll}
&\\
\Sigma^{48}\overline{\Delta}_{2m-1}\alpha(\GG_2^1)& m\not\equiv 0\mod(3) &&(v_1,\b^3)\\  
\Sigma^{48}\overline{\Delta}_{3^n(6m+1)-3}\alpha(\GG_2^1)& n\geq 0       &&(v_1,\b^2) \\ 
\Sigma^{48}\overline{\Delta}_{3^n(6m+5)-3}\alpha(\GG_2^1)&m\not\equiv 1\mod(3),\ n\geq 1 
&&(v_1,\b^3) \\ 
\Sigma^{48}\overline{\Delta}_{3^n(6m+5)-3}\alpha(\GG_2^1)&m\equiv 1\mod(3),\ n\geq 1
&&(v_1,\b^2) \\ 
\Sigma^{48}\overline{\Delta}_{6m}\widetilde{\alpha}     
&&&(v_1,\b^4)\\
\Sigma^{48}\overline{\Delta}_{6m+3}\widetilde{\alpha}& m\not\equiv 1\mod (3)
&& (v_1,\b^4)\\ 
&\\
\end{array}
$$
modulo the following relations (in which module generators 
are put into brackets in order to distinguish between module 
multiplications and generators.)  
$$
\begin{array}{lllll}
\beta^3[\Delta_{k}\a(\GG_2^1)]&=
& \b^2\zeta[\Delta_{k}\b v_1]\  & k=2(3m-1)&\ m\not\equiv 0\mod(3) \\
\beta^2[\Delta_{k}\a(\GG_2^1)\b]&=
& \b^2\zeta[\Delta_{k}\b v_1]\    & k=2(3m-1)&\ m\equiv 0\mod(3)  \\
\b^4[{\Delta}_{k}\b]&=& \b^4\zeta[{\Delta}_{k}\widetilde{\a}] \ \ 
& k=6m+3 &\\
\b^2[\Sigma^{48}\overline{\Delta}_{k}\a(\GG_2^1)]&=
& \b^2\zeta[\Sigma^{48}\overline{\Delta}_{k}v_1] & k=6m+1   \\
\b^2[\Sigma^{48}\overline{\Delta}_{k}\a(\GG_2^1)]&=
& \b^2v_1\zeta[\Sigma^{48}\overline{\Delta}_{k}] & k=6m+3  \\
\b^2[\Sigma^{48}\overline{\Delta}_{k}\a(\GG_2^1)]&=
& \b^2v_1\zeta[\Sigma^{48}\overline{\Delta}_{k}] \ \ 
& k=3^n(6m+5)-3 & \ m\not\equiv 1\mod(3) ,\ n\geq 1 \\
\b^4[\Sigma^{48}\overline{\Delta}_{k}]&=
& \b^3\zeta[\Sigma^{48}\overline{\Delta}_{k}\widetilde{\a}]\ \ & k=6m &\ .\\
\end{array}
$$
\end{thm} 

We remark that some but not all of the relations figuring in this 
result could have been avoided by choosing different generators, e.g. 
if we had chosen, for $k=2(3m-1)$ and $m\equiv 0\mod (3)$, 
$\Delta_k\alpha(\GG_2)\b$ as a generator instead of $\Delta_k\alpha(\GG_2^1)\b$.   

Furthermore we remark that this description of the $E_{\infty}$-term as an 
$\F_3[\b,v_1]\otimes\Lambda(\zeta)$-module does not lift 
to $\pi_*(L_{K(2)}V(0))$. In fact, it is not hard to see that
there are exotic relations like $v_1\Delta\a=\b^2\widetilde{\a}$ 
which hold in $\pi_*(E_2^{hG_{24}}\wedge V(0))$, in particular  
$v_1[\Delta\a(\GG_2^1)]\neq 0$ in $\pi_*(L_{K(2)}V(0))$. 
  
As stated above this result is obtained without much trouble 
from the calculation of the Adams-Novikov $E_2$-term given in Theorem 
\ref{H*G21} and Theorem \ref{H*G2} by using knowledge of the Adams-Novikov 
differentials for $L_{K(2)}V(1)$. However, even if the rigorous 
proof proceeds this way, we feel that the final result 
can be better appreciated from the following point of view.  
In \cite{GHMR} the algebraic resolution 
(\ref{alg}) for $\GG_2^1$ resp. the companion resolution for $\GG_2$ 
(obtained by tensoring with a minimal resolution for $\Z_3$)  
was ``realized'' by resolutions of the homotopy fixed point spectrum 
$E_2^{h\G_2^1}$ resp. of $L_{K(2)}S^0$ via homotopy fixed point spectra 
with respect to the corresponding finite subgroups of $\G_2$. In particular 
there is a resolution    
\begin{eqnarray}\label{hom}
*\to L_{K(2)}S^0\to X_0\to X_1\to X_2\to X_3\to X_4\to *  
\end{eqnarray}  
with $X_0=E_2^{hG_{24}}$, $X_1=E_2^{hG_{24}}\vee \Sigma^8E_2^{hSD_{16}}$, 
$X_2=\Sigma^8E_2^{hSD_{16}}\vee \Sigma^{40}E_2^{hSD_{16}}$,  
$X_3=\Sigma^{48}E_2^{hG_{24}}\vee \Sigma^{40}E_2^{hSD_{16}}$ and 
$X_4=\Sigma^{48}E_2^{hG_{24}}$.  
We note that the $48$-fold suspension appearing in the definition of $X_3$ 
and $X_4$ is the reason for the (abusive) change of notation from $\Delta_k$ 
in Theorem \ref{H*G21} to $\Sigma^{48}\Delta_{k-2}$ in Theorem \ref{E-infty}. 
Furthermore, the spectrum $E_2^{hG_{24}}$ 
can be identified with the $K(2)$-localization of the spectrum $TMF$ 
of topological modular forms and $E_2^{hSD_{16}}$ with ``half" of the 
$K(2)$-localization of the spectrum $TMF_0(2)$ of topological modular 
forms of level $2$ (cf. \cite{behk}). These spectra are of 
considerable independent interest and their Adams-Novikov spectral 
sequences and their homotopy is well understood (cf. the appendix or 
\cite {behk}, \cite{GHMR}).  The Adams-Novikov differentials for 
$L_{K(2)}V(0)$ can be completely understood  
by those for $E_2^{hG_{24}}\wedge V(0)$ (cf. the remarks 
following Lemma \ref{d5} and Lemma \ref{d9} for more precise statements). 
The complicated final result  described in Theorem \ref{E-infty} 
can thus be deduced, just as in the case of Theorem \ref{H*G21}, 
from more basic structures by an essentially simple though 
elaborate mechanism.  

We believe that our results have the following advantages over those 
by Shimomura \cite{shi1}. In our approach the final result relates 
well to modular forms and the homotopy of the spectrum $TMF$ 
of topological modular forms; in particular the approach 
helps to understand how the complicated structure of $\pi_*(L_{K(2)}V(0))$ 
is built from that of the comparatively simple homotopy of $TMF$. 
This is also reflected in our notation which is very different from that 
in \cite{shi} where classical chromatic notation is used.   
Furthermore we determine $E_{\infty}$ as a module over $\F_3[\beta,v_1]$.  
In contrast Theorem 2.8. in \cite{shi1} gives a direct sum decomposition 
as an $\F_3[v_1]$-module (of $E_{\infty}$ and not as claimed in \cite{shi1} 
of $\pi_*(L_{K(2)}V(1))$) and this decomposition only partially reflects 
the $\F_3[\beta,v_1]$-module structure. In fact, many non-trivial 
$\b$-multiplications are not recorded in \cite{shi1}, for example those 
on the classes $\Delta_{2.3^n(9m+2)}\a(\GG_2^1)$, 
$\Delta_{2.3^n(9m+8)}\a(\GG_2^1)$, $\overline{b}_{3(18m+11)}$,   
... , nor are the additional $\b$-relations of Theorem \ref{E-infty}.  
There are related discrepancies on the height of $\b$-torsion; for example, 
in \cite{shi1} all elements in the same bidegree as 
the elements $\Sigma^{48}\overline{\Delta}_{6m}$ 
appear to be already killed by $\b^4$.   
Finally the classes $v_2^{9m+2}\xi/v_1^3$ of \cite{shi1} 
which correspond to $v_1\Sigma^{48}\overline{\Delta}_{(6m+1)-2}$ 
in our notation and which support a non-trivial Adams-Novikov 
$d_9$-differential (cf. Lemma \ref{d9}) seem to be permanent 
cycles in \cite{shi1}. 
%\marginpar{\bf $\b$-torsion!}

The paper is organized as follows. In section 
\ref{background} 
we recall 
background material on the stabilizer groups, we introduce important 
elements of $\GG_2$ and we recall the definition of its subgroups 
$SD_{16}$ and $G_{24}$ as well as that of an 
important torsionfree subgroup $K$ of finite index in $\GG_2^1$. 
In section 
\ref{resolution} 
we study the maps in the 
permutation resolution (\ref{alg}). 
In fact, in \cite{GHMR} the maps $C_3\to C_2$ and $C_2\to C_1$ of 
the permutation resolution (\ref{alg}) were not described explicitly 
so that the resolution was not ready yet to be used for 
detailed calculations. The subgroup $K$ plays a crucial role in 
finding an approximation to the map $C_2\to C_1$. 
We also show that the map $C_3\to C_2$ is 
in a suitable sense dual to the map $C_1\to C_0$. In section 
\ref{action} 
we study the action of the stabilizer group on $(E_2)_*/(3)$ and 
we derive formulae for the action of the elements of $\GG_2$ introduced 
in section 2. In section 
\ref{e2-term} we comment on Theorem \ref{shapiro} and we 
verify Theorem \ref{d1} (cf. Proposition \ref{d_1^{1,0}}, 
Proposition \ref{d_1^{2,0}} and Proposition \ref{d_1^{3,0}}).  
Most of the new results of these sections, 
in particular the formulae for the action of the stabilizer group, 
the approximation of the map $C_2\to C_1$ and the evaluation 
of the induced map 
$$
E_1^{1,0,*}=\Ext_{\Z_3[[\G_2^1]]}^0(C_1,(E_2)_*/(3))\to 
\Ext_{\Z_3[[\G_2^1]]}^0(C_2,(E_2)_*/(3))=E_1^{2,0,*}  
$$
are taken from the second author's thesis \cite{Kar}. The evaluation of 
this map is by far the hardest calculation in our approach. In section 
\ref{higher+extensions} 
we prove Proposition \ref{d2+d3} 
and Proposition \ref{extensions}.  In a short section 
\ref{passing} 
we discuss the subtleties of the K\"unneth isomorphism 
of Theorem \ref{H*G2} and section 
\ref{ANSS} 
contains the discussion of the differentials in the 
Adams-Novikov spectral sequence and proves Theorem \ref{E-infty}.  
For the convenience of the reader we have collected the description 
of the related Adams-Novikov spectral sequences for 
$E_2^{hG_{24}}\wedge V(0)$, for $E_2^{hG_{24}}\wedge V(1)$ and 
for $L_{K(2)}V(1)$ in an appendix.

\bigbreak

\section{Background on the Morava Stabilizer Group}\label{background}

\medbreak

In the sequel we will recall some of the basic
properties of the Morava stabilizer groups $\s_n$ resp. $\G_n$. The reader 
is refered to \cite{Rav} for more details (see also \cite{Henn1} and 
\cite{GHMR} 
for a summary of what will be important in this paper). 

\subsection{Generalities} 
We recall that the Morava stabilizer group $\s_n$ is the
group of automorphisms of the $p$-typical formal group law $\Gamma_n$
over the field $\F_{q}$ (with $q=p^n$) whose $[p]$-series is given by
$[p]_{\Gamma_n}(x)=x^{p^n}$. Because $\Gamma_n$ is already defined over
$\F_p$ the Galois group $\Gal(\F_q/\F_p)$ of the finite field
extension $\F_p\subset \F_q$ acts naturally on 
$\Aut(\Gamma_n)=\s_n$ and $\G_n$ can be identified 
with the semidirect product $\s_n\rtimes \Gal(\F_q/\F_p)$.

The group $\s_n$ is also equal to the group of units in the endomorphism
ring of $\Gamma_n$, and this endomorphism ring can be identified with
the maximal order $\Oa_n$ of the division algebra $\DD_n$ over $\QQ_p$
of dimension $n^2$ and Hasse invariant $\frac{1}{n}$. In more concrete
terms, $\Oa_n$ can be described as follows: let $\W:=\W_{\F_q}$ denote
the Witt vectors over $\F_{q}$. Then $\Oa_n$ is the non-commutative
ring extension of $\W$ generated by an element $S$ which
satisfies $S^n=p$ and $Sw=w^{\sigma}S$, where $w\in \W$ and
$w^{\sigma}$ is the image of $w$ with respect to the lift of the
Frobenius automorphism of $\F_q$. The element $S$ generates a two
sided maximal ideal ${\goth m}$ in $\Oa_n$ with quotient 
$\Oa_n/{\goth m}$ canonically isomorphic to $\F_q$. 
Inverting $p$ in $\Oa_n$ yields the division algebra $\DD_n$, 
and $\Oa_n$ is its maximal order.

Reduction modulo ${\goth m}$ induces an epimorphism 
$\Oa_n^{\times}\longrightarrow \F_q^{\times}$. Its kernel will be denoted by
$S_n$ and is also called the strict Morava stabilizer group. The group
$S_n$ is equipped with a canonical filtration by subgroups $F_iS_n$,
$i=\frac{k}{n}$, 
%$i=k/n$, 
$k=1,2,\ldots$, defined by
\[
F_i S_n:=\{g\in S_n| g\equiv 1 \ \text{mod}\ (S^{in})\} \ .
\]
The intersection of all these subgroups contains only the element
$1$ and $S_n$ is complete with respect to this filtration, i.e. we
have $S_n=\lim_i S_n/F_iS_n$. Furthermore, we have canonical
isomorphisms
\[
F_iS_n/F_{i+\frac{1}{n}}S_n\cong  \F_q 
\]
induced by 
\[
x=1+aS^{in}\mapsto \bar a \ .
\]
Here $a$ is an element in $\Oa_n$, i.e. $x\in F_iS_n$ and $\bar a$ is
the residue class of $a$ in $\Oa_n/{\goth m}\cong \FF_q$.

The associated graded object $gr S_n$ with 
$gr_i S_n: =F_iS_n/F_{i+\frac{1}{n}}S_n$, 
$i=\frac{1}{n}, \frac{2}{n},\ldots$ becomes a graded Lie algebra 
with Lie bracket $[\bar a, \bar b]$ induced by
the commutator $[x,y]:=xyx^{-1}y^{-1}$ in $S_n$. Furthermore, 
if we define a function $\vp$ from the positive real numbers to 
itself by $\vp (i): =\min\lbrace i+1,pi\rbrace$ then the $p$-th power map 
on $S_n$ induces maps $P: gr_iS_n\longra gr_{\vp (i)}S_n$ which define on 
$grS_n$ the structure of a mixed Lie algebra in the sense of Lazard 
(cf. Chap. II.1. of \cite{laz}). If we identify the filtration quotients 
with $\FF_q$ as above then the Lie bracket and the map $P$ are 
explicitly given as follows (cf. Lemma 3.1.4 in \cite{Henn1}).

\begin{lem}\label{lemcom} 
Let $\bar a\in gr_iS_n$, $\bar b\in gr_jS_n$. Then %\\

\noindent a)
$$
[\bar a, \bar b]=
\bar a {\bar b}^{p^{ni}} - \bar b {\bar a}^{p^{nj}}\in gr_{i+j}S_n 
$$

\noindent b)
$$
P\bar a=\begin{cases}
{\bar a}^{1+p^{ni}+\ldots +p^{(p-1)ni}}         &\  i<(p-1)^{-1} \\ 
{\bar a}+{\bar a}^{1+p^{ni}+\ldots +p^{(p-1)ni}}&\  i=(p-1)^{-1} \\
{\bar a}                                        &\  i>(p-1)^{-1} . 
\qed 
\end{cases}
$$
\end{lem}

\medbreak 
The right action of $\s_n$ on $\mathcal{O}_n$ determines a 
group homomorphism $\s_n \to GL_n{(\wv)}$. The resulting 
determinant homomorphism 
$\s_n \to \wv^\times$ extends to a homomorphism 
$$
\mathbb{G}_n \to \wv^\times \rtimes \Gal(\f/\mathbb{F}_p)
$$ 
which factors through $\mathbb{Z}_p^\times \times \Gal(\f/\mathbb{F}_p)$. 
By choosing a fixed isomorphism between the quotient of $\mathbb{Z}_p^\times$ 
by its maximal finite subgroup with $\mathbb{Z}_p$ we get the 
``reduced determinant" homomorphism 
$$
\mathbb{G}_n \to \mathbb{Z}_p \ .
$$ 
We denote its kernel by $\mathbb{G}_n^1$ and the intersection of 
$\G_n^1$ with $\mathbb{S}_n$ resp. $\mathbb{S}_n^1$ by  
$S_n$ resp. $S_n^1$. 
The center of $\mathbb{G}_n$ is equal to the center of $\mathbb{S}_n$ 
and can be identified with $\mathbb{Z}_p^\times$ (if we identify 
$\mathbb{S}_n$ with $\mathcal{O}_n^{\times}$) 
and the composite 
$$\mathbb{Z}_p^\times \to \mathbb{G}_n \to \mathbb{Z}_p^\times$$ sends 
$z$ to $z^n$. Thus if $p$ does not divide  $n$ we get an isomorphism 
$$
\mathbb{G}_n \cong \mathbb{Z}_p \times \mathbb{G}_n^1 \ .
$$

\subsection {Important subgroups in the case $n=2$ and $p=3$}
\label{subgroups} 
From now on we assume $n=2$ and $p=3$. 
Let $\omega$  be a primitive eighth root of unity in 
$\W^{\times}:=\W_{\F_9}^{\times}$. Then 
\begin{equation} 
a=-\frac{1}{2}(1+\omega S)
\end{equation}
is an element of $\s_2^1$ of order $3$. (This element was denoted $s$ in 
\cite{GHM} and \cite{GHMR}.) 
We can and will in the sequel choose $\omega$ such 
that we have the following relation in $\W/(3)\cong\F_9$
\begin{equation}\label{omegaeq}
\omega^2+\omega-1 \equiv 0\ .
\end{equation}

Next we let $t:=\omega^2$. Then we have $ta=a^2t.$
Furthermore, if $\phi\in \Gal(\F_9/\F_3)$ is the Frobenius automorphism then 
the elements $\psi:=\omega\phi$ and $t$ generate a subgroup $Q_8$ 
which normalizes the subgroup generated by $a$ so that $\psi$, $t$ and $a$ 
generate a subgroup $G_{24}$ of $\s_2^1$ of order 24 which is isomorphic to 
a semidirect product $\Z/3\rtimes Q_8$. 

The elements $\omega$ and $\phi$ generate a subgroup $SD_{16}$ of $\s_2^1$ of 
order 16, isomorphic to the semidihedral group of order $16$.

Finally there is a torsionfree subgroup $K$ of $S_2^1$ which has 
already played an important role in \cite{Henn1}.   
It is defined as follows: Lemma \ref{lemcom} implies that 
an element $1+xS$ in $S_2$ of order 3 satisfies 
$\overline{x} \not =0$ and $\overline{x}+\overline{x}^{1+3+9}=0$,  i.e. 
$\overline{x}^4=-1$ where $\overline{x}$ is the class of $x$ in 
$gr_{\frac{1}{2}}S_2^1 \cong \mathbb{F}_9$. 
There are no such elements $x$ such that $\overline{x}\in\mathbb{F}_3$. 
Hence, if we define $K$ to be the kernel of the homomorphism   
$$
S_2^1 \to gr_{\frac{1}{2}}S_2^1\cong \F_9 \to \F_9/\F_3
$$ 
then $K$ is torsion free, and we have a split short exact sequence 
\begin{equation}\label{ksz}
1 \to K \to S_2^1 \to \mathbb{Z}/3 \to 1 \ .
\end{equation}
$K$ inherits a complete filtration from $S_2$ by setting 
$F_{\frac{k}{2}}K=F_{\frac{k}{2}}S_2\cap K$ and it is easy to check that the associated 
graded is given as 
\begin{equation}\label{filtK}
gr_{{\frac{k}{2}}}K=\begin{cases}
\Ker(Tr:\F_9\to \F_3)  & k>0\ \textrm{even} \\
\F_3                   & k=1    \\
\F_9                   & k>1\ \textrm{odd} \\
\end{cases}
\end{equation}
where $Tr$ denotes the trace from $\F_9$ to $\F_3$. 
\medbreak
The following elements will play an important role in our 
later calculations. 
\begin{equation}
b:=[a,\omega], \ \ \ \ c:=[a,b], \ \ \ \ d:=[b,c] \ . 
\end{equation} 

In the next lemma we record approximations to these elements  
which we will use repeatedly.    
%\vfil\eject
 
\begin{lem}\label{abcd} \mbox{}
\begin{itemize}
\item[a)] $a\equiv 1+\omega S+S^2\ \ \ \ \  \mod\ (S^3)$
\item[b)] $b\equiv 1-S-\omega S^2\ \ \ \  \ \mod\ (S^3)$  
\item[c)] $c\equiv 1-\omega^2S^2-\omega S^3 \ \mod\ (S^4)$  
\item[d)] $d\equiv 1+\omega^2S^3 \ \ \ \ \ \ \ \ \ \mod\ (S^4)$
\end{itemize}
\end{lem}

\begin{proof} a) The approximation for $a$ is immediate from its 
definition.    
\smallbreak\noindent
b) By explicit calculation in $\Oa_2$ we find  
$$
b=\frac{1}{4}(1+\omega S)\omega(1-\omega S)\omega^{-1} 
=\frac{1}{4}(1+\omega S)(1-\omega^{-1}S)=
\frac{1}{4}(1+3\omega^2+(\omega+\omega^3)S)  
$$
and then we use that our choice of $\omega$ yields     
$\omega^3+\omega=-1$ and $\omega^2-1=-\omega$ in $\F_9$. 
\smallbreak\noindent
c) Similarly we get  
$$
c=\frac{1}{64}(1+\omega S)(1+3\omega^2+(\omega+\omega^3)S)
(1-\omega S)(1-3\omega^2-(\omega+\omega^3)S)
=-\frac{1}{8}(1+6\omega^2-3\omega S)\ .
$$
d) Finally the formula for $d$ can be obtained from Lemma \ref{lemcom}.   
\end{proof} 

The information in the following proposition will be important  
for a closer inspection of the permutation resolution (\ref{alg}).
 
\begin{prop}\label{propK} \mbox{}
\begin{itemize}
\item[a)] $H^*(K;\F_3)$ is a Poincar\'e duality algebra of dimension 3. 
\item[b)] $H_2(K;\F_3)\cong H_1(K;\F_3)\cong (\F_3)^2$.  
\item[c)] $H_1(K;\mathbb{Z}_3) \cong \Z/9 \oplus \Z/3$  
where $\Z/9$ resp. $\Z/3$ is generated by $b$ resp. by $c$. 
\item[d)] $H_2(K;\mathbb{Z}_3)= 0$.
\item[e)] $H_0(K;\mathbb{Z}_3) \cong H_3(K,\Z_3)\cong \mathbb{Z}_3$. 
\end{itemize}
\end{prop}

\begin{proof} Parts (a) and (b) have already been shown in 
Proposition 4.4. of \cite{Henn1}.  

For part (c) we note that  $H_1(K;\mathbb{Z}_3) \cong K/\overline{[K,K]}$ 
where $\overline{[K,K]}$ is the closure of the commutator subgroup 
$[K,K]$ of $K$. 
By Lemma \ref{lemcom} the commutator map
$$
gr_iK \times gr_jK \to gr_{i+j}K, \ \ (x,y) \mapsto [x,y]
$$
is surjective when  $i=\frac{3}{2}$ and $j=\frac{k}{2}$ with $k$ even, 
and also when $i=\frac{1}{2}$ and $j=\frac{l}{2}$ with $l>1$ odd. 
Thus $F_2K\subset [K,K]$. If $i=\frac{1}{2}$ and $j=1$ then the image of 
the commutator map is the kernel of the trace map. Together with part (b) of 
Lemma \ref{lemcom} this shows that $K/\overline{[K,K]}\cong \Z/9\oplus\Z/3$ 
and it is easy to check that $b$ and $c$ generate $\Z/9$ resp. $\Z/3$.  
 
The remaining parts (d) and (e) now follow from a simple Bockstein 
calculation.  
\end{proof}

%\vfil\eject

\bigbreak
\section{The maps in the permutation resolution}\label{resolution}  

\subsection {Generalities}\label{generalities} 
Let $G$ be a profinite $p$-group. We say that $G$ is finitely generated 
if $H_1(G,\Z_p)$ is finitely generated over $\Z_p$. The kernel 
of the augmentation $\Z_p[[G]]\to \Z_p$ is denoted by $IG$, 
or simply by $I$. We say that a $\Z_p[[G]]$-module $M$ is 
$I$-complete if the filtration by the submodules $I^nM$, $n\geq 0$, 
is complete. As in \cite{GHMR} we use a Nakayama type lemma 
to show that certain homomorphisms are surjective. Its proof is the same 
as that of Lemma 4.3 of \cite{GHMR}. 

\begin{lem}\label{nak} Let $G$ be a finitely generated 
profinite $p$-group and $f:M\to N$ a homomorphism of $IG$-complete 
$\Z_p[[G]]$-modules. 
Suppose that $H_0(f):H_0(G,M)\to H_0(G,N)$ is surjective. 
Then $f$ is surjective. \qed 
\end{lem}

In \cite{GHMR} we used the analogous Lemma 4.3 for $G=S_2^1$ in order 
to show that certain $\Z_3[[S_2^1]]$-linear maps are surjective. 
Here we use Lemma \ref{nak} for $G=K$ together with the action 
of the element $a$ on $H_0(K,-)$ resulting from the exact sequence 
(\ref{ksz}) in order to show that the same maps are surjective. 
The advantage of working with $K$ will become clear when we will 
discuss the kernel of the map $C_1\to C_0$ 
(see the remark after Proposition \ref{propN2} below). We begin the 
construction of the permutation resolution exactly as in \cite{GHMR}.  

\subsection {The homomorphism $\partial_1$}\label{partial1} 
 
Let 
$C_0=\mathbb{Z}_3[[\mathbb{G}_2^1]]
\otimes_{\mathbb{Z}_3[G_{24}]}\mathbb{Z}_3$ 
and $e_0=e\otimes 1\in C_0$ if $e$ is the unit in $\mathbb{G}_2^1$. 
Let $\partial_0:C_0 \to \mathbb{Z}_3$ be the standard augmentation  
and let $N_0$ be the kernel of $\partial_0$ so that we have a 
short exact sequence 
\begin{equation}\label{sec1}  
0 \to N_0 \to C_0 \to \mathbb{Z}_3 \to 0\ .
\end{equation} 

\begin{prop}\label{propN1} \mbox{ }
\begin{itemize}
\item [a)] As $\Z_3[[K]]$-module $N_0$ is generated by the 
elements $f_1:=(e-\omega)e_0$, $f_2:=(e-b)e_0$ and $f_3:=(e-c)e_0$. 
If we denote the class of $f_i$ in $H_0(K,N_0)$ 
by $\overline{f_i}$ then we have an isomorphism 
$$
H_0(K;N_0) \cong \Z_3\{\overline{f_1}\} 
\oplus \Z/9\{\overline{f_2}\} \oplus \Z/3\{\overline{f_3}\} \ .
$$ 
\item [b)] The action of $a$ on $H_0(K,N_0)$ is given by :
$$
\begin{array}{ccccc}
a_*\overline{f_1}=\overline{f_1}+\overline{f_2}, && 
a_*\overline{f_2}=\overline{f_2}+\overline{f_3}, &&
a_*\overline{f_3}=\overline{f_3}-3\overline{f_2}\ .
\end{array}
$$
\item[c)] $H_1(K;N_0)=0$. 
\end{itemize}
\end{prop}
%\vfil\eject

\begin{proof}  a) We consider the long exact sequence in $H_*(K,-)$ 
associated to the short exact sequence (\ref{sec1}). 
As $C_0$ is free $K$-module we have $H_1(K;C_0)=0$. Furthermore,  
$H_0(K;C_0)\cong \Z_3^2$, generated by the classes of $e_0$ and $\omega e_0$ 
so that the end of this long exact sequence has the following form   
$$
0 \to H_1(K;\Z_3)\cong\Z/9\oplus\Z/3\to H_0(K;N_0) \to H_0(K;C_0)\cong\Z_3^2 
\to H_0(K;\Z_3)\cong \Z_3\to 0 \ .
$$
Now (a) follows from Proposition \ref{propK} and the identification of 
$H_1(K,\Z_3)$ with $IK/(IK)^2$ which sends $x\in K$ to the class of 
$x-e$ in $IK/(IK)^2$.   

b) By definition we have $ae_0=e_0$ and thus 
$$
a\omega e_0=a\omega a^{-1}\omega^{-1}\omega ae_0=
[a,\omega]\omega e_0=b\omega e_0 \ . 
$$
Consequently 
\begin{equation}\label{a-action1}
a(e-\omega)e_0=(e-a\omega)e_0=(e-b\omega)e_0= 
b(e-\omega)e_0+(e-b)e_0
\end{equation} 
and we obtain the first formula by passing to $K$-coinvariants. 

Similarly, $ab=cba$ and $ae_0=e_0$ imply  
\begin{equation}\label{a-action2} 
a(e-b)e_0=(e-ab)e_0=(e-cba)e_0 =(e-cb)e_0=c(e-b)e_0+(e-c)e_0
\end{equation} 
and by passing to $K$-coinvariants we get the second formula.

The third formula can now be deduced from the fact that 
${a}_*^3\overline{f}_1=\overline{f}_1$.

c) This follows from the long exact sequence in $H_0(K,-)$ associated to 
the exact sequence (\ref{sec1}) by using that $H_i(K,C_0)=0$ for $i=1,2$ 
and $H_2(K,\Z_3)=0$.   
\end{proof}

Let $C_1=\mathbb{Z}_3[[\mathbb{G}_2^1]]\otimes_{\mathbb{Z}_3[SD_{16}]}\chi$ 
where $\chi$ is the non trivial character of $SD_{16}$ defined over 
$\mathbb{Z}_3$ on which $\omega$ and $\phi$ both act by multiplication 
by $-1$.  
Let $e_1$ be the generator of $C_1$ given by $e\otimes 1$ where $e$ 
is as before the unit in $\G_2^1$. 

\begin{cor}\label{delta1} 
There is a $\Z_3[[\G_2^1]]$-linear epimorphism 
$\partial_1:C_1 \to N_0$ given by $e_1 \mapsto (e-\omega)e_0$.
\end{cor} 

\begin{proof} The elements $\omega^2$, $\phi\omega$ and $\omega^{-1}\phi$ 
all belong to $G_{24}$ and hence they act trivially on $e_0$. 
Therefore we have  
$$
\omega(e-\omega)e_0=(\omega-\omega^2)e_0=-(e-\omega)e_0
$$ 
and 
$$
\phi(e-\omega)e_0=(\phi-\phi\omega)e_0=
(\omega(\omega^{-1}\phi)-e)e_0=-(e-\omega)e_0 \ .
$$  
This implies that there is a well defined homomorphism $C_1\to N_0$ 
which sends $e_1$ to $(e-\omega)e_0$. 
To see that this homomorphism is surjective we note that 
$C_1$ is free as $\Z_3[[K]]$-module of rank $3$ with generators 
$e_1$, $ae_1$ and $a^2e_1$. Then we use 
Lemma \ref{nak} and Proposition \ref{propN1}.  
\end{proof}
 
\subsection {The homomorphism $\partial_2$}\label{partial2} 
Now we turn towards the construction of the second homomorphism 
in our permutation resolution. This is substantially more intricate; 
in \cite{GHMR} its existence was established but no explicit formula was given. 

Let $N_1$ be the kernel of $\partial_1$ so that we have a 
short exact sequence 
\begin{equation}\label{sec2}  
0 \to N_1 \to C_1 \to N_0 \to 0\ . 
\end{equation}

\begin{prop}\label{propN2} \mbox{ }
\begin{itemize}
\item[a)] $H_0(K;N_1)\cong \Z_3^2$. The inclusion of $N_1$ into $C_1$ 
induces an injection $H_0(K,N_1)\to H_0(K,C_1)$ and identifies 
$H_0(K,N_1)$ with the submodule generated by the classes 
$\overline{g}_i$, $i=1,2$, of  $g_1=3(a-e)^2e_1$ and  $g_2=9(a-e)e_1$. 
\item[b)] The action of $a$ on $H_0(K,N_1)$ is determined by 
$$
\begin{array}{rcl}a_*\overline{g}_1=-2\overline{g}_1-\overline{g}_2,
&&a_*\overline{g}_2=3\overline{g}_1+\overline{g}_2\ . 
\end{array}
$$
\item[c)] $H_0(S_2^1,N_1)\cong \Z/3$ and if 
$n_1$ is any element of $N_1$ which agrees in 
$H_0(K,C_1)$ with $\overline{g}_1$ then its class in  
$H_0(S_2^1,N_1)$ is non-trivial. 
\item[d)] The elements $\omega$ and $\phi$ both act on $H_0(S_2^1,N_1)$ 
by multiplication by $-1$. 
\end{itemize}
\end{prop} 

\begin{proof} a) We observe that $C_1$ is a free $K$-module of 
rank $3$ generated by $e_1$, $ae_1$ and $ae_2$. Then (a) follows 
from Proposition \ref{propN1} by using the long exact sequence in $H_0(K,-)$ 
associated to the short exact sequence (\ref{sec2}) . 

b) The formulae for the action of $a$ already hold in $H_0(K,C_1)$. 
  
c) This follows from part (b) by passing to coinvariants with respect to 
the action of $a$. 

d) This has already been observed in Lemma 4.6. of \cite{GHMR}. 
\end{proof}

\remarkstyle{Remark} We remark that working with $S_2^1$-coinvariants only 
(as in \cite{GHMR}) does not give us a good hold on a generator 
of $N_1$. The reason is that the map 
$H_0(S_2^1,N_1)\cong \Z/3\to H_0(S_2^1,C_1)\cong \Z_3$ 
is necessarily trivial and therefore such a generator 
cannot be easily associated with an element in $C_1$. 
Working with $K$-coinvariants gives us a starting point, 
namely the element $g_1\in C_1$, from which we can try to construct 
an element $n_1$ of $N_1$ whose class in $H_0(K,C_1)$ agrees 
with that of $\overline{g}_1$ and thus projects to a 
generator of $H_0(S_2^1,N_1)$. A first step in the direction of 
finding such a generator $n_1$ is taken in Lemma \ref{lemn2} below.  

\begin{cor}\label{corn2} Let $C_2=\Z_3[[\GG_2^1]]\otimes_{\Z_3[SD_{16}]}\chi$,  
let $n_1\in N_1$ be any element which projects non-trivially to  
the coinvariants $H_0(S_2^1;N_1)$ and let 
$$
n_1':=\frac{1}{16}\sum_{g\in SD_{16}}\chi(g^{-1})g(n_1)\ .
$$
Then there is a $\Z_3[[\G_2^1]]$-linear epimorphism 
$\partial_2:C_2 \to N_2$ given by $e_2 \mapsto n_1'$.  
\end{cor} 

\begin{proof} By construction the group $SD_{16}$ acts on $n_1'$ 
via the mod-$3$ reduction of the character $\chi$ 
and thus there is a homomorphism $\partial_2$ as claimed. 
Surjectivity of $\partial_2$ follows from Lemma \ref{nak}. 
\end{proof}

\begin{lem}\label{lemn2} \mbox{ }
\begin{itemize} 
\item[a)] 
Let 
$l_1=(a-b)e_1$, $l_2=(a-c)l_1$ and $l_3=3cl_2+(e-c)^2l_2$.  
Then 
$$
\partial_1(l_1)=(e-b)e_0, \ \ \partial_1(l_2)=(e-c)e_0 \ ,
\ \ \partial_1(l_3)=(e-c^3)e_0 \ .
$$
\item[b)] There exist elements $x,y\in IK$ such that 
$e-c^3=x(e-b)+y(e-c)$. 
\item[c)] If $x,y\in IK$ satisfy $e-c^3=x(e-b)+y(e-c)$ then   
$$
n_1:=l_3-xl_1-yl_2
$$ 
belongs to $N_1$ and projects non-trivially to $H_0(S_2^1,N_1)$.     
\end{itemize}  
\end{lem}

\begin{proof} We start with the following two observations: 
\begin{itemize} 
\item  Proposition \ref{propN1} implies that 
$\partial_1((a-e)^2e_1)\equiv (e-c)e_0\ \mod (IK)N_0$.  
\item In $\Z_3[[K]]$ we have the relation $3c(e-c)+(e-c)^3=e-c^3$. 
\end{itemize} 

a) By equations (\ref{a-action1}) and (\ref{a-action2}) 
of the proof of Proposition \ref{propN1} we see that 
$$
\partial_1(l_1)=(e-b)e_0 \ \ \textrm{and}\ \ \partial_1(l_2)=(e-c)e_0  \ . 
$$  
The result for $\partial_1(l_3)$ is now obvious. 
\smallbreak
b) By Proposition \ref{propK}.c we know that $IK$ is generated by 
$e-b$ and $e-c$. 
Furthermore $c^3$ belongs to $F_2K$, hence it is trivial in 
$H_1(K,\Z_3)$ by Proposition \ref{propK}.c\ . Therefore 
$e-c^3$ belongs to $IK^2$  
and we get the existence of $x,y\in IK$ as required in (b). 
\smallbreak
c) By (a) and (b) $n_1$ belongs to $N_1$. Furthermore it is clear that 
$n_1$ and $3(a-e)^2e_1$ agree in $H_0(K,C_1)$, hence 
$n_1$ projects non-trivially to $H_0(S_2^1,N_1)$.  
\end{proof} 

The question becomes now how we can determine $x$ and $y$. In fact, 
we do not have explicit formulae for $x$ and $y$. However, 
in section \ref{middle-differential} we will give approximations 
for them which are sufficient for our homological calculations.

\subsection{The homomorphism $\partial_3$}\label{partial3} 
In \cite{GHMR} it was shown (by using that $K$ is a 
Poincar\'e duality group) that the kernel of $\partial_2$ 
can be identified with $\Z_3[[\G_2^1/G_{24}]]$. 
However, the identification and thus the 
construction of $\partial_3$ was not explicit. The following result 
shows that $\partial_3$ can be replaced by the dual of $\partial_1$, 
at least up to isomorphism.   
\smallbreak
If $G$ is a profinite group and $M$ a continuous left $\Z_p[[G]]$-module 
then we define its dual $M^*$ by $\Hom_{\Z_p[[G]]}(M,\Z_p[[G]])$. 
This becomes a left $\Z_p[[G]]$-module via 
$(g.\varphi)(m)=\varphi(m)g^{-1}$ if 
$g\in G$, $\varphi\in \Hom_{\Z_p[[G]]}(M,\Z_p[[G]])$ and $m\in M$. 
We observe that for a finite subgroup $H$ there is a 
canonical $\Z_p[[G]]$-linear isomorphism 
\begin{equation}\label{dual} 
\Z_p[[G/H]]\to \Z_p[[G/H]]^*\cong \Z_p[[G]]^H, \ \ 
g\mapsto \Big(g^*:\widetilde{g}\mapsto 
\widetilde{g}(\sum_{h\in H}h)g^{-1}\Big)
\leftrightarrow(\sum_{h\in H}h)g^{-1} \ . 
\end{equation} 
%\vfil\eject

\begin{prop}\label{dualcomplex}
\mbox{ }
\begin{itemize}
\item[a)] There is an exact complex of $\Z_3[[\GG_2^1]]$-modules 
$$
0 \longrightarrow C_0^* 
\stackrel{{\partial}_1^*}\longrightarrow C_1^*  
\stackrel{{\partial}_2^*}\longrightarrow C_2^*  
\stackrel{{\partial}_3^*}\longrightarrow C_3^*  
\stackrel{\overline{\epsilon}}\longrightarrow  \Z_3 \longrightarrow 0
$$
in which ${\partial}_i^*$ is the dual of $\partial_i$ for $i=1,2,3$.  
\item[b)] There is an isomorphism of complexes of 
$\Z_3[[\GG_2^1]]$-modules 
$$
\begin{array}{cccccccccccccc} 
0 &\longrightarrow & C_0^*& 
\stackrel{{\partial}_1^*}\longrightarrow &C_1^* &
\stackrel{{\partial}_2^*}\longrightarrow &C_2^* & 
\stackrel{{\partial}_3^*}\longrightarrow &C_3^* & 
\stackrel{\overline{\epsilon}}\longrightarrow &  \Z_3 
&\longrightarrow &0& \\
& &\downarrow f_3&&\downarrow f_2&&\downarrow f_1&&\downarrow f_0
&&\downarrow = &&\\
0 &\longrightarrow &C_3& 
\stackrel{{\partial}_3}\longrightarrow &C_2 &
\stackrel{{\partial}_2}\longrightarrow &C_1 & 
\stackrel{{\partial}_1}\longrightarrow &C_0 & 
\stackrel{\epsilon}\longrightarrow &  \Z_3 &\longrightarrow &0& \ .\\
\end{array}
$$
such that $f_i$ induces the identity on 
$\Tor_0^{\Z_3[[S_2^1]]}(\F_3,C_i)$ for $i=2,3$ 
if we identify $C_i^{*}$ with 
$C_{3-i}$ via the isomorphism of (\ref{dual}). 

\item[c)] The homomorphism $\partial_1^*:C_0^* \to C_1^*$ 
is given by \ 
$
e_0^* \mapsto (e+a+a^2)e_1^* %\ .
$.
\end{itemize}
\end{prop}

\begin{proof} a) Each $C_i$ is free as a $\Z_3[[K]]$-module 
and therefore the complex 
$$
0 \longrightarrow C_3\stackrel{{\partial}_3}\longrightarrow C_2 
\stackrel{{\partial}_2}\longrightarrow C_1  
\stackrel{{\partial}_1}\longrightarrow C_0  
\stackrel{\epsilon}\longrightarrow \Z_3 \longrightarrow 0
$$
is a free $\Z_3[[K]]$-resolution of $\Z_3$. Because $K$ is of 
finite index in $\GG_2^1$ the coinduced module of $\Z_3[[K]]$ 
is isomorphic to $\Z_3[[\GG_2^1]]$.  
Therefore there are natural isomorphisms 
$$
C_i^*=\Hom_{\Z_3[[\G_2^1]]}(C_i,\Z_3[[\G_2^1]])\cong 
\Hom_{\Z_3[[K]]}(C_i,\Z_3[[K]])
$$
and the $n$-th cohomology of the complex 
$\Hom_{\Z_3[[K]]}(C_i;\mathbb{Z}_3[[K]])$ 
is $H^n(K;\mathbb{Z}_3[[K]])$. Because $K$ is a 
Poincar\'e duality group this is zero except when $n=3$ 
and then it is isomorphic to $\Z_3$. Finally, one sees as in 
Proposition 5 of \cite{S-GH}  that the $\Z_3[[\GG_2^1]]$-module structure on 
$H^n(\GG_2^1;\mathbb{Z}_3[[\GG_2^1]])\cong H^n(K;\mathbb{Z}_3[[K]])$ 
is trivial.    

b) The augmentation $\Z_3[[\GG_2^1]]\to \Z_3$ induces an isomorphism 
$$
\Hom_{\Z_3[[\G_2^1]]}(\Z_3,\Z_3)\cong 
\Hom_{\Z_3[[\G_2^1]]}(\Z_3[[\G_2^1]],\Z_3)\ .
$$ 
Thus the right hand square is commutative up to a unit in $\Z_3$ 
if we choose for $f_0$ the isomorphism given in (\ref{dual}), 
and we can modify $f_0$ by a unit so that it commutes on the nose.    
Then $f_0$ induces an isomorphism $\Ker\epsilon\cong \Ker\overline{\epsilon}$  
and    
$$
C_2 \stackrel{{\partial}_2}\longrightarrow C_1  
\stackrel{{\partial}_1}\longrightarrow \Ker\epsilon\ \ \ \text{resp.}\ \ \  
C_1^*\stackrel{{\partial}_2^*}\longrightarrow C_2^*  
\stackrel{{\partial}_3^*}\longrightarrow \Ker\overline{\epsilon}
$$  
is the beginning of a resolution of $\Ker\epsilon$ resp. 
$\Ker\overline{\epsilon}$ by projective $\Z_3[[\GG_2^1]]$-modules 
and the isomorphism induced by $f_0$ lifts to 
a chain map $f_{\bullet}$ between the projective resolutions. 
By Lemma 4.5 of \cite{GHMR} we have 
$\Tor_i^{\Z_3[[S_2^1]]}(\F_3,\Ker\epsilon)\cong \F_3$ if $i=0,1$ 
and this implies that the maps $f_1:C_1\to C_2^*$ and $f_2:C_2\to C_1^*$ 
induce isomorphisms on $\Tor_0^{\Z_3[[S_2^1]]}(\F_3,-)$ and hence 
they are themselves isomorphisms by Lemma 4.3 of \cite{GHMR}. Finally, 
$f_3$ is trivially an isomorphism because $f_2$ and $f_1$ are isomorphisms.

The isomorphism of chain complexes (considered as automomorphism 
via (\ref{dual})) induces an automorphism of spectral sequences 
$$
\Tor_j^{\Z_3[[S_2^1]]}(\F_3,C_i)\Rightarrow \Tor_{i+j}^{\Z_3[[S_2^1]]}(\F_3,\Z_3)
$$ 
which converges towards the identity, and this easily implies 
the remaining part of (b).

c) By (\ref{dual}) we have for each $\widetilde{g}\in \GG_2^1$
$$
(e+a+a^2)^*(e_1^*)(\widetilde{g})=
\widetilde{g}\big(\sum_{h\in SD_{16}}\chi(h^{-1})h\big)(e+a^{-1}+a^{-2})
=\widetilde{g}\big(\sum_{h\in SD_{16}}\chi(h^{-1})h\big)(e+a+a^2)
$$ 
and 
$$
\partial_1^*(e_0^*)(\widetilde{g})=e_0^*(\widetilde{g}(e-\omega))=
\widetilde{g}(e-\omega)\big(\sum_{h\in G_{24}}h\big)=
\widetilde{g}(e-\omega)\big(\sum_{h\in Q_8}h\big)(e+a+a^2) \ . 
$$
Then we conclude via the identity 
$(e-\omega)\big(\sum_{h\in Q_8}h\big)=\sum_{h\in SD_{16}}\chi(h^{-1})h$. 
\end{proof}

\bigbreak 

\section{On the action of the stabilizer group}\label{action}

In this section we will produce formulae for the action of the elements 
$a$, $b$, $c$ and $d$ of the stabilizer group $\mathbb{G}_2$ on 
$\mathbb{F}_9[[u_1]][u^{\pm 1}]$, at least modulo suitable powers of the 
invariant ideal generated by $u_1$. It turns out that it is sufficient 
to have a formula for the action of $a$ and $b$ on $u$ modulo $(u_1^6)$, 
and for the action of $c$ and $d$ on $u$ modulo $(u_1^{10})$.   
\smallbreak

\subsection{Generalities}  
We recall (cf.  \cite{Rav}) that $BP_*\cong \Z_{(p)}[v_1,v_2,\ldots]$ 
where the Araki generators $v_i$ satisfy the following equation 
(in $BP_*\otimes \mathbb{Q}$)    
\begin{equation}\label{haz}
p\lambda_k=\sum_{0\leq i \leq  k}\lambda_iv_{k-i}^{p^i}\ .
\end{equation} 
Here the $\lambda_i\in BP_*\otimes \QQ$ are the coefficients 
of the logarithm of the universal $p$-typical formal group law $F$ 
on $BP_*$,  
$$
\log_{F}(x)=\sum_{i\geq 0}\lambda_ix^{p^i} \, 
$$ 
(with $\lambda_0=1$), and thus the $[p]$-series of $F$ is given by 
$$
[p]_{F}(x)={\sum_{i\geq 0}}^Fv_ix^{p^i}\ .
$$
The homomorphism 
$$
BP_*\to  (E_n)_*=\W_{\F_{p^n}}[[u_1,\ldots,u_{n-1}]][u^{\pm 1}],
\ \ v_i\mapsto 
\left\{\begin{array}{lc}
u_iu^{1-p^i}&i<n \\ 
u^{1-p^n}   &i=n \\ 
0           &i>n 
\end{array}\right. 
$$ 
defines a $p$-typical formal group law $F_n$ over $(E_n)_*$.  
Then the formal group law $G_n$ over $({E_n})_0$ defined by 
$G_n(x,y)=u^{-1}F_n(ux,uy)$ is a universal deformation of $\Gamma_n$ and 
is $p$-typical with $p$-series    
\begin{equation}\label{p-series}
[p]_{G_n}(x)=px+_{G_n}u_1x^p+_{G_n}\cdots
+_{G_n}u_{n-1}x^{p^{n-1}}+_{G_n}x^{p^n} \ . 
\end{equation}

Next we recall how one can get at the action of an element 
$g\in \mathbb{S}_n$ on $(E_n)_0$.  
For a given $g$ we choose a lift 
$\widetilde{g}\in (E_n)_0[[x]]$ of $g$ and let $\widetilde{G}$ 
be the formal group law defined by 
$$
\widetilde{G}(x,y)=
\widetilde{g}^{-1}\big(G_n(\widetilde{g}(x),\widetilde{g}(y))\big) 
\textrm{ .} 
$$
Then there is a unique ring homomorphism $g_*:(E_n)_0\to (E_n)_0$ and a 
unique $*$-isomorphism from $g_*G_n$ to $\widetilde{G}$ such that 
the composition  
$$
h_g:g_*G_n \to \widetilde{G} \stackrel{\widetilde{g}}\to G_n 
$$ 
is an isomorphism of $p$-typical formal group laws 
and can therefore be written (cf. Appendix 2 of  \cite{Rav}) as 
$$
h_g(x)=\sum_{i\geq 0}\phantom{}^{G_n}t_i(g)x^{p^i}
$$
for unique continuous functions  
\begin{equation}\label{ti}
t_i:\mathbb{S}_n\to(E_n)_0 \textrm{ .}
\end{equation}
We note that 
\begin{equation}\label{t_imodm}
t_i(g)\equiv g_i \ \text{mod}\ (3,u_1)
\end{equation}
if $g=\sum_ig_iS^i\in \s_n$ with $g_i^{p^2}=g_i$. 
Then we have 
\begin{equation}\label{u-action} 
g_*(u)=t_0(g)u
\end{equation}  
and the equation   
\begin{equation}\label{eqnfgl} 
h_g([p]_{g_*{G_n}}(x))=[p]_{G_n}(h_g(x)) 
\end{equation} 
can be used to recursively find better and better approximations 
for $t_0(g)$ as well as for the action of $g$ on the deformation 
parameters $u_1,\ldots,u_{n-1}$.    

\medbreak

\subsection{A formula modulo $(3)$ for the formal group law $G_2$ } 
From now on we restrict attention to the case $n=2$ and $p=3$.

\begin{lem} The logarithm and exponential of the formal group law 
$G_2$ satisfy
$$
\begin{array}{ccll}
\log_{G_2}(x)&=&x-\frac{u_1}{24}x^3+
\frac{1}{1-3^8}\big(\frac{1}{3}-\frac{u_1^4}{72}\big)x^9 
\ &\mod\ (x^{27}) \\
& & \\
\exp_{G_2}(x)&=&x+\frac{u_1}{24}x^3+\frac{3u_1^2}{24^2}x^5+
\frac{12u_1^3}{24^3}x^7-
\Big(\frac{1}{1-3^8}\big(\frac{1}{3}-\frac{u_1^4}{72}\big)+
\frac{55u_1^4}{24^4}\Big)x^9\ &\mod\ (x^{11})\ . \\ 
\end{array}
$$
\end{lem}

\begin{proof} From (\ref{haz}) we get 
$\lambda_1=-\frac{v_1}{24}$ and  
$\lambda_2=\frac{1}{1-3^8}\big(\frac{v_2}{3}-\frac{v_1^{4}}{72}\big)$.   
To obtain the result for $\log_{G_2}$ we use the classifying homomorphism 
for $F_2$ and that $\log_{G_2}(x)=u^{-1}\log_{F_2}(ux)$.  
\end{proof} 

\begin{cor}\label{fgl} The formal group law $G_2$ satisfies
$$
\begin{array}{rcll}x+_{G_2}y
&\equiv & x+y-u_1(xy^2+x^2y)+u_1^2(xy^4+x^4y)\\
& &-u_1^3(xy^6+x^6y)-u_1^3(x^3y^4+x^4y^3)&\\
& &-(x^3y^6+x^6y^3)+u_1^4(x^4y^5+x^5y^4)&\mod (3,(x,y)^{11}) \ .
\end{array}
$$
\end{cor} 

\begin{proof} This follows directly from 
$x+_{G_2}y=\exp_{G_2}(\log_{G_2}(x)+\log_{G_2}(y))$.  
\end{proof} 

\medbreak 

\subsection{Formulae for the action modulo $(3)$} 
To simplify notation we will denote in the remainder of this section 
the mod-$3$ reduction of the value of the function $t_i$ of (\ref{ti}) 
on an element $g$ again by $t_i(g)$, or even by $t_i$ 
if $g$ is clear from the context.  
 
\begin{prop}\label{t_i} Let $g\in \s_n$ and let $u_1$ be Araki's $u_1$.  
Then the following equations hold  
\begin{itemize} 
\item[a)] $g_*(u_1)=t_0^2u_1$ 
\item[b)] $t_0+t_0^6t_1u_1^3=t_0^9+t_1^3u_1$
\item[c)] $t_1-t_0^8t_1u_1^4\equiv 
t_1^9+t_2^3u_1-t_0^{18}t_1^3u_1^2-t_0^9t_1^6u_1^3
\ \ \ \ \ \ \mod\  (u_1^7)$
\item[d)] If \ $g\equiv 1 \mod (S^2)$\ \ then  \ \
$t_2\equiv t_2^9+t_3^3u_1 \ \ \ \ \  \mod\ (u_1^2)$.  
\end{itemize}
\end{prop} 
 
\remarkstyle{Remark}  
If we want to know $g_*(u_1)$ modulo $(u_1^7)$ 
then (a) shows that it is enough to know $t_0$ modulo $(u_1^6)$ and 
this can be calculated from (b) if we know $t_0$ modulo 
$(u_1)$ and $t_1$ modulo $(u_1^3)$. Furthermore, $t_1$ 
can be calculated modulo $(u_1^3)$ from (c)  
if we know $t_0$, $t_1$ and $t_2$ modulo $(u_1)$. 
In the same manner we can even calculate $g_*(u_1)$ modulo $(u_1^8)$.   

Similarly, if we want to know $g_*(u_1)$ modulo $(u_1^{11})$ 
then (a) shows that it is enough to know $t_0$ modulo $(u_1^{10})$ and 
this can be calculated from (b) if we know $t_0$ modulo $(u_1)$  
and $t_1$ modulo $(u_1^7)$. Furthermore, $t_1$ can be calculated modulo 
$(u_1^7)$ from (c) 
if we know $t_0$ modulo $(u_1^3)$, $t_1$ modulo $(u_1)$ and 
$t_2$ modulo $(u_1^2)$. Finally (d) can be used to calculate 
$t_2$ modulo $(u_1^2)$ if we know $t_2$ and $t_3$ modulo $(u_1)$. 
 
\begin{proof} In this proof we abbreviate $G_2$ simply by $G$. 
We consider equation (\ref{eqnfgl})
$$
h_g([3]_{g_*G})(x)=[3]_G(h_g(x))
$$ 
over $(E_2)_0/(3)[[x]]$ and compare coefficients of 
$x^{3^k}$ for $k=1,2,3,4$.   
\smallskip
By (\ref{p-series}) we have 
$$
\begin{array}{llll}
h_g([3]_{g_*G})(x)&\equiv & \ t_0\big(g_*(u_1)x^3+_{g_*G}x^9\big)
+_G t_1\big(g_*(u_1)x^3+_{g_*G}x^9\big)^3 & \\ 
&&\ +_G t_2\big(g_*(u_1)x^3+_{g_*G}x^9\big)^9 
+_G t_3\big(g_*(u_1)x^3+_{g_*G}x^9\big)^{27}& \text{mod}\ (x^{82}) \\ 
&&&\\
{[3]_G}(h_g(x))&\equiv & \ 
u_1(t_0x+_Gt_1x^3+_Gt_2x^9+t_3x^{27})^3 \\ 
&&+_G(t_0x+_Gt_1x^3+_Gt_2x^9)^9   & \text{mod}\ (x^{82}) \ .\\ 
\end{array} 
$$

a) For the coefficient of $x^3$ we obviously get $g_*(u_1)t_0=u_1t_0^3$ 
which proves (a). 

b) The coefficient of $x^9$ in $h_g([3]_{g_*G})(x)$ 
is equal to $t_0+g_*(u_1)^3t_1$ which by (a) is equal to 
$t_0+u_1^3t_0^6t_1$. The coefficient of $x^9$ in ${[3]_G}(h_g(x))$ is 
equal to the same coefficient in 
$$
u_1(t_0x+_Gt_1x^3)^3+_G(t_0x)^9 
$$
which is clearly equal to $u_1t_1^3+t_0^9$ and hence we get (b). 
 
c) The coefficient of $x^{27}$ in $h_g([3]_{g_*G})(x)$
is equal to the coefficient of $x^{27}$ in 
$$
t_0\big(g_*(u_1)x^3+_{g_*G}x^9\big)
+_G t_1\big(g_*(u_1)x^3+_{g_*G}x^9\big)^3+_G t_2\big(g_*(u_1)x^3\big)^9 
$$ 
and the latter coefficient is equal to 
$$
t_1+t_2g_*(u_1)^9+c
$$ 
where $c$ is the coefficient of $x^{27}$ in 
$$
t_0\big(g_*(u_1)x^3+_{g_*G}x^9\big)+_Gt_1g_*(u_1)^3x^9 \ . 
$$ 
Next we observe that Corollary \ref{fgl} yields   
\begin{eqnarray}
g_*(u_1)x^3+_{g_*G}x^9&\equiv & 
g_*(u_1)x^3+x^9-g_*(u_1)^3x^{15} \nonumber \\ 
&& -g_*(u_1)^2x^{21}+g_*(u_1)^6x^{21}-g_*(u_1)^9x^{27}
\ \  \mod\ (x^{28}) \ .\nonumber   
\end{eqnarray} 
Applying Corollary \ref{fgl} once more and using (a) 
and calculating modulo $(u_1^7)$ we obtain  
$$
c\equiv -u_1t_0^2t_1g_*(u_1)^3=-u_1^4t_0^8t_1\ \ \ \mod\ (u_1^7)
$$ 
and hence modulo $(u_1^7)$ the coefficient of $x^{27}$ in 
$h_g([3]_{g_*G})(x)$ is equal to 
$$
t_1-u_1^4t_0^8t_1  \ .
$$  
On the other hand the coefficient of $x^{27}$ in ${[3]_G}(h_g(x))$ is equal 
to the same coefficient in 
$$
u_1(t_0x+_Gt_1x^3+_Gt_2x^9)^3+_G(t_0x+_Gt_1x^{3})^9   
$$ 
and this coefficient is equal to 
$$
u_1t_2^3+t_1^9+d
$$
where $d$ is the coefficient of $x^{27}$ in 
$$
u_1(t_0x+_Gt_1x^3)^3+_Gt_0^9x^9 \ .  
$$ 
Next we observe that Corollary \ref{fgl} yields   
\begin{eqnarray}
(t_0x+_Gt_1x^3)^3&\equiv & 
t_0^3x^3+t_1^3x^9-u_1^3t_0^6t_1^3x^{15} \nonumber \\ 
&&-u_1^3t_0^3t_1^6x^{21}+u_1^6t_0^{12}t_1^3x^{21}
-u_1^9t_0^{18}t_1^3x^{27}\ \ \ \ \mod\ (x^{28}) \ .\nonumber   
\end{eqnarray} 
Applying Corollary \ref{fgl} once more and using (a) 
and calculating modulo $(u_1^7)$ we obtain 
$$
\begin{array}{lcll}
u_1(t_0x+_Gt_1x^3)^3+_Gt_0^9x^9&\equiv &u_1\big(t_0^3x^3+t_1^3x^9
-u_1^3t_0^6t_1^3x^{15}-u_1^3t_0^3t_1^6x^{21})+t_0^9x^9& \\ 
&&-u_1^3\big(t_0^3x^3+t_1^3x^9-u_1^3t_0^6t_1^3x^{15}-
u_1^3t_0^3t_1^6x^{21})^2t_0^9x^9& \\ 
&&-u_1^2\big(t_0^3x^3+t_1^3x^9-u_1^3t_0^6t_1^3x^{15}-
u_1^3t_0^3t_1^6x^{21})t_0^{18}x^{18}& \\ 
&&+u_1^6\big(t_0^3x^3+t_1^3x^9-u_1^3t_0^6t_1^3x^{15}-
u_1^3t_0^3t_1^6x^{21})^4t_0^9x^9&  \text{mod} \ (x^{28})  
\end{array}
$$
and hence 
$$
d\equiv -u_1^3(t_1^6-2u_1^3t_0^9t_1^3)t_0^9-u_1^2t_0^{18}t_1^3
+4u_1^6t_0^{18}t_1^3 \mod (u_1^7) \ .
$$ 
Therefore the coefficient of $x^{27}$ in ${[3]_G}(h_g(x))$ is equal to 
$$
t_1^9+u_1t_2^3-u_1^2t_0^{18}t_1^3-u_1^3t_0^9t_1^6+2u_1^6t_0^{18}t_1^3
+4u_1^6t_0^{18}t_1^3 \ \ \ \ \ \mod (u_1^7)
$$
and (c) follows.  

d) The coefficient of $x^{81}$ in $h_g([3]_{g_*G})(x)$
is equal to the same coefficient in 
$$
t_0\big(g_*(u_1)x^3+_{g_*G}x^9\big)
+_G t_1\big(g_*(u_1)x^3+_{g_*G}x^9\big)^3+_G
t_2\big(g_*(u_1)x^3+_{g_*G}x^9\big)^9+_Gt_3(g_*(u_1)x^3)^{27}
$$ 
which modulo $(u_1^3)$ is equal to the same coefficient in 
$$
t_0\big(g_*(u_1)x^3+_{g_*G}x^9\big)+_G t_1x^{27}+_Gt_2x^{81} 
$$ 
and by Corollary \ref{fgl} this is easily seen to be equal to $t_2$ 
modulo $(u_1^2)$.

On the other hand the coefficient of $x^{81}$ in ${[3]_G}(h_g(x))$ is equal 
to the same coefficient in 
$$
u_1(t_0x+_Gt_1x^3+_Gt_2x^9+t_3x^{27})^3+_G(t_0x+_Gt_1x^{3}+t_2x^9)^9   
$$ 
and this coefficient is equal to 
$$
u_1t_3^3+t_2^9+e
$$
where $e$ is the coefficient of $x^{81}$ in the series 
$$
u_1(t_0x+_Gt_1x^3+_Gt_2x^9)^3+_G(t_0x+_Gt_1x^{3})^9  \ . 
$$ 
Now $g\equiv 1\ \text{mod}\ (S^2)$ implies 
$t_1\equiv 0\ \text{mod}\ (u_1)$  and thus modulo 
$(u_1^2)$ we find that $e$ is also the coefficient of $x^{81}$ in 
$$
u_1(t_0x+_Gt_2x^9)^3+_Gt_0^9x^9   
$$ 
and by Corollary \ref{fgl} even of the coefficient of $x^{81}$ in 
$$
u_1(t_0x+_Gt_2x^9)^3\ .   
$$ 
Now Lemma \ref{G_2mod-m} below shows that the coefficient of 
$x^{27}$ in $t_0x+_Gt_2x^9$ is trivial modulo $(u_1)$ 
(if not, either the coefficient of $x^{18}y$ or of $x^9y^2$ in $x+_Gy$ 
would have to be nontrivial modulo $(u_1)$), 
hence $e$ is trivial modulo $(u_1^2)$ 
and the proof of (d) is complete. 
\end{proof}  
\smallskip
 
\begin{lem}\label{G_2mod-m}
$$
x+_{G_2}y \equiv x+y+\sum_{i\geq 1}P_{8i+1}(x,y)  \mod (3,u_1)
$$
where $P_{8i+1}$ is a homogeneous polynomial of degree $8i+1$ without terms 
$x^{8i+1}$ and $y^{8i+1}$. 
\end{lem}

\begin{proof} It is enough to show this for the graded formal group law 
$F_2$ over $(E_2)_*$. This group law is a homogeneous series 
of degree $-2$ if $x$ and $y$ are given degree $-2$ and thus, if we write 
$$
x+_{G_2}y \equiv x+y+\sum_{j\geq 1}P_j(x,y) \mod (3,u_1)
$$
with homogeneous polynomials in $x$ and $y$ of degree $-2j$  
then the coefficients in $P_j$ have to be in $(E_2)_{2j-2}$. 
Furthermore, this group law has its coefficients in the 
subring generated by $u_1u^{-2}$ and $u^{-8}$. 
However, $(E_2)_*/(3,u_1)\cong \F_9[[u^{\pm 1}]]$ and thus 
$2j-2$ has to be a multiple of $16$.  
\end{proof}
\smallskip

\vfil\eject
\begin{cor}\label{ti-action} \mbox{} The following equations hold in 
$(E_2)_0/(3)$.  
\begin{itemize}				
\item[a)] Let $g=1+g_1S+g_2S^2 \mod (S^3)$. Then we have 
$$
\begin{array}{lcll} 
t_1 &\equiv & g_1+g_2^3u_1-g_1^3u_1^2-g_1^6u_1^3 
&\ \mod \ (u_1^4) \\
t_0 &\equiv & 
1+g_1^3u_1-g_1u_1^3+(g_2-g_2^3)u_1^4+g_1^3u_1^5+(g_1^2+g_1^6)u_1^6 
&\ \mod \ (u_1^7) \ .\\
\end{array}
$$
\item[b)] Let $g=1+g_2S^2+g_3S^3\mod (S^4)$. Then we have 
$$
\begin{array}{lcll} 
t_2 &\equiv & g_2+ g_3^3u_1  &\ \mod \ (u_1^2) \\
t_1 &\equiv & g_2^3u_1+g_3u_1^4+(g_2^3-g_2)u_1^5
&\ \mod \ (u_1^7) \\
t_0 &\equiv & 1+(g_2-g_2^3)u_1^4-g_3u_1^7+(g_2-g_2^3)u_1^8 
&\ \mod \ (u_1^{10}) \ . \\
\end{array}
$$
\end{itemize}
\end{cor}

\begin{proof} a) From Proposition \ref{t_i}.c we obtain 
$$
t_1\equiv t_1^9+t_2^3u_1-t_0^{18}t_1^3u_1^2-t_0^9t_1^6u_1^3
\ \ \ \text{mod} \ (u_1^4)
$$ 
and by using (\ref{t_imodm}) we immediately 
get the formula for $t_1$ modulo $(u_1^4)$. 
Then Proposition \ref{t_i}.b and (\ref{t_imodm}) yield 
$$
t_0+t_0^6(g_1+g_2^3u_1-g_1^3u_1^2-g_1^6u_1^3)u_1^3\equiv 
1+(g_1^3+g_2u_1^3)u_1 \ \ \ \text{mod} \ (u_1^7) 
$$ 
from which we easily get the formula for $t_0$ modulo $(u_1^7)$. 
The formula for $g_*(u_1)$ follows now from Proposition \ref{t_i}.a\ . 

b) From Proposition \ref{t_i}.d and (\ref{t_imodm}) 
we immediately obtain the formula for $t_2$. 
Substituting the value for $t_2$ into the formula of 
Proposition \ref{t_i}.c and using (\ref{t_imodm}) 
yields
$$
t_1-t_0^8t_1u_1^4\equiv (g_2^3+g_3u_1^3)u_1-t_0^{18}t_1^3u_1^2
\ \ \ \text{mod} \ (u_1^7)\ . 
$$ 
Substituting the values of $t_0$ modulo $(u_1^7)$ and $t_1$ modulo 
$(u_1^4)$ from (a) into this yields 
$$
t_1-g_2^3u_1^5\equiv (g_2^3+g_3u_1^3)u_1-g_2u_1^5
\ \ \ \text{mod} \ (u_1^7)  
$$ 
from which we get the value of $t_1$ modulo $(u_1^7)$. 
Next we substitute this value of $t_1$ together with 
the value of $t_0$ modulo $(u_1^7)$ of (a) into the formula of 
Proposition \ref{t_i}.b and obtain 
$$
t_0+\big(1+(g_2-g_2^4)u_1^4\big)^6
\big(g_2^3u_1+g_3u_1^4+(g_2^3-g_2)u_1^5\big)
u_1^3\equiv 1+g_2u_1^4 
\ \ \ \text{mod}\ (u_1^{10}) 
$$
from which we easily get the formula for $t_0$ modulo $(u_1^{10})$. 
\end{proof}
 
\noindent

The following calculation will be used repeatedly in  
later sections. The result is only given to the precision needed later.  

\begin{lem}\label{t0k} Let $g=1+g_1S+g_2S^2 \mod (S^3)$ 
and let $k$ be an integer. 
Then we have 
$$
t_0(g)^k\equiv 
\left\{\begin{array}{lll}
1+g_1^3u_1+(k'-1)g_1u_1^3+(k'g_1^4+g_2-g_2^3)u_1^4+g_1^3u_1^5 
&\mod (3,u_1^6)&  k=3k'+1\\ 
1-g_1^3u_1+g_1^6u_1^2+(k'+1)g_1u_1^3 +\\
(g_1^4-k'g_1^4+g_2^3-g_2)u_1^4+(k'g_1^7-g_1^3-g_1^3g_2+g_1^3g_2^3)u_1^5 
&\mod (3,u_1^6)&  k=3k'+2\ .\\
\end{array} \right. 
$$
\end{lem}

\begin{proof} The result follows easily from Corollary 
\ref{ti-action} and from  
$$
t_0(g)^k=\big((1+(t_0(g)-1)\big))^k 
\equiv \sum_{j=1}^5 \binom{k}{3}
(t_0(g)-1)^j \ \mod \ \ (3,u_1^6)
$$ 
by using that 
$\binom{k}{j}\equiv \prod_i\binom{k_i}{j_i}\ \mod (p)$ 
if $k=\sum_ik_ip^i$ and $j=\sum_ij_ip^i$ are the 
$p$-adic expansions of $k$ and $j$ respectively.   
\end{proof}

Using Lemma \ref{abcd} we finally get the following 
information on the action of $a$, $b$, $c$ and $d$ on $(E_2)_*/(3)$. 
We use $1$ and $\omega^2$ as a basis of $\FF_9$ considered 
as an $\FF_3$-vector space (rather then $1$ and $\omega$).

\begin{cor}\label{u1-action} The action of the elements 
$a$, $b$, $c$ and $d$ on $(E_2)_*/(3)$ satisfy the formulae 
$$
\begin{array}{lcll}
a_*u_1&\equiv & u_1-(1+\omega^2)u_1^2-\omega^2u_1^3+(1-\omega^2)u_1^4- 
u_1^5 -(1+\omega^2)u_1^6  & \mod (u_1^{7}) \\  
b_* u_1&\equiv & u_1+ u_1^2+u_1^3-u_1^4 +(1+\omega^2)u_1^5+
(1-\omega^2)u_1^6  & \mod (u_1^{7}) \\  
c_*u_1&\equiv & u_1-\omega^2u_1^5+(-1+\omega^2)u_1^8-(1+\omega^2)u_1^9  
& \mod (u_1^{11})\\ 
d_*u_1&\equiv &u_1+\omega^2u_1^8 & \mod (u_1^{11}) \\  
&&&\\
a_*u&\equiv & \big(1+(1+\omega^2)u_1+(-1+\omega^2)u_1^3
+(1+\omega^2)u_1^5\big)u
& \mod (u_1^{6}) \\  
b_* u&\equiv &  \big(1-u_1+u_1^3-\omega^2u_1^4-u_1^5\big)u
& \mod (u_1^{6}) \\  
c_*u&\equiv & \big(1+\omega^2u_1^4+(1-\omega^2)u_1^7+\omega^2u_1^8\big)u
& \mod (u_1^{10})\\ 
d_*u&\equiv &\big(1-\omega^2u_1^7\big)u& \mod (u_1^{10})\ . \qed  
\end{array}
$$
\end{cor}  
\bigbreak

\section{The $E_2$-term of the algebraic spectral sequence}\label{e2-term} 

\subsection{The $E_1$-term}\label{e1-term} We begin by giving 
some background on Theorem \ref{shapiro}, or equivalently, on the 
$E_1$-term of the spectral sequence (\ref{algss}) 
$$
E_1^{s,t,*}=\Ext^t_{\Z_3[[\GG_2^1]]}(C_s,(E_2)_*/(3))
\Longrightarrow H^{s+t}(\GG_2^1,(E_2)_*/(3)) \ . 
$$ 
We note that for $s=1,2$ the module $C_s$ is projective 
as $\Z_3[[\GG_2^1]]$-module and thus we have a Shapiro isomorphism 
\begin{equation}
E_1^{s,t}=
\Ext^t_{\Z_3[[\GG_2^1]]}(\chi\uparrow^{\G_2^1}_{SD_{16}},(E_2)_*/(3))
\cong 
\begin{cases} \Hom_{\Z_3[SD_{16}]}(\chi,(E_2)_*/(3)) & t=0\\
0                                                 & t>0 \  .
\end{cases} 
\end{equation}
The action of $SD_{16}$ on $(E_2)_{*}$ is known (cf. the proof of Lemma 22 
of \cite{hennkn} for an explicit reference) to be given by 
\begin{equation}\label{SD16-action}
\omega_* u_1 = \omega^2 u_1\ \ \textrm{and}\ \ \omega_* u = \omega u 
\end{equation} 
and the Frobenius $\phi$ acts $\Z_3$-linearly by extending the 
action of Frobenius on $\W$ via 
\begin{equation}\label{Frobenius-action}
\phi_*u_1 =u_1 \ \ \textrm{and} \ \ \phi_*u =u \ .
\end{equation}
This implies immediately that 
$(E_2)_*/(3)^{SD_{16}}$ is isomorphic to $\FF_3[[u_1^4]][v_1,u^{\pm 8}]$ 
as a graded algebra and that there 
is an isomorphism of $(E_2)_*/(3)^{SD_{16}}$-modules  
\begin{equation}\label{TMF(2)}
\Ext^0_{\Z_3[SD_{16}]}(\chi,(E_2)_*/(3)) \cong 
\omega^2u^{4}\mathbb{F}_3[[u_1^{4}]][v_1,u^{\pm8}].
\end{equation}

For $s=0,3$ we have a Shapiro isomorphism 
$$
E_1^{s,t}=\Ext^t_{\Z_3[[\GG_2^1]]}(\Z_3[[\GG_2^1/G_{24}]],(E_2)_*/(3))
\cong H^t(G_{24},(E_2)_*/(3)) \ .
$$ 
Let $G_{12}$ be the subgroup of $G_{24}$ generated by 
the elements $a$ and $t$. The calculation of the cohomology algebra 
$H^*(G_{12},(E_2)_*/(3))$ was deduced from that of $H^*(G_{12},(E_2)_*)$ 
in section 1.3 of \cite{GHM}. In precisely the same way one deduces 
the calculation of $H^*(G_{24},(E_2)_*/(3))$ from that of 
$H^*(G_{12},(E_2)_*)$ which was given in section 3 of \cite{GHMR}. 
In particular there are classes  
$$
\begin{array}{ll}
\Delta\in H^0(G_{24},(E_2)_{24}/(3)),&  \a\in H^1(G_{24},(E_2)_{4}/(3))\\ 
\widetilde{\a}\in H^1(G_{24},(E_2)_{12}/(3)),& \b\in H^2(G_{24},(E_2)_{12}/(3)) \\
\end{array}
$$
and an isomorphism of algebras 
\begin{equation}\label{TMF} 
H^*(G_{24},(E_2)_*/(3))\cong 
\F_3[[v_1^6\Delta^{-1}]][v_1,\Delta^{\pm 1},\b,\a,\widetilde{\a}]/
(\a^2,\widetilde{\a}^2,v_1\a,v_1\widetilde{\a},\a\widetilde{\a}+v_1\b)\ .
\end{equation} 
In the sequel we need some control over the elements 
occuring in this isomorphism (cf. section 1.3 of \cite{GHM}). 
First we recall that $\alpha$ is defined as $\delta^0(v_1)$ 
where $\delta^0$ is the Bockstein with respect to the 
short exact sequence of continuous $\Z_3[[\GG_2]]$-modules
\begin{equation}\label{delta^0}
0\to (E_2)_*/(3)\stackrel{3}\to (E_2)_*/(9)\to (E_2)_*/(3)\to 0 \ .
\end{equation}
Similarly, $v_2:=u^{-8}$ determines an invariant in 
$H^0(\GG_2,(E_2)_{16}/(3,u_1))$ and  
$\widetilde{\a}$ is defined as $\delta^1(v_2)$ where $\delta^1$ 
is the Bockstein with respect to the short exact sequence 
of continuous $\Z_3[[\GG_2]]$-modules 
\begin{equation}\label{delta^1}
0\to\Sigma^4(E_2)_*/(3)\stackrel{v_1}\to (E_2)_*/(3)\to 
(E_2)_*/(3,u_1)\to 0 \ . 
\end{equation}
Next $\b$ is defined to be the mod $3$-reduction of $\d^0\d^1(v_2)$. 
These elements are thus defined as elements in $H^*(\GG_2,(E_2)_*/(3))$. 
We denote their restriction to $H^*(G_{24},(E_2)_*/(3))$ 
by the same name.  

The relation between $\Delta$ (which lifts to an invariant 
of the same name in $H^0(G_{24},(E_2)_{24})$) and the classes 
$u$ and $u_1$ is more subtle. Here we record the following result. 
\smallbreak 

\begin{prop}\label{Delta}  
$\Delta\equiv (1-\omega^2u_1^2+u_1^4)\omega^2u^{-12}
\ \mod\ (3,u_1^6)$.   
\end{prop}

\begin{proof}
By (3.11) of \cite{GHMR} the integral lift of $\Delta$ is defined as  
$$
\Delta = \frac{\omega^2}{4\big(x(a_*x)(a_*(a_*x))\big)^4}
$$ 
and by the proof of Lemma 3.1 of \cite{GHMR} we know %that 
$x\equiv u \ \text{mod}\ (3,u_1)$, hence 
$\Delta \equiv \omega^2u^{-12} \ \text{mod}\ (3,u_1)$.  
Because $\Delta$ is invariant with respect to $G_{24}$, 
in particular with respect to $Q_8$, we get from 
(\ref{SD16-action}) and (\ref{Frobenius-action}) 
that $\Delta$ is of the form 
$$
\Delta\equiv (1+\lambda_2\omega^2u_1^2+\lambda_4u_1^4)\omega^2u^{-12} 
\ \ \ \mod (3,u_1^6)\ \ 
\textrm{with} \ \ \lambda_i \in \mathbb{F}_3 \subset \mathbb{F}_9\ . 
$$ 
Because $\Delta$ is also invariant with respect to the action of $a$ we get 
from (\ref{u-action}) 
$$
\Delta=a_*(\Delta)\equiv 
t_0(a)^{-12}(1+\l_2\omega^2a_*(u_1)^2+\l_4a_*(u_1)^4)\omega^{2}u^{-12} 
\ \ \ \text{mod}\ (3,u_1^6)\ \ . 
$$ 
The right hand side of this equation can be evaluated modulo $(u_1^6)$  
by using Corollary \ref{ti-action}.a and Corollary \ref{u1-action}. 
By looking at the coefficients of $u_1^3$ and $u_1^5$ in the 
right hand side we obtain $\l_2=-1$ and $\l_4=1$.  
\end{proof} 
\medbreak

\subsection{The $d_1$-differential}
 
First of all we note that all differentials are $v_1$-linear.  

\begin{lem}\label{propD} Let $k\not\equiv 0 \mod (3)$. 
Then the differential $d_1:E_1^{0,0} \to E_1^{1,0}$ satisfies 
$$
d_1(\Delta^k) \equiv  
\left\{\begin{array}{llcl}
(-1)^{m+1}\omega^2(1+u_1^4)u^{-12k}&\mod (u_1^8)& &k=2m+1 \\
(-1)^{m+1}m\omega^2u_1^2u^{-12k} &\mod (u_1^6)& &k=2m. \\
\end{array}\right.
$$ 
\end{lem}

\begin{proof}
By Corollary \ref{delta1} the differential is induced by the 
homomorphism $C_1\to C_0$ which sends  $e_1$ to $(e-\omega)e_0$. 
Furthermore Proposition \ref{Delta} and (\ref{SD16-action}) give 
$$
\begin{array}{lcll} 
\Delta^k&\equiv & 
\big(1-k\omega^2u_1^2+ku_1^4-\binom{k}{2}u_1^4\big)\omega^{2k}u^{-12k} 
&\ \mod (u_1^6) \\
\omega_*(\Delta^k)&\equiv &
\big(1+k\omega^2u_1^2+ku_1^4-\binom{k}{2}u_1^4\big)\omega^{2k-12k}u^{-12k} 
&\ \mod (u_1^6) \\
\end{array}
$$ 
and the result follows easily. (Note that by (\ref{TMF(2)}) 
the congruence for $d_1(\Delta^{2m+1})$ improves to a 
congruence modulo $(u_1^8)$ rather than only modulo $(u_1^6)$.)  
\end{proof} 
\medbreak 

\begin{prop}\label{d_1^{0,0}}  
For each integer $k\neq 0$ there exists 
an element $\Delta_k \in E_1^{0,0,24k}$ 
such that
\begin{itemize}
\item[a)] $\Delta_k \equiv \Delta^k \mod  (u_1^6)$
\item[b)] the differential $d_1:E_{1}^{0,0}\to E_1^{1,0}$ satisfies 
$$
d_1(\Delta_k) \equiv  
\left\{\begin{array}{lll}
(-1)^{m+1}\omega^2u^{-12k}& \mod (u_1^4)& k=2m+1  \\
(-1)^{m+1}m\omega^2u_1^{4.3^n-2}u^{-12k}  
&\mod (u_1^{4.3^n+2})& 
k=2.3^nm, \ m\not\equiv 0 \mod\ (3) \\
0 &   & k=0 \ .\\
\end{array}\right.
$$
\end{itemize}
\end{prop}

\begin{proof} For $k=0$, $k$ odd or $k=2m$ with 
$m\not\equiv 0 \textrm{ mod } (3)$ 
we define $\Delta_k$ to be equal to 
$\Delta^k$. The formula for $d_1$ is then satisfied by 
the previous result. 

If $k=2.3^{n}m$ with $n\geq 0$ and 
$m\not \equiv 0 \textrm{ mod }(3)$ we recursively define 
\begin{equation}
\Delta_{3k}:\ =\Delta_k^3-
mv_1^{3(4\cdot 3^n-2)}\Delta^{3k-2.3^n+1}.
\end{equation}
The previous proposition gives   
$$
d_1(\Delta^{3k-2.3^n+1})
\equiv (-1)^{m}\omega^2(1+u_1^4)u^{-12(3k-2.3^n+1)}\mod (u_1^6)\ .
$$
and by induction on $n$ we have  
$$
d_1(\Delta_{k})^3\equiv 
\big((-1)^{m+1}m\omega^2u_1^{4.3^n-2}u^{-12k}\big)^3
\equiv (-1)^{m}m\omega^2u_1^{4.3^{n+1}-6}u^{-12\cdot3k}
\mod (u_1^{4.3^{n+1}+6})\ .
$$ 
Therefore $v_1$-linearity of $d_1$ yields 
$$
\begin{array}{llll}
d_1(\Delta_{3k})&=&d_1(\Delta_{k}^3)-mv_1^{3(4.3^n-2)}d_1(\Delta_{3k-2.3^n+1})& \\
&\equiv &(-1)^{m+1}\omega^2mu_1^{4.3^{n+1}-2}u^{12\cdot3k} 
&\mod (u_1^{4.3^{n+1}+2}) 
\end{array}
$$ 
and the induction step is established. 
\end{proof}

\begin{cor} There is an isomorphism of  $\mathbb{F}_3[v_1]$-modules 
$E_2^{0,0}\cong \mathbb{F}_3[v_1]$. \qed
\end{cor}

\remarkstyle{Remark} 
By (\ref{TMF}) the elements $\Delta_k$ form a 
topological basis of the continuous graded $\F_3[v_1]$-module 
$E_1^{0,0}$ (in fact, this has been implicitly used in the last 
corollary) and by (\ref{TMF(2)}) a topological basis of the 
continuous graded $\F_3[v_1]$-module $E_1^{1,0}$ can be 
given by any family of elements $b_{2k+1}$, $k\in \Z$, such that 
$b_{2k+1}\equiv \omega^2u^{-8k-4}\ \text{mod} \ (u_1^4)$. 
%(Likewise for $E_1^{2,0}$ and $E_1^{3,0}$.) 
By Proposition \ref{d_1^{0,0}} we know that there are such 
elements $b_{2.(3m+1)+1}$ for $m\in \Z$ and $b_{2.3^n(3m-1)+1}$ 
for $n\geq 0$, $m\not\equiv 0\ \text{mod} \ (3)$, such that 
the first formula of Theorem \ref{d1} holds. Because the 
$E_1$-term is torsion free as a $\F_3[v_1]$-module and the 
$d_1$-differential is $\F_3[v_1]$-linear it is clear that 
those $b_{2k+1}$'s are in the kernel of the differential 
$d_1:E_1^{1,0}\to E_1^{2,0}$. To complete this family 
to a topological basis we need to choose elements $b_{2k+1}$ for 
$k=3^{n+1}(3m+1)$ with $n\geq 0$, $m\in \Z$, for $k=3^n(9m+8)$ with 
$n\geq 0$, $m\in \Z$, and for $k=0$. Thus we are lead to 
concentrate on the differential on $\omega^2u^{-4(2k+1)}$ 
for such $k$. The crucial step is given by Proposition 
\ref{d_1b_k} below whose proof is quite elaborate and 
will be postponed to section \ref{middle-differential}. 
The proof of Proposition \ref{b_0} will be given in section 
\ref{higher+extensions}. 

\begin{prop}\label{b_0} There exists an element $b_1 \in E_1^{1,0}$ such that 
$b_1 \equiv \omega^2u^{-4} \mod (u_1^4)$ and 
$v_1\alpha=b_1$ in $H^*(\mathbb{G}_2^1;(E_2)_{*}/(3))$. In particular, 
$d_1(b_1)=0$. 
\end{prop}

\begin{prop}\label{d_1b_k} Let $k$ be an integer such that $8k+4$ is not 
divisible by $3$. Then the differential $d_1:E_1^{1,0}\to E_1^{2,0}$ satisfies  
$$
\begin{array}{lcrll}
d_1(\omega^2u^{8k+4})&\equiv 
&-(k'+k'^2)\omega^2u_1^{12}u^{8k+4}&\mod (u_1^{16})&\ 8k+4=3k'+1 \\
d_1(\omega^2u^{8k+4})&\equiv 
&(k'-k'^2)\omega^2u_1^{8}u^{8k+4}&\mod (u_1^{12})  &\ 8k+4=3k'+2\ .
\end{array}
$$
\end{prop}

\begin{proof} This will be proved in section \ref{middle-differential}.  
\end{proof}

\begin{prop}\label{d_1^{1,0}} 
For each integer $k$ there exists an element 
$b_{2k+1} \in E_1^{1,0,8(2k+1)}$ such that 
\begin{itemize}
\item[a)] $b_{2k+1}\equiv \omega^2 u^{-4(2k+1)} \mod (u_1^4)$
\item[b)]
$d_1(\Delta_k)=
\left\{\begin{array}{ll}
(-1)^{m+1}b_{2(3m+1)+1}& k=2m+1, m\in \Z \\ 
(-1)^{m+1}mv_1^{4.3^n-2}b_{2.3^n(3m-1)+1}& k=2m.3^n,
\ m\not\equiv 0 \mod (3) \\
0 & k=0 \\
\end{array}\right.$ 
\item[c)] the differential $d_1:E_1^{1,0} \to E_1^{2,0}$ satisfies 
$$
d_1(b_{2k+1})\equiv  
\left\{\begin{array}{llll} 
(-1)^n\omega^2u_1^{6.3^{n}+2}u^{-4(2k+1)}&\mod (u_1^{2.3^{n+1}+6}) 
& k=3^{n+1}(3m+1), &m\in \Z  \\
(-1)^n\omega^2u_1^{10.3^n+2}u^{-4(2k+1)}  &\mod (u_1^{10.3^n+6})
& k=3^n(9m+8), &m\in \Z \\
0 &&\textrm{otherwise} \ .&  
\end{array}\right.
$$
\end{itemize}
\end{prop}

\begin{proof} 
For  $k=3m+1$ with $m\in \Z$ we define $b_{2k+1}$ to be 
$(-1)^{m+1}d_1(\Delta_{2m+1})$. For $k=3^n(3m-1)$ with $n\geq 0$ and 
$m\not\equiv 0\ \text{mod}\ (3)$ we note that Proposition 
\ref{d_1^{0,0}} shows that $d_1(\Delta_{2m.3^n})$ is divisible 
by $v_1^{4.3^n-2}$ and we define $b_{2k+1}$ to be 
$(-1)^{m+1}mv_1^{-(4.3^n-2)}d_1(\Delta_{2m.3^n})$. For $k=0$ 
we take the element given in Proposition \ref{b_0}. 
With these definitions (b) holds as well as the last case of (c).  

For $k=3^{n+1}(3m+1)$ resp. $k=3^n(9m+8)$, 
with $n\geq 0$ and $m\in \Z$, we define elements 
$b_{2k+1}$ by induction 
on $n$ such that (a) and (c) are satisfied. 
In fact, for $n=0$ we define 
$$
b_{2k+1}:\ =\omega^2u^{-4(2k+1)} 
$$
and then Proposition \ref{d_1b_k} gives
$$
d_1(b_{18m+7})\equiv\omega^2u_1^8u^{-4(18m+7)} \mod (u_1^{12}), \ \ 
d_1(b_{18m+17})\equiv\omega^2u_1^{12}u^{-4(18m+17)} \mod (u_1^{12})\ . 
$$
Now suppose that $b_{2k+1}$ has already been defined for 
$k=3^{n+1}(3m+1)$ resp. $k=3^n(9m+8)$ with  
with $n\leq N$ and $m\in \Z$ so that (a) and (c) are satisfied.   
Then we observe that by Proposition \ref{propD} the elements   
$$
b_{2k+1}^3+b_{2(3k+1)+1}=
b_{2k+1}^3+(-1)^{k+1}d_1(\Delta_{2k+1})\equiv 
\big(\omega^6+\omega^2(1+u_1^4)\big)u^{-4(6k+3)} 
\ \text{mod} \ (u_1^8)
$$
are divisible by $v_1^4$ and thus we can define 
\begin{equation}
b_{6k+1}:\  =v_1^{-4}\big(b_{2k+1}^3+b_{2(3k+1)+1}\big)  \ . 
\end{equation}
Then it is clear that 
$b_{6k+1}\equiv \omega^2u^{-4(6k+1)}\mod (u_1^4)$. 
Furthermore,  $d_1d_1(\Delta_{2k+1})=0$ and because $d_1$ 
%satisfies $d_1(x^3)=d_1(x)^3$ for any $x$ 
commutes with taking third powers  
and is $\F_3[v_1]$-linear we see that both (a) and (c) are 
satisfied for $k=3^{n+1}(3m+1)$ resp. $k=3^n(9m+8)$   
with $n\leq N+1$ and $m\in \Z$ and thus the induction step is complete. 
\end{proof}

\begin{cor} There is an isomorphism of $\mathbb{F}_3[v_1]$-modules  
$$
E_2^{1,0}\cong\prod_{n\geq 0,m\in\mathbb{Z}\backslash 3\mathbb{Z}} 
\mathbb{F}_3[v_1]/(v_1^{4.3^n-2})\{b_{2.3^n(3m-1)+1}\}\times 
\mathbb{F}_3[v_1]\{b_1\}\ . \qed 
$$
\end{cor}

\begin{proof} Because the elements $\Delta_k$ and $b_k$ form a topological 
basis of the graded continuous $\F_3[v_1]$-modules $E_1^{0,0}$ 
and $E_1^{1,0}$ this follows immediately from Proposition \ref{d_1^{1,0}}. 
\end{proof}

\remarkstyle{Remark} By inspection one sees that the infinite product 
is finite in each bidegree and therefore it can also be identified 
with the direct sum.

To evaluate the homomorphism $d_1:E_1^{2,0}\to E_1^{3,0}$ we need 
the following result.

\begin{lem}\label{d_1^{2,0}u} Let $k$ be any integer.  
Then 
$$
\begin{array}{llll} 
(e_*+a_*+(a^2)_*)(u^k)&\equiv & \big((k-k^2)\omega^2u_1^2
+(k\binom{k}{3}+k-k^2)u_1^4\big)u^k& \mod \ (3,u_1^5) \ .
\end{array}
$$
\end{lem}

\begin{proof} 
By (\ref{u-action}) we have $a_*(u^k)=u^kt_0(a)^k$ 
and $(a^2)_*(u^k)=u^kt_0(a)^ka_*(t_0(a)^k)$. Corollary \ref{ti-action} 
gives 
$$
\begin{array}{lcll}
t_0(a)^k&\equiv &
\big(1+(1+\omega^2)u_1+(-1+\omega^2)u_1^3\big)^k &\\
&\equiv& 1+k(1+\omega^2)u_1+k(-1+\omega^2)u_1^3 &\\
&&-\binom{k}{2}\omega^2u_1^2
-\binom{k}{2}(1+\omega^2)(-1+\omega^2)u_1^4+
\binom{k}{3}(1+\omega^2)^3u_1^3-\binom{k}{4}u_1^4& \\
&\equiv & 1+k(1+\omega^2)u_1-\binom{k}{2}\omega^2u_1^2& \\
&&+(\binom{k}{3}-k)(1-\omega^2)u_1^3
-(\binom{k}{2}+\binom{k}{4})u_1^4 & \text{mod}\ (3,u_1^{5})\\
\end{array}
$$ 
and by Corollary \ref{u1-action} we get 
$$
\begin{array}{lcll}
a_*(t_0(a)^k)&\equiv& 1+k(1+\omega^2)
\big(u_1-(1+\omega^2)u_1^2-\omega^2u_1^3+(1-\omega^2)u_1^4\big)&\\
&&-\binom{k}{2}\omega^2\big(u_1^2+(1+\omega^2)u_1^3\big)&\\
&&+(\binom{k}{3}-k)(1-\omega^2)u_1^3-(\binom{k}{2}+\binom{k}{4})u_1^4 &\\
&\equiv & 1+ k(1+\omega^2)u_1+(k-\binom{k}{2})\omega^2u_1^2 &\\
&&+(\binom{k}{3}+\binom{k}{2})(1-\omega^2)u_1^3
-(k+\binom{k}{2}+\binom{k}{4})u_1^4 & \text{mod}\ (3,u_1^{5}) \ .\\ 
\end{array}
$$
Finally an easy calculation (which only uses that 
$\binom{k}{2}\equiv -k(k-1)\ \text{mod}\ (3)$ and 
$k^3\equiv k\ \text{mod}\ (3)$) gives   
$$
\begin{array}{llll}
t_0(a)^ka_*(t_0(a)^k)&\equiv &
1-k(1+\omega^2)u_1+(k^2-k)\omega^2u_1^2 &\\
&&+(-\binom{k}{3}+k)(1-\omega^2)u_1^3 &\\
&&+(\binom{k}{4}+\binom{k}{2}+k\binom{k}{3}+k-k^2)u_1^4 
& \text{mod} \ (3,u_1^5) 
\end{array} 
$$ 
and the result clearly follows. 
\end{proof} 

\begin{prop}\label{d_1^{2,0}} For each integer $k$ there exists an element 
$\overline{b}_{2k+1} \in E_1^{2,0,4(2k+1)}$ such that 
\begin{itemize}
\item[a)] $\overline{b}_{2k+1} \equiv \omega^2u^{-4(2k+1)} \mod (u_1^4)$
\item[b)] 
$d_1(b_{2k+1})=
\left\{\begin{array}{ll}
(-1)^nv_1^{6.3^{n}+2}\overline{b}_{3^{n+1}(6m+1)} & k=3^{n+1}(3m+1) \\
(-1)^nv_1^{10.3^n+2}\overline{b}_{3^{n}(18m+11)}  &k=3^n(9m+8)
\end{array}\right.$
\item[c)] the differential $d_1:E_1^{2,0}\to E_1^{3,0}$ satisfies 
$$
d_1(\overline{b}_{2k+1})\equiv  
\left\{\begin{array}{lll} 
-u_1^2u^{-4(2k+1)}&\mod (u_1^{4})& 2k+1=6m+1 \\
\omega^2u_1^{4.3^n}u^{-4(2k+1)} &\mod (u_1^{4.3^n+2}) 
& 2k+1=3^n(18m+17)  \\
-\omega^2u_1^{4.3^n}u^{-4(2k+1)} &\mod (u_1^{4.3^n+2}) 
& 2k+1=3^n(18m+5)  \\
0 &&\textrm{otherwise}\ . 
\end{array}\right.
$$
\end{itemize}
\end{prop}

\begin{proof} For $2k+1=3^{n+1}(3m+1)$ resp. $2k+1=3^n(9m+8)$  
Proposition \ref{d_1^{1,0}} shows that $d_1(b_{2k+1})$ 
is divisible by $v_1^{6.3^n+2}$ resp. $v_1^{10.3^n+2}$ and can thus 
be written as $(-1)^nv_1^{6.3^n+2}\overline{b}_{3^{n+1}(6m+1)}$ resp. 
$(-1)^nv_1^{10.3^n+2}\overline{b}_{3^n(18m+11)}$ for unique elements 
$\overline{b}_{3^{n+1}(6m+1)}$ resp. $\overline{b}_{3^n(18m+11)}$ 
which satisfy (a) and (b).

So we still need to define $\overline{b}_{2k+1}$ if $2k+1$ can be written  
as $2k+1=6m+1$ with $m\in \Z$ or $2k+1=3^n(18m+11\pm 6)$ with $n\geq 0$ and 
$m\in\Z$. In those cases we define 
$\overline{b}_{2k+1}:\ =\omega^2u^{-4(2k+1)}$
and note that $-4(2k+1)\equiv 2\ \text{mod}\ (3)$ if $2k+1=6m+1$ and that 
$-4(2k+1)\equiv 7\ \text{mod}\ (9)$ if $2k+1=18m+5$ resp. 
$-4(2k+1)\equiv 4\ \text{mod}\ (9)$ if $2k+1=18m+17$. 
Then (c) holds by Lemma \ref{d_1^{2,0}u}
and Proposition \ref{dualcomplex}.c, at least if we pretend that 
the differential is induced by the map $\partial_1^*:C_0^*\to C_1^*$
after identification of  $C_3$ with $C_0^*$ and of $C_2$ with $C_1^*$ 
via the isomorphisms given by (\ref{dual}). 
In reality the differential is induced by $\partial_1^*$ 
only up to the automorphisms of 
$E_1^{i,0}$, $i=2,3$, induced by the isomorphisms $f_i$ of Proposition 
\ref{dualcomplex} and the isomorphisms of (\ref{dual}). 
However, by Proposition \ref{dualcomplex} these automorphisms 
induce the identity on $\Tor_0^{\Z_3[[S_2^1]]}(\F_3,C_i)$ for $i=2,3$.  
Then Corollary \ref{ti-action} shows that they induce 
automorphisms of $E_1^{i,0}$ as continuous graded $\F_3[v_1]$-modules 
which map $\omega^2u^{-4(2k+1)}$ to itself modulo $(u_1^4)$ 
respectively $\omega^{2k}u^{-12k}$ to itself modulo $(u_1^2)$ and 
part (c) follows. 
\end{proof}

\begin{cor} There is an isomorphism of $\mathbb{F}_3[v_1]$-modules  
$$
E_2^{2,0} \cong \mathop{\prod_{n\geq 0}}_{m \in \mathbb{Z}}
\mathbb{F}_3[v_1]/(v_1^{2.3^{n+1}+2})\{\overline{b}_{3^{n+1}(6m+1)}\}\times 
\mathop{\prod_{n\geq 0}}_{m \in \mathbb{Z}} 
\mathbb{F}_3[v_1]/(v_1^{10.3^n+2})\{\overline{b}_{3^n(18m+11)}\}\ .
\qed  
$$ 
\end{cor}

\remarkstyle{Remark} By inspection one sees again that the infinite 
product is finite in each bidegree and therefore it can also be 
identified with the direct sum.

\begin{prop}\label{d_1^{3,0}} For each integer $k$ there exists an element  
$\overline{\Delta}_k \in E_1^{3,0}$ such that
\begin{itemize}
\item[a)] $\overline{\Delta}_k \equiv \Delta^k \mod (u_1^2)$
\item[b)] The differential $d_1:E_{1}^{2,0}\to E_1^{3,0}$ is given by 
$$
d_1(\overline{b}_{2k+1})\equiv 
\left\{\begin{array}{ll}
(-1)^{m+1}v_1^2\overline{\Delta}_{2m}&  2k+1=6m+1 \\
(-1)^{m+n}v_1^{4\cdot3^n}\overline{\Delta}_{3^n(6m+5)}&2k+1=3^n(18m+17) \\  
(-1)^{m+n+1}v_1^{4\cdot3^n}\overline{\Delta}_{3^n(6m+1)} &2k+1=3^n(18m+5) \\  
0&\textrm{otherwise}\ . 
\end{array}\right.$$
\end{itemize}
\end{prop}

\begin{proof} 
Proposition \ref{d_1^{2,0}} shows that  
$d_1(\overline{b}_{2k+1})$ is divisible by the appropriate 
power of $v_1$. The sign is then determined by comparing the coefficients 
of the ``leading term" $u_1^2u^{-4(2k+1)}$ resp. 
$\omega^2u_1^{4.3^n}u^{-4{2k+1}}$ in $d_1(\overline{b}_{2k+1})$ 
on one hand and in $v_1^2\Delta^{2m}$ resp. 
$v_1^{4.3^n}\Delta^{3^n(6m+3\pm 2)}$  
on the other hand.    
\end{proof}

\begin{cor}\label{E2-30}  
There is an isomorphism of $\mathbb{F}_3[v_1]$-modules 
$$
E_2^{3,0} \cong 
\prod_{m \in \mathbb{Z}}\mathbb{F}_3[v_1]/(v_1^2)
\{\overline{\Delta}_{2m}\}\times 
\prod_{{n\geq 0}, m \in \mathbb{Z}}
\mathbb{F}_3[v_1]/(v_1^{4\cdot 3^n})
\{\overline{\Delta}_{3^n(6m+1)},\overline{\Delta}_{3^n(6m+5)}\} \ .  
\qed
$$
\end{cor} 

\remarkstyle{Remark} By inspection one sees once again that the infinite 
product is finite in each bidegree and it can therefore also be 
identified with the direct sum.   

\medbreak

\subsection{The proof of Proposition \ref{d_1b_k}}\label{middle-differential}

By Corollary \ref{corn2} the differential $E_1^{1,0}\to E_1^{2,0}$ 
is induced by the homomorphism $C_2\to C_1,\ e_2\mapsto n_1'$ where 
$$
n_1':=\frac{1}{16}\sum_{g\in SD_{16}}\chi(g^{-1})g(n_1) \ ,  
$$ 
and by Lemma \ref{lemn2} we can take for $n_1$ any element of the form 
$n_1=\theta e_1$ with 
\begin{equation}\label{theta}  
\theta:=3c(a-c)(a-b)+(e-c)^2(a-c)(a-b)-x(a-b)-y(a-c)(a-b) 
\end{equation}
and $x,y\in IK$ satisfying $e-c^3=x(e-b)+y(e-c)$.  The next result  
gives approximations for $x$ and $y$ which are 
sufficient for our homological calculations.

\begin{lem}\label{approx}  
Let 
$$
\begin{array}{lll}
\widetilde{x}&=&b^{-1}d^{-1}(e-d)-b^{-1}d^{-1}(e-b)b^{-1}c^{-1}(e-c) \\
\widetilde{y}&=&b^{-1}d^{-1}(e-b)b^{-1}c^{-1}(e-b)\ . 
\end{array}
$$ 
Then there exists $z\in IF_{\frac{5}{2}}S_2^1+IK.IF_2S_2^1$ 
such that the following identity holds in $IK$
$$
e-c^3=\widetilde{x}(e-b)+\widetilde{y}(e-c)+z \ .
$$
\end{lem}

\begin{proof} From Lemma \ref{lemcom} and Lemma \ref{abcd} we deduce  
\begin{equation}\nonumber
c^3 \equiv [b^{-1},d^{-1}]  
\mod F_{\frac{5}{2}}S_2^1 \label{c3}. 
\end{equation} 
Thus by using the elementary formulae   
\begin{equation}\label{eqcomm}
1-[X,Y]=XY\big((1-Y^{-1})(1-X^{-1})-(1-X^{-1})(1-Y^{-1})\big)
\end{equation}
\begin{equation}\label{eqcom} 
1-XY=X(1-Y)+(1-X)
\end{equation} 
which hold in any associative algebra we obtain 
\begin{equation}\label{e-c3}
e-c^3 \equiv b^{-1}d^{-1}\big((e-d)(e-b)-(e-b)(e-d)\big) 
\ \textrm{mod}\  IF_{\frac{5}{2}}S_2^1 \ .
\end{equation}
Using Lemma \ref{lemcom} and Lemma \ref{abcd} again we get   
$$
d=[b,c]\equiv [b^{-1},c^{-1}]\ \text{mod}\ F_{2}S_2^1
$$ 
and hence we obtain from (\ref{eqcomm}) 
\begin{equation}\nonumber
e-d \equiv b^{-1}c^{-1}\big((e-c)(e-b)-(e-b)(e-c)\big)
\ \textrm{mod}\  IF_{2}S_2^1 \ .
\end{equation} 
Substituting this into (\ref{e-c3}) gives the result. 
\end{proof}

We will thus be interested in analyzing the action of  
\begin{equation}\label{thetatilde}
3c(a-c)(a-b)+(e-c)^2(a-c)(a-b)-\widetilde{x}(a-b)-\widetilde{y}(a-c)(a-b)
\end{equation}
as well as in the influence of the ``error term" $z$ 
on the elements $\omega^2u^{-4k+2}$.  This analysis will be simplified 
by the following result in which $I$ denotes the ideal $IS_2^1$. 
%\vfil\eject

\begin{lem}\label{ideals} Let $r\geq 1$ be an integer. Then 
we have the following inclusions of left ideals 
\begin{itemize}
\item[a)] $IF_{r}S_2^1\subset 
I^{3^{r}-1}(e-b)+I^{2.(3^{r-1}-1)}(e-c)+3I\subset I^{2.3^{r-1}}+3I$ 
\item[b)] $IF_{\frac{r}{2}}S_2^1
\subset I^{3^{r}-1}(e-b)+I^{3^{r}-2}(e-c)+3I\subset I^{3^r}+3I$. 
\end{itemize}
\end{lem} 

\begin{proof} We note that for every integer $r\geq 0$ 
we have an isomorphism  
$$
IF_{\frac{r}{2}}S_2^1 \cong 
\mathop{\lim_{q>r}}I(F_{\frac{r}{2}}S_2^1/F_{\frac{q}{2}}S_2^1)  
$$ 
and it will therefore be enough to show the corresponding statements for 
the corresponding ideals in the finite quotient groups 
$F_rS_2^1/F_qS_2^1$.   
Next we remark that for every finite $p$-group $G$ the ideal 
$IG$ is a free $\Z_p$-module with basis 
$e-g$, $g\in G-\{e\}$. Therefore 
it is enough to show that 
$$
e-g\in I^{3^{r}-1}(e-b)+I^{2.(3^{r-1}-1)}(e-c)+3I\ \ \text{for\ every} \ 
g\in F_{r}S_2^1/F_qS_2^1
$$ 
resp. 
$$
e-g\in I^{3^{r}-1}(e-b)+I^{3^{r}-2}(e-c)+3I\ \ \text{for\ every} \ 
g\in F_{r+\frac{1}{2}}S_2^1/F_qS_2^1\ .
$$ 
(By abuse of notation we do not distinguish between $g\in S_2^1$ 
and its image in the quotients $S_2^1/F_{q}S_2^1$.)  
Furthermore by (\ref{eqcom}) it is enough to show  this for a system of 
multiplicative generators of $F_{r}S_2^1/F_qS_2^1$ resp. 
$F_{r+\frac{1}{2}}S_2^1/F_qS_2^1$. By Lemma \ref{lemcom} 
the element $c^{3^{q-1}}$ forms a basis of the one dimensional 
$\F_3$-vector space $F_qS_2^1/F_{q+\frac{1}{2}}S_2^1$, 
and $d^{3^{q-1}}$ and $b^{3^q}$ form a basis of the two dimensional 
$\F_3$-vector space $F_{q+\frac{1}{2}}S_2^1/F_{q+1}S_2^1$ and therefore 
it is enough to consider those elements.   

We have $c=[a,b]$ and $d=[b,c]$, and thus (\ref{eqcomm}) 
shows first 
$$
e-c\in I^2
$$ 
and then 
$$
e-d\in I^2(e-b)+I(e-c)\subset I^3\ .
$$ 
Furthermore $e-g^3\equiv (e-g)^3\ \text{mod}\ (3)$ and hence, 
modulo $(3)$, we obtain for any integer $r\geq 1$ 
$$   
\begin{array}{lll}
e-c^{3^{r-1}}&\equiv &(1-c)^{3^{r-1}-1}(e-c)
\subset I^{2.(3^{r-1}-1)}(e-c) \\
e-b^{3^r}&\equiv &(e-b)^{3^{r}-1}(e-b)
\subset I^{3^r-1}(e-b) \\
e-d^{3^{r-1}}&\equiv &(e-d)^{3^{r-1}-1}(e-d)\subset I^{3^r-3}(e-d)
\subset I^{3^r-1}(e-b)+I^{3^r-2}(e-c)\  \\
\end{array}
$$ 
and (a) and (b) follow. 
\end{proof}

\remarkstyle{Remark} The previous lemma can in principle 
also be used to get better explicit approximations of  
the elements $x$ and $y$ of Lemma \ref{approx}, at least modulo $(3)$. 
For this one has to express the element 
$c^3[b^{-1},d^{-1}]^{-1}$ in $F_{\frac{5}{2}}S_2^1/F_qS_2^1$ as explicit 
product of the elements $b^{3^{r-1}}$, $d^{3^{r-2}}$ and $c^{3^{r-1}}$ 
for $q\geq r\geq 3$ and then use (\ref{eqcom}) and the formulae in the 
proof of the previous lemma.   

We will now give a qualitative description of 
the action of powers of $I$ on $(E_2)_*/(3)$.  
The following lemma is an immediate consequence of 
Lemma \ref{t0k} and of the formula 
$g_*(u_1^lu^k)=t_0(g)^{k+2l}u_1^lu^k$ 
(cf. (\ref{u-action}) and Lemma \ref{t_i}.a).    

%\vfil\eject
 
\begin{lem}\label{I-action}\mbox{ } 
\begin{itemize}
\item[a)] $IS_2^1$ sends the $\FF_9[[u_1]]$-submodule of    
$\FF_9[[u_1]]u^k$ generated by  
$u_1^lu^k$ to the submodule generated by $u_1^{l+1}u^k$. 

\item[b)] If $k+2l\equiv 1\ \text{mod}\ (3)$ then 
$IS_2^1$ sends $u_1^lu^k$ to the additive subgroup of $u^k\FF_9[[u_1]]$ 
generated by $u_1^{l+1}u^k$ and the ideal generated by $u_1^{l+3}u^k$.  

\item[c)] If $k+2l\equiv 0\ \text{mod}\ (3)$  
then $IS_2^1$ sends the $\FF_9[[u_1^3]]$-submodule generated 
by $u_1^lu^k$ to the $\FF_9[[u_1^3]]$-submodule generated 
by $u_1^{l+3}u^k$.  \qed 
\end{itemize} 
\end{lem}

\begin{lem}\label{ideal-action} \mbox{ }
\begin{itemize}
\item[a)] Let $k+2l=3m+1$ and $r\geq 1$ be an integer.  
Then %there exist $\a,\b\in \F_9$ such that 
$(IS_2)^r$ sends $u_1^lu^k$ to an element of the form 
$(\a u_1^{l+3(r-1)+1}+\b u_1^{l+3r})u^k$ modulo $(u_1^{l+3r+1})$ 
for suitable $\a,\b\in\F_9$.

\item[b)] Let $k+2l=3m+2$ and $r\geq 2$ be an integer. 
Then $(IS_2)^r$ sends $u_1^lu^k$ to an element of the form 
$(\g u_1^{l+3(r-2)+2}+\delta u_1^{l+3(r-2)+4})u^k$ 
modulo $(u_1^{l+3(r-2)+5})$ for suitable $\g,\d\in\F_9$. 
\end{itemize}
\end{lem}  

\begin{proof} a) This follows by an easy induction 
on $r$ by using the previous lemma. 

b) If $k+2l=3m+2$, the previous lemmma shows that 
$(IS_2)^2$ sends $u_1^lu^{k}$ to $(\l u_1^{l+2}+\mu u_1^4)u^k$ 
modulo $(u_1^{l+5})$ for suitable $\l,\mu\in \F_9$.  
Now the result follows again by an easy induction on $r\geq 2$  
by using once more the previous lemma.   
\end{proof}

The following immediate corollary tells us that for the evaluation of the 
differential we should concentrate on the coefficients of 
$u_1^8$ and $u_1^{10}$ in the case of $u^{3k'+2}$ resp. of $u_1^{10}$ and 
$u_1^{12}$ in the case of $u^{3k'+1}$. It also gives 
us more flexibility for approximating $\theta$. 
  
\begin{cor}\label{8-12-approx}  
Let $k'$ be an integer and $\vartheta\in (3,I^4)$.  Then there exist 
$\a,\b,\g,\d\in \FF_9$ which only depend on $\vartheta$ modulo $(3,I^5)$   
such that we have the following congruences. 
\begin{itemize}
\item[a)] $\vartheta_*(u^{3k'+1})\equiv (\a u_1^{10}+\b u_1^{12})u^{3k'+1}
\ \mod \  (3,u_1^{13})$ for suitable $\a,\b\in \F_9$. 

\item[b)] $\vartheta_*(u^{3k'+2})\equiv (\g u_1^{8}+\d u_1^{10})u^{3k'+2}
\ \ \ \mod \ (3,u_1^{11})\ \ $ for a suitable $\g,\d\in \F_9$. \qed 
\end{itemize} 
\end{cor}  
%\vfil\eject

The next lemma gives a simplified approximation to $\theta$ and hence to $d_1$. 

\begin{prop}\label{d-theta'} Let 
$$
x'=(e-d)-(e-b)(e-c) \textrm{ and } \ y'=(e-b)(e-b)  
$$ 
and define 
$$
\widetilde{\theta}':=\frac{1}{16}\sum_{g \in SD_{16}}\chi(g^{-1})g
\Big(3c(a-c)(a-b)+(e-c)^2(a-c)(a-b)-x'(a-b)-y'(a-c)(a-b)\Big)\ . 
$$ 
Then the differential 
$d_1:E_1^{1,0}\to E_1^{2,0}$ satisfies 
$$
\begin{array}{lllll}
d_1(\omega^2u^{-4(2k+1)})&\equiv &(\widetilde{\theta}')_*(\omega^2u^{-4(2k+1)}) 
&\mod (u_1^{16})&-4(2k+1)=3k'+1  \\ 
d_1(\omega^2u^{-4(2k+1)})&\equiv &(\widetilde{\theta}')_*(\omega^2u^{-4(2k+1)}) 
&\mod (u_1^{12})&-4(2k+1)=3k'+2\  . \\
\end{array}
$$
\end{prop} 

\begin{proof} First we note that Lemma \ref{ideals} shows 
that the element $z$ of Lemma \ref{approx} belongs to 
$(I^8(e-b)+I^7(e-c))+I.(I^8(e-b)+I^4.(e-c))$ 
and therefore does not contribute to the calculation of $d_1$ 
modulo the specified precision. Furthermore $\widetilde{\theta}'$ 
belongs to $(3,I^4)$. Now the last corollary shows that we have 
equality modulo $(u_1^{13})$ resp. $(u_1^{11})$ if we replace $\widetilde{x}$ 
and $\widetilde{y}$ from Lemma \ref{approx} by $x'$ and $y'$, 
and then the following lemma implies equality even modulo 
$(u_1^{16})$ resp. $(u_1^{12})$.     
\end{proof} 

\begin{lem}\label{averaging} Let $k$ and $l\geq 0$ be 
integers and $\l\in \F_9$. Then 
$$
\frac{1}{16}\sum_{g \in SD_{16}}\chi(g^{-1})g(\l u_1^lu^k)=
\left\{\begin{array}{cl}
\frac{1}{2}(\l-\l^3)u_1^lu^k&k+2l\equiv 4 \mod  (8) \\ 
0&\textrm{else} \ . \end{array}\right.
$$
\end{lem}

\begin{proof} By (\ref{SD16-action}) and 
(\ref{Frobenius-action}) we have 
$$
\begin{array}{lll}
\sum_{g \in SD_{16}}\chi(g^{-1})g(\l u_1^lu^k)
&=&\sum_{j=0}^7(-1)^j(\omega^j)_*(\l u_1^lu^k)+
\sum_{j=0}^7(-1)^{j+1}(\omega^j\phi)_*(\l u_1^lu^k) \\
&=&\big(\sum_{j=0}^7(-1)^j\omega^{j(k+2l)}\big)\l u_1^lu^k-
\big(\sum_{j=0}^7(-1)^{j}\omega^{j(k+2l)}\big)\l^3u_1^lu^k \ .
\end{array}
$$
Furthermore 
$\sum_{j=0}^7(-1)^j\omega^{j(k+2l)}=0$ unless 
$k+2l\equiv 4\ \text{mod}\ (8)$ in which case it is equal to $8$. 
The result follows. 
\end{proof}

The previous result tells us how to get the ``leading term" 
in the differential once we know the coefficients  
$\a$, $\b$, $\g$ and $\d$ of Corollary  \ref{8-12-approx} in the case of 
$\vartheta=\widetilde{\theta}:=\sum_{i=1}^4\widetilde{\theta}_i$ with  
$$
\begin{array}{l} 
\widetilde{\theta}_1:=3c(a-c)(a-b)+(e-c)^2(a-c)(a-b)\\
\widetilde{\theta}_2:=-(e-d)(a-b)\\
\widetilde{\theta}_3:=(e-b)(e-c)(a-b)\\
\widetilde{\theta}_4:=-(e-b)(e-b)(a-c)(a-b)  \ ,\\
\end{array}
$$ 
and in fact $\a$ and $\d$ will not even matter. 
The coefficients $\b$ and $\g$ of Corollary \ref{8-12-approx} 
for the action of each $\widetilde{\theta}_i$ are given in the 
following result.

\begin{lem} \mbox{ } Let $k$ be an integer not divisible by $3$. 
For $k=3k'+1$ there are elements $\a_{i,k},\b_{i,k}\in\F_9$, $i=1,2,3,4$, 
and for $k=3k'+2$ there are elements $\g_{i,k},\d_{i,k}\in\F_9$, 
$i=1,2,3,4$, such that we have the following congruences 
$$
(\widetilde{\theta}_i)_*u^k \equiv \left\{ 
\begin{array}{lll} 
(\a_{i,k}u_1^{10}+\b_{i,k}u_1^{12})u^k &\mod (3,u_1^{13})  &k=3k'+1  \\
(\g_{i,k}u_1^8+\d_{i,k}u_1^{10})u^k &\mod (3,u_1^{11})   &k=3k'+2  \ .\\
\end{array} \right.
$$ 
Furthermore we have 
$$
\begin{array}{llrcllrc} 
\b_{1,k}&=&0 & &\g_{1,k}&=&0&\\
\b_{2,k}&=&-(1+\omega^2) & &\g_{2,k}&=&-(1+\omega^2)&\\
\b_{3,k}&=&(k'-1)\omega^2 & &\g_{3,k}&=&(k'+1)\omega^2&\\
\b_{4,k}&=&-\big(k'^2+k'-1+(k'+1)\omega^2\big) & &
\g_{4,k}&=&1+k'-k'^2-k'\omega^2&\ .\\
\end{array} 
$$
\end{lem}

\begin{proof} The first part of the lemma is clear by 
Corollary \ref{8-12-approx}. Furthermore, the case of $\b_{i,k}$ and 
$\g_{i,k}$ follows immediately from Lemma \ref{ideal-action} 
because by Lemma \ref{ideals} we have $\widetilde{\theta}_1\in I^6+3I$. 
The actual evaluation of the other $\b_{i,k}$ and $\g_{i,k}$  
uses Lemma \ref{t0k} and is a lengthy but straightforward calculation 
whose verification we leave to the reader.  
Here we just note that Lemma \ref{I-action} guarantees that for 
this calculation it is enough to know the action of $I$ on $u_1^lu^k$ 
modulo $(u_1^{l+6})$.
\end{proof}

\noindent {\it Proof of Proposition \ref{d_1b_k}.}  
This is an immediate consequence of the previous lemma, of 
Proposition \ref{d-theta'} and Lemma \ref{averaging}.  \qed

\bigbreak

\section{Higher differentials and extensions in the 
algebraic spectral sequence}\label{higher+extensions}

In this section we will give a proof of Proposition \ref{b_0} 
and we will determine the extensions and higher differentials in the 
algebraic spectral sequence (\ref{algss}) in the case of $M_*=(E_2)_*/(3)$.  
This spectral sequence allows for non-trivial $d_2$- and $d_3$-differentials. 
Their evaluation will be reduced to studying the long exact sequence 
in $\Ext_{\Z_3[[\GG_2^1]]}(-,M_*)$ associated to the short exact sequence 
(\ref{sec1}). In fact, we have the following more general lemma. 

\begin{lem}\label{higher-differentials} 
Let $R$ be a ring and $n>0$ be an integer and 
let $M$ be a left $R$-module. Suppose that  
$$
0 \longrightarrow  C_{n+1} 
\stackrel{{\partial}_{n+1}}\longrightarrow C_n
\stackrel{{\partial}_n}\longrightarrow
\ldots \stackrel{{\partial}_1}\longrightarrow C_0 
\stackrel{{\partial}_0}\longrightarrow L\longrightarrow 0   
$$ 
is an exact complex of left $R$-modules such that $C_i$ is 
projective for each 
$0<i<n+1$.   
\begin{itemize}
\item[a)] Then there is a first quadrant cohomological 
spectral sequence $E_r^{*,*}$, $r\geq 1$, converging to $Ext_R^*(L,M)$
with  
$$
E_1^{s,t}=\Ext_R^t(C_s,M) \Rightarrow \Ext_R^{s+t}(L,M)
$$ 
in which $E_1^{s,t}=0$ for $0<s<n+1$ and $t>0$, and $E_1^{s,t}=0$ for 
$t\geq 0$ and $s>n+1$. 
\item[b)] The higher differentials in this spectral sequence can be 
described as follows. Let $N_0$ be the kernel of $\partial_0$ and 
let $j:N_0\to C_0$ denote the resulting inclusion. 
Then there are isomorphisms which are natural in $M$ 
$$
\Ext^t_R(N_0,M) \cong  \left\{ 
\begin{array}{lr} 
\Ker: E_1^{1,0}\stackrel{d_1^{1,0}}\longrightarrow E_1^{2,0} & t=0 \\
E_2^{t+1,0}=E_{t+1}^{t+1,0}        &1\leq t\leq n  \\
E_{n+1}^{n+1,t-n}    &t>n      \\ 
\end{array} \right.
$$ 
such that the homomorphism $\Ext^{t}_R(C_0,M)\to \Ext^{t}_R(N_0,M)$ 
induced by $j$ identifies with 
$d_{t+1}^{0,t}:E_t^{0,t}\to E_{t}^{t+1,0}$ if $1\leq t\leq n$ and with 
$d_{n+1}^{0,t}:E_{n+1}^{0,t}\to E_{n+1}^{n+1,t-n}$ if $t>n$. 
(Note that by (a) these are the only potentially non-trivial differentials 
in this spectral sequence.) 
\end{itemize}
\end{lem}

\begin{proof}
a) The spectral sequence can be obtained as the spectral sequence 
of an exact couple. In fact, if $N_i$ is the kernel of $\partial_i$ 
then we have short exact sequences $0\to N_i\to C_i\to N_{i-1}\to 0$ for 
$0\leq i\leq n$ (with $N_{-1}:\ =L$) and the long exact sequences 
in $\Ext^*_R(-,M)$ combine to give an exact couple from which 
the spectral sequence is derived. Projectivity of the modules $C_i$ 
for $1\leq i\leq n$ gives the vanishing results. 

b) For the first statement we note that $N_0$ admits a projective 
resolution $Q_{\bullet}$ which is obtained from 
splicing the exact complex 
$$
0 \longrightarrow  C_{n+1} 
\stackrel{{\partial}_{n+1}}\longrightarrow C_n 
\stackrel{{\partial}_n}\longrightarrow
\ldots \stackrel{{\partial}_2}\longrightarrow C_1 
\stackrel{{\partial}_1}\longrightarrow N_0\longrightarrow 0   
$$ 
with a projective resolution of $C_{n+1}$. The second statement 
is easily seen by inspection of the higher differentials 
in an exact couple. We leave the details to the reader. 
\end{proof}

\remarkstyle{Remark} The higher differentials can therefore 
be evaluated if we know projective resolutions $Q_{\bullet}$ of $N_0$ 
and $P_{\bullet}$ of $C_0$ as well as a chain map 
$\phi:Q_{\bullet}\to P_{\bullet}$ covering $j$. These data can 
also be assembled in a double complex $T_{\bullet \bullet}$ 
with $T_{\bullet 0}=P_{\bullet}$, 
$T_{\bullet 1}=Q_{\bullet}$, vertical differentials $\delta_P$ and $\delta_Q$ 
and ``horizontal differentials" $(-1)^n\phi_n:Q_n\to P_n$. 
The lemma implies that the filtration of the spectral sequence of this 
double complex agrees (up to reindexing) with that of the spectral sequence 
of the lemma. Hence extension problems in the spectral sequence of the lemma 
can also be studied by using the double complex.  

We apply this lemma and the remark to the algebraic spectral sequence 
associated to the case of the exact complex (\ref{alg}).  
We will make use of explicit projective resolutions $Q_{\bullet}$ 
of $N_0$ and $P_{\bullet}$ of $C_0$ and a suitable chain map 
$\phi:Q_{\bullet}\to P_{\bullet}$ covering $j$. The essential step is given 
in the following elementary result whose proof is left to the reader.

\begin{lem}\label{resol-G24} 
Let $\overline{\chi}$ be the $\Z_3[Q_8]$-module 
whose underlying $\Z_3$-module is $\Z_3$ and on which $t$ 
acts by multiplication by $-1$ and $\psi$ by the identity. 
Then the trivial $\Z_3[G_{24}]$-module $\Z_3$ admits a 
projective resolution $P_{\bullet}'$ of period $4$ 
of the following form  
$$
\stackrel{a^2-a}{\longrightarrow} 
1\uparrow^{G_{24}}_{Q_8}\stackrel{e+a+a^2}{\longrightarrow} 
1\uparrow^{G_{24}}_{Q_8} \stackrel{a^2-a}{\longrightarrow} \overline{\chi} 
\uparrow^{G_{24}}_{Q_8} \stackrel{e+a+a^2}{\longrightarrow} 
\overline{\chi} \uparrow^{G_{24}}_{Q_8} 
\stackrel{a^2-a}{\longrightarrow} 
1\uparrow^{G_{24}}_{Q_8} {\longrightarrow} \mathbb{Z}_3 \ .\qed 
$$ 
\end{lem}

In the sequel we work with the induced projective resolution 
$P_{\bullet}:=\Z_3[[\GG_2^1]]\otimes_{\Z_3[G_{24}]}P_{\bullet}'$ of 
$C_0=\Z_3[[\GG_2^1]]\otimes_{\Z_3[G_{24}]}\Z_3$ and the resolution 
$Q_{\bullet}$ of $N_0$ which is obtained from splicing the exact complex 
obtained from (\ref{alg})
$$
0\to C_3\to C_2\to C_1\to N_0\to 0
$$ 
with a projective resolution of $C_3=C_0$, %
i.e. by splicing it with $P_{\bullet}$. 
The next result records all we need to know about the 
chain map $\phi$.  

\begin{lem}\label{phi} 
There is a chain map $\phi:Q_{\bullet}\to P_{\bullet}$ 
covering the homomorphism $j$ such that 
$$
\phi_0:Q_0=C_1\to 1\uparrow_{Q_8}^{\GG_2^1}=P_0 
$$ 
sends $e_1$ to $(e-\omega)\widetilde{e}_0$.  
(Here the generator $\widetilde{e_0}\in P_0$ 
is given by $e\otimes 1\in P_0$ and we continue to denote 
the generator of $C_1$ introduced in section \ref{partial1}, and 
also given by $e\otimes 1$, by $e_1$.)    
\end{lem} 

\begin{proof} In fact, as in the proof of Corollary \ref{delta1} 
we see that $SD_{16}$ acts on $(e-\omega)\widetilde{e_0}$ via the 
character $\chi$. Hence $\phi_0$ is well defined and 
it is clear that $\phi_0$ covers $j$. \end{proof} 

\noindent {\it Proof of Proposition \ref{b_0}.}  
From its definition it is clear that 
$\a\in H^1(G_{24},M_4)$ is a permanent cycle in the algebraic 
spectral sequence for $M_*=(E_2)_*/(3)$ 
(the restriction of the class with the same name in $H^1(\GG_2^1,M_4)$), 
i.e. there are cochains $c\in \Hom_{\Z_3[[\GG_2^1]]}(P_1,M_4)$ and 
$d\in \Hom_{\Z_3[[\GG_2^1]]}(Q_0,M_4)$ such that $c+d$ is a 
cocycle in the total complex of the double complex 
$\Hom_{\Z_3[[\GG_2^1]]}(T_{\bullet \bullet},M_4)$ and 
such that $c$ represents  
$\a\in \Ext^1_{\Z_3[[\GG_2^1]]}(C_0,M_4)=H^1(G_{24},M_4)$. 
Furthermore, the cocycle $c$ can be obtained   
as the mod-$3$ reduction of a cocycle representing $\delta^0(v_1)$ 
(cf. the discussion in section \ref{e1-term}), i.e. we can take 
$c=\frac{1}{3}({a_*}^2-{a_*})v_1$ and this is known to be of the form  
$\omega u^{-2}\ \text{mod}\ (u_1)$ (cf. the proof of Lemma 1 in \cite{GHM}). 
Then $v_1c$ is a cocycle representing $v_1\a$. However, $v_1\a$ is trivial 
in $\Ext^1_{\Z_3[[\GG_2^1]]}(C_0,M_4)$ by Theorem \ref{shapiro}, 
and hence there exists $h\in \Hom_{\Z_3[[\GG_2^1]]}(P_0,M_4)$ 
such that 
$$
v_1c=\delta_P(h)=a_*^2h-a_*h \ .
$$ 
Because $v_1c$ is equal to $\omega u_1u^{-4}\ \text{mod}\ (u_1^2)$,  
$h$ must have the form $\epsilon u^{-4}\ \text{mod}\ (u_1)$ 
for some unit $\epsilon\in \F_9$. Corollary \ref{u1-action} shows that 
$$
({a_*}^2-a_*)(u^{-4})=-(1+\omega^2)u_1u^{-4}\ \text{mod} \ (u_1)
$$ 
and hence $\epsilon=\omega^2$ by (\ref{omegaeq}). In the double complex 
$\Hom_{\Z_3[[\GG_2^1]]}(T_{\bullet \bullet},M_4)$ 
the cochain $v_1c$ is therefore cohomologous to 
$$
-\phi_0^*(h)\equiv-(e-\omega)_*(\omega^2u^{-4})
\equiv -\omega^2u^{-4}+\omega^{-2}u^{-4}\equiv \omega^2 u^{-4}
\ \text{mod}\ (u_1) 
$$ 
and hence $v_1(c+d)$ is cohomologous 
to $\omega^2 u^{-4}+v_1d\ \text{mod}\ (u_1)$ and this   
implies the proposition. \qed 

Now we turn towards the calculation of higher differentials 
and extensions.     
  
\begin{prop}\label{d2+ext} \mbox{ }
\begin{itemize}
\item[a)] 
The differential $d_2:E_2^{0,1} \to E_2^{2,0}$ in the 
algebraic spectral sequence (\ref{algss}) for $M=(E_2)_*/(3)$ 
is given by  
$$
\begin{array}{lcl}
d_2(\Delta_k \alpha) &= &
\left\{\begin{array}{lcl} 
(-1)^{m+n+1}v_1^{6.3^{n}+1}\overline{b}_{3^{n+1}(6m+1)} &&k=2.3^n(3m+1) \\
(-1)^{m+n}v_1^{10.3^{n+1}+1}\overline{b}_{3^{n+1}(18m+11)}&&k=2.3^n(9m+8) \\
0 &&  \text{otherwise}
\end{array}\right. \\
&\\
d_2(\Delta_k \widetilde{\alpha}) &=& 
\left\{\begin{array}{lcl} 
(-1)^{m}v_1^{11}\overline{b}_{18m+11}&&\ \ \ \ \ \ \ \ \ \ \ \ \ \ \ \ \ \ k=6m+5  \\
0 &&\ \ \ \ \ \ \ \ \ \ \ \ \ \ \ \ \ \ \text{otherwise} \ . 
\end{array}\right. \end{array}
$$
\item[b)] In $H^*(\mathbb{G}_2^1;(E_2)_*/(3))$ we have the following 
relations 
$$
\begin{array}{lll}
v_1\Delta_{2.3^n(9m+2)}\alpha &=&(-1)^{m+1}b_{2.3^{n+1}(9m+2)+1} \\
v_1\Delta_{2.3^n(9m+5)}\alpha &=&(-1)^{m+1}b_{2.3^{n+1}(9m+5)+1} \\
v_1\Delta_{2k+1}\alpha&=&0 \\
v_1\Delta_{6m+1}\widetilde{\alpha}&=&(-1)^mb_{2.(9m+2)+1}  \\
v_1\Delta_{6m+3}\widetilde{\alpha}&=&(-1)^{m+1}b_{2.(9m+5)+1}  \\
v_1\Delta_{2k}\widetilde{\alpha}&=&0\ .  \\
\end{array}
$$
\end{itemize}
\end{prop}

\remarkstyle{Remark} We repeat that after the fact 
one can check that the $d_2$-differential is, up to sign, 
determined by $v_1$-linearity and the principle 
that it is non-trivial whenever it has a chance to be so.

In the proof of the proposition we make repeated use of 
the following lemma.  

\begin{lem}\label{v1-free} 
Let $s>0$, $x\in H^s(G_{24},M_*)$ 
and $c\in \Hom_{\Z_3[[\GG_2^1]]}(P_s,M_*)$ be a representing cocycle.  
Suppose that $v_1^kx=0\in H^s(G_{24},M_*)$ and 
suppose that $h\in \Hom_{\Z_3[[\GG_2^1]]}(P_{s-1},M_*)$ satisfies 
$\delta_P(h)=v_1^kc$. 
\begin{itemize}
\item[a)] Then there are elements 
$d,d'\in \Hom_{\Z_3[[\GG_2^1]]}(Q_{s-1},M_*)$ and 
$d''\in \Hom_{\Z_3[[\GG_2^1]]}(Q_{s},M_*)$ 
such that $\phi_{s-1}^*(h)=d'+v_1^kd$ and $\delta_Q(d')=v_1^kd''$. 
\item[b)] If $d$, $d'$ and $d''$ are as in (a) then 
$$
j^*(x)=(-1)^s[d'']\in \Ext^{s}(N_0,M_*)\ . 
$$ 
\item[c)] If $d$, $d'$ and $d''$ are as in (a) and $d''=0$ 
then 
$$
v_1^kx'=(-1)^s[d']\in H^*(\GG_2^1,M_*)
$$ 
for any $x'$ in $H^*(\GG_2^1,M_*)$ which restricts to $x$. 
In particular, if $d'=0$, then $v_1^kx'=0$. %\in H^*(\GG_2^1,M_*)$.  
\end{itemize} 
\end{lem}

\begin{proof} We can write  
$\phi_{s-1}^*(h)\in \Hom_{\Z_3[[\GG_2^1]]}(Q_{s-1},M_*)\subset M_*$ 
as $u^tf$ for some $t\in \Z$ and $f$ a power series in $u_1$. 
Then we write $f=f_0+u_1^kf_1$ with $f_0$ a polynomial in $u_1$ 
of degree less than $k$ and $f_1$ a power series in $u_1$. If   
we put $d'=u^{t}f_0$ and $d=u^{2k+t}f_1$ then the first part of (a) holds.  
Next we use that the double complex 
$\Hom_{\Z_3[[\GG_2^1]]}(T_{\bullet \bullet},M_*)$ 
is a double complex of torsionfree $\F_3[[v_1]]$-modules and 
$\F_3[[v_1]]$-linear differentials. Then we get   
$$
\delta_Q(d')+v_1^k\delta_Q(d)=\delta_Q(d'+v_1^kd)=
\delta_Q(\phi_{s-1}^*(h))=\phi_s^*(\delta_P(h))=\phi_s^*(v_1^kc)
\equiv 0\ \text{mod}\ (v_1^k)  
$$  
and 
$$
\begin{array}{lll}
v_1^kj^*(c)&=&v_1^k(-1)^s\phi_s^*(c)=(-1)^s\phi_s^*(v_1^kc)
=(-1)^s\phi_s^*(\delta_P(h))\\
&=&(-1)^s\delta_Q\phi_{s-1}^*(h)
=(-1)^s\delta_Q(d'+v_1^kd)=(-1)^sv_1^kd''+(-1)^sv_1^k\delta_Q(d)\ .  
\end{array}
$$
and thus the second part of (a) and (b) follow. 

If $d''=0$ then the equations  
$$
v_1^kc=\delta_P(h), \ \ \   
(-1)^{s-1}v_1^kd+(-1)^{s-1}d'=(-1)^{s-1}\phi^*_{s-1}(h)
$$  
show that $v_1^k(c+(-1)^{s-1}d)$ and $(-1)^sd'$ 
are cohomologous in the double complex. Furthermore, $d'$ is a cycle by 
assumption and hence $v_1^k(c+(-1)^{s-1}d)$ is a cycle. Because the 
double complex is $v_1$-torsion free, $c+(-1)^{s-1}d$ is a cycle as well 
and (c) follows. 
\end{proof}

\noindent {\it Proof of Proposition \ref{d2+ext}.} 
We begin with the case of $\Delta_k\a$. As in the proof of 
Proposition \ref{b_0} let 
$c\in \Hom_{\Z_3[[\GG_2^1]]}(P_0,M_4)$ be a cocycle representing 
$\a\in \Ext^1_{\Z_3[[\GG_2^1]]}(C_0,M_4)$. 
Then $\Delta_k\a$ is represented by the cocycle $\Delta_kc$.  
As $\Delta_k$ is $G_{24}$-invariant we have 
$$
\Delta_{k}v_1c=\delta_P(\Delta_kh)   
$$ 
where $h$ is as in the proof of Proposition \ref{b_0} above. In particular  
$h\equiv \omega^2u^{-4}$ and hence  
$$
\phi_0^*(\Delta_kh)\equiv (e-\omega)_*(\omega^{2k+2}u^{-12k-4})
\equiv (1-(-1)^{-3k-1})\omega^{2k+2}u^{-12k-4}\ \text{mod}\ (u_1^4)\ .
$$
(We note that the congruence is modulo $(u_1^4)$ by (\ref{TMF(2)})!) 
In particular, if $k$ is odd we see that 
$$
\phi_0^*(\Delta_kh)=v_1^4z
$$ 
for some $z\in \Hom_{\Z_3[[\GG_2^1]]}(Q_0,M_*)$ and  
if $k$ is even, $k=2.3^n(3m\pm 1)$, we get 
$$
\phi_0^*(\Delta_kh)=(-1)^mb_{3k+1}+v_1^4z
$$ 
for some $z\in \Hom_{\Z_3[[\GG_2^1]]}(Q_0,M_*)$. Then Lemma 
\ref{v1-free} and Theorem \ref{d1} give easily the differentials 
and extensions on all elements $\Delta_k\a$ for $k\neq 0$. 

In the case of $\Delta_k\widetilde{\a}$ the 
definition of $\widetilde{\a}$ (cf. section \ref{e1-term}) 
shows that $\widetilde{\a}$ can be represented by a cocycle 
$\widetilde{c}\in \Hom_{\Z_3[[\GG_2^1]]}(P_1,M_{12})$ such that  
$v_1\widetilde{c}=\delta_P(v_2)$. Because $\Delta_k$ is $G_{24}$-invariant 
we have $\delta_P(\Delta_kv_2)=\Delta_k\delta_P(v_2)$. 
Furthermore, 
$$
\Delta_kv_2\equiv \omega^{2k}u^{-12k-8} 
\ \text{mod}\ (u_1^2)
$$ 
and thus 
$$
\phi_0^*(\Delta_kv_2)\equiv(e-\omega)_*(\omega^{2k}u^{-12k-8})
\equiv (1-(-1)^{-3k-2})\omega^{2k}u^{-12k-8}\ \text{mod}\ (u_1^4)\ .
$$
(Again we note that the congruence is modulo $(u_1^4)$ by (\ref{TMF(2)})!) 
In particular, if $k$ is even we deduce that 
$$
\phi_0^*(\Delta_kv_2)=v_1^4z
$$ 
for some $z\in \Hom_{\Z_3[[\GG_2^1]]}(Q_0,M_*)$ and 
for $k=6m+j$ with $j\in \{1,3,5\}$ we have 
$$
\phi_0^*(\Delta_kv_2)=(-1)^{m+1}\omega^{2j-2}b_{18m+3j+2}+v_1^4z
$$
for some $z\in \Hom_{\Z_3[[\GG_2^1]]}(Q_0,M_*)$  
and as before Lemma \ref{v1-free} and Theorem \ref{d1}   
give easily the differentials 
and extensions on all elements $\Delta_k\widetilde{\a}$. \qed 
\medbreak 

The following result together with 
Proposition \ref{b_0} and Proposition \ref{d2+ext} finishes off the 
proof of Proposition \ref{d2+d3} and Proposition \ref{extensions}.  

\begin{prop} \mbox{ } 
\begin{itemize} 
\item[a)] 
For $i\geq 3$ the differentials $d_i$ in the 
algebraic spectral sequence (\ref{algss}) for the mod-$3$ Moore spectrum 
are all trivial.  
\item[b)] For each integer $k\in \Z$ and $l>0$ we have 
$v_1^2\Delta_k \beta^l=v_1\Delta_k\beta^l\alpha=
v_1\Delta_k\beta^l\widetilde{\alpha}=0$ 
in $H^*(\mathbb{G}_2^1;(E_2)_*/(3))$. 
\end{itemize}
\end{prop}

\begin{proof} The differential is linear with respect to 
the natural $\Ext_{\Z_3[[\GG_2^1]]}(\Z_3,\F_3)=H^*(\GG_2^1,\F_3)$-module 
structure on its target and source. Hence it is enough show 
that the classes $\Delta_k\b\a$, $\Delta_k\b\widetilde{\a}$ and $\Delta_k\b$ 
are $d_3$-cycles and to prove (b) in the case $l=1$. 
In both cases we use Lemma \ref{v1-free} once again. 

We begin with the case of $\Delta^k\b$. Let $c_1$ be a cocycle 
in $\Hom_{\Z_3[[\GG_2^1]]}(P_2,M_{12})$ representing $\b$. 
As $v_1^2\beta=0$ in $H^2(G_{24},M_{20})$ 
there exists $h_1\in \Hom_{\Z_3[[\GG_2^1]]}(P_1,M_{20})$
such that 
$$
v_1^2c_1=\delta_P(h_1)=(e+a+a^2)_*h_1 \ . 
$$
Next we use that 
$Q_1=C_2=\mathbb{Z}_3[[\mathbb{G}_2^1]]\otimes_{\Z_3[SD_{16}]}\chi$.  
Hence we have (cf. (\ref{TMF(2)}))    
$$
\Hom_{\mathbb{G}_2^1}(Q_1,M_*)  
\cong \omega^2u^4\mathbb{F}_3[[u_1^4]][v_1,u^{\pm 8}]
$$ 
and by degree reasons we need to have $\phi_1^*(h_1)=v_1^3z$
for some $z\in \Hom_{\mathbb{Z}_3[[{\GG}_2^1]]}(Q_1,M_{8})$, and 
then Lemma \ref{v1-free} shows that $\b$ is not only the 
restriction of a permanent cycle (which we knew anyway), 
but also that $v_1^2\b=0$ in $H^2(\GG_2^1,M_{20})$. 
Furthermore, as $\Delta_k$ is $G_{24}$-invariant we have 
$\delta_P(\Delta_kh_1)=\Delta_kv_1^2c_1$
and in order to apply Lemma \ref{v1-free} in the case of $\Delta_k\b$ 
we need to understand $\phi_1^*(\Delta_kh_1)$. For this we 
recall that $Q_1=C_2$ is a free $\Z_3[[S_2^1]]$-module generated by $e_2$.  
Hence we can write $\phi_1(\widetilde{e_1})=xe_2$ for (a unique)  
$x\in \Z_3[[S_2^1]]$ and thus $\phi_1^*(\Delta_kh_1)=x_*(\Delta_kh_1)$. 
Furthermore, because $\Delta \equiv \omega^2 u^{-12} \mod (u_1^2)$ 
we have for any $g\in S_2^1$  
$$
g_*(\Delta_kh_1)=g_*(\Delta_k)g_*(h_1)\equiv  
t_0(g)^{12k}\Delta_kg_*(h_1)\equiv \Delta_kg_*(h_1) \mod (u_1^2) \ , 
$$   
thus  
$$
\phi_1^*(\Delta_kh_1)=x_*(\Delta_kh_1) 
\equiv \Delta_kx_*(h_1)=\Delta_k \phi_1^*(h_1) \equiv 0 \mod (u_1^2)
$$
and then Lemma \ref{v1-free} shows that $\Delta_k\beta$  
is a permanent cycle and in $H^2(\mathbb{G}_2^1;M)$ we have 
$v_1^2\Delta_k\beta=0$. 

In the case of $\Delta_k\beta\alpha$ we choose a representing cocycle 
$c_2\in \Hom_{\Z_3[[\GG_2^1]]}(P_3,M_{16})$ for $\b\a$. Because $v_1\b\a=0$ in 
$H^1(G_{24},M_{20})$ there exists 
$h_2\in \Hom_{\Z_3[[\GG_2^1]]}(P_2,M_{20})$ such that 
$v_1c_2=\delta_P(h_2)$.  
Then 
$$
\phi_2^*(h_2)\in \Hom_{\mathbb{Z}_3[[{\GG}_2^1]]}(Q_2,M_{20})
=\Hom_{\mathbb{Z}_3[[{\GG}_2^1]]}(P_0,M_{20})\cong (M_{20})^{Q_8}
$$
is divisible by $v_1$ by degree reasons 
(cf. (\ref{SD16-action}) or Remark 3.12.3 of \cite{GHMR})
% {\bf (avoid this reference!!)}) 
and then Lemma \ref{v1-free} shows that $\Delta_k \beta\a$ 
is a permanent cycle and in $H^2(\mathbb{G}_2^1,M_*)$ we have 
$v_1\Delta_k\beta\a=0$. 
The case of $v_1\beta\widetilde{\alpha}$ is completely analogous. 
\end{proof} 

\medbreak 

\begin{prop} If $x\in H^1(\GG_2^1,M_*)$ is represented in 
$E_{\infty}^{1,0,*}$ then $\b x=0$ in $H^3(\GG_2^1,M_*)$.  
\end{prop}

\begin{proof} It is enough to show $\beta b_1=0$ and 
$\beta b_{2.3^n(3m-1)+1}=0$ whenever $m\not\equiv 0 \mod (3)$.  
By Proposition \ref{b_0} we have $b_1=v_1\a$ and thus 
$\b b_1=v_1\b\a=0$ by the previous proposition. 
Similarly, by Proposition \ref{d2+ext} we have  
$b_{2.3^{n+1}(3m-1)+1}=\pm v_1\Delta_{2.3^n(3m-1)}\alpha$ 
and $b_{2(3m-1)+1}=\pm v_1\Delta_{2m-1}\widetilde{\alpha}$ and using  
the previous proposition once more shows $\b b_{2.3^{n}(3m-1)+1}=0$. 
\end{proof} 

The following result finishes off the proof of Proposition \ref{extensions}.

\begin{prop}\label{beta-extensions}  
The following relations hold in $H^*(\mathbb{G}_2^1;M_*)$   
$$ 
\begin{array}{lll}
\beta \overline{b}_{3^{n+1}(6m+1)}&=
&\pm \overline{\Delta}_{3^n(6m+1)}\widetilde{\alpha} \\
\beta \overline{b}_{3^{n+1}(18m+11)}&=
&\pm \overline{\Delta}_{3^n(18m+11)}\widetilde{\alpha}\\
\beta \overline{b}_{18m+11}&=&\pm 
\overline{\Delta}_{6m+4}\a  \ .   \\
\end{array}
$$
\end{prop}

\begin{proof} First we observe that $\widetilde{\a}\in H^*(G_{24},(E_2)_*/(3))$ 
is non-divisible by $v_1$ and this implies %by degree reasons 
that the mod-$u_1$ reduction homomorphism 
$H^*(G_{24},(E_2)_*/(3))\to H^*(G_{24},(E_2)_*/(3,u_1))$ must 
send $\widetilde{\a}$ to $\pm \omega^2 u^{-4}\a$ (cf. Theorem \ref{E_1V(1)}). 
Likewise, this map must send $\Delta_{2k+1}$ to $\pm \omega^2v_2^{3k+1}u^{-4}$, 
and it clearly sends $\a$ to $\a$.  
Thus the proposition follows from Theorem \ref{E_1V(1)}.c and naturality.   
\end{proof} 

\bigbreak 

\section{Passing from $\GG_2^1$ to $\GG_2$}\label{passing} 

Theorem \ref{H*G2} is a simple instance of a K\"unneth isomorphism: 
in fact, if we have an isomorphism of profinite $3$-groups 
$G=F\times \Z_3$, and if $\Z_3$ acts trivially on a $3$-profinite 
module $M$, then the exterior product in cohomology induces an isomorphism 
\begin{equation}\label{Kuenneth}
H^*(F,M)\otimes_{\Z_3} \Lambda_{\Z_3}(\zeta)\cong 
H^*(F,M)\otimes_{\Z_3} H^*(\Z_3;\Z_3)\to H^*(\GG_2,M)\ .
\end{equation} 
In particular this holds if $G=\GG_2$, $F=\GG_2^1$ and 
$M=(E_2)_*/(3)$ or $M=(E_2)_*/(3,u_1)$. 
We will need to know how the Bockstein homomorphisms 
$\delta^1$ and $\delta^0$ associated to the exact sequences 
(\ref{delta^0}) and (\ref{delta^1}) 
behave with respect to these isomorphisms. 

The proof of the following lemma is a 
straightforward exercise with the double complex obtained 
from tensoring a projective resolution of the trivial 
$\Z_3[[F]]$-module $\Z_3$ with the projective resolution of the trivial 
$\Z_3[[\Z_3]]$-module $\Z_3$ given by 
$$
0\to \Z_3[[T]]\buildrel{T}\over\longra\Z_3[[T]]\to \Z_3\to 0 
$$  
and is left to the reader. (Here we have identified $\Z_3[[T]]$ with 
$\Z_3[[\Z_3]]$ via the continuous isomorphism which sends $T$ to $t-e$ if 
$t$ is a topological generator of $\Z_3$.) 

\begin{lem} Let $G= F\times \Z_3$ and $H$ be a closed 
subgroup of $G$, $0\to M_1\to M_2\to M_3\to 0$ be a 
short exact sequence of continuous $G$-modules and 
$\delta_H$ be the associated Bockstein in $H^*(H,-)$. 
If $\Z_3$ acts trivially on $M_1$ and $M_3$, and if we identify 
$H^*(G,M_i)$ with $H^*(F,M_i)\oplus H^{*-1}(F,M_i)\zeta$ for $i=1,3$  
via (\ref{Kuenneth}), $M_3$ with $H^0(\Z_3,M_3)$ and 
$M_1$ with $H^1(\Z_3,M_1)$, then 
$$
\begin{array}{llllrll} 
\delta_G(x)&=\ &\delta_{F}(x)&+&(-1)^nH^n(F,\delta_{\Z_3})(x)\zeta&
&\ \ x\in H^n(F,M_3)\\ 
\delta_G(y\zeta )&=\ &&&\delta_{F}(y)\zeta&
&\ \ y\in H^{n-1}(F,M_3)\ . \qed 
\end{array}
$$ 
\end{lem}
%\vfil\eject
 
\begin{cor}\label{Bockstein} \mbox{ } 
Let $M=(E_2)_{4k}/(3)$ resp. $M=(E_2)_{4k}/(3,u_1)$ 
and identify $H^*(\GG_2,M)$ with 
$H^*(\GG_2^1,M)\oplus H^{*-1}(\GG_2^1,M)\zeta$ 
via (\ref{Kuenneth}).  
\begin{itemize}
\item[a)] For $x\in H^n(\GG_2^1,(E_2)_{4k}/(3,u_1))$ we have  
$\delta^1_{\GG_2}(x)=\delta^1_{\GG_2^1}(x)$.  
\item[b)] For $x\in H^n(\GG_2^1,(E_2)_{4k}/(3))$ we have 
$\delta^0_{\GG_2}(x)=\delta^0_{\GG_2^1}(x)+(-1)^nk x\zeta$.  
\end{itemize}
In particular, if we define 
$$
\begin{array}{lll}
\alpha(\GG_2):=\delta^0_{\GG_2}(v_1) &  
\widetilde{\alpha}(\GG_2):=\delta^0_{\GG_2}(v_2) &
\b(\GG_2):=\delta^0_{\GG_2}\delta^1_{\GG_2}(v_2) \\
\alpha(\GG_2^1):=\delta^0_{\GG_2^1}(v_1) & 
\widetilde{\alpha}(\GG_2^1):=\delta^0_{\GG_2^1}(v_2) &  
\b(\GG_2^1):=\delta^0_{\GG_2^1}\delta^1_{\GG_2^1}(v_2) \\
\end{array}
$$ 
then  
$$
\alpha(\GG_2)=\alpha(\GG_2^1)-v_1\zeta, \ \ 
\widetilde{\alpha}(\GG_2)=\widetilde{\alpha}(\GG_2^1), \ \ 
\b(\GG_2)=\b(\GG_2^1) \ . 
$$
\end{cor}

\begin{proof} The central factor $\Z_3$ of $\GG_2$ is generated by the element 
$t:=1+3\in \Z_3^{\times}$. In this case $t$ acts trivially on $(E_2)_{0}$ 
and on $u$ via $t_*(u)=4u$. Therefore $t$ acts trivially on $M$ and 
via multiplication by $(1+3)^{-2k}=1-6k$ on $(E_2)_{4k}/(9)$. Hence 
$\delta_{\Z_3}$ is given by multiplication by $-2k\equiv k\ \text{mod}\ (3)$ 
and the result follows.   
\end{proof} 

\bigbreak 

\section{The Adams-Novikov spectral sequence for $L_{K(2)}V(0)$}
\label{ANSS}

As before we use the $E_1$-term of the algebraic spectral sequence 
(\ref{algss}) for $M=(E_2)_*V(0)$ to represent elements in the 
$E_2$-term of the Adams-Novikov spectral sequence for $\pi_*(L_{K(2)}V(0))$. 
However, unlike in the introduction, we do not always insist 
on writing elements in terms of the $\FF_3[\b,v_1]\otimes\Lambda(\zeta)$-module  
generators of Theorem \ref{E-infinity}. This allows for 
simplified statements of Lemma \ref{d5}, Corollary \ref{D5}, Lemma \ref{d9}  
and Corollary \ref{D9} below. 

The $E_2$-term satisfies $E_2^{s,t}=0$ unless $t\equiv 0\ \mod (4)$, 
hence its differentials $d_r$ are trivial if 
$r\not\equiv 1\ \text{mod}\ (4)$. Furthermore, 
the differentials are linear with respect to 
$\F_3[\b,v_1]\otimes\Lambda(\zeta)$, and the existence of the resolution 
(\ref{hom}) of \cite{GHMR} gives further restrictions 
on the behaviour of the Adams-Novikov differentials.  
In fact, they have to preserve the filtration on its $E_2$-term 
given by the algebraic resolution for $\GG_2$, %(\ref{algss}), 
and modulo this filtration the differentials are easily determined 
by the differentials in the Adams-Novikov spectral sequence for 
$E_2^{hG_{24}}\wedge V(0)$  (cf. Theorem \ref{EO2V(0)}). However, 
to settle the ambiguities coming from potential contributions of 
smaller filtration terms we need to fall back on knowledge of the 
differentials in $\pi_*(L_{K(2)}V(1))$ (cf. Theorem \ref{LK(2)V(1)}). 

The following Lemma records some immediate consequences 
of the knowledge of the $d_5$-differential in the Adams-Novikov 
spectral sequence for $\pi_*(L_{K(2)}V(1))$. 

\begin{lem}\label{d5} \mbox{ } The following identities hold in  
$H^*(\GG_2,(E_2)_*/(3))\cong E_2^{*,*}\cong E_5^{*,*}$ 
of the Adams-Novikov spectral sequence 
for $\pi_*(L_{K(2)}V(0))$. 
\begin{itemize} 
\item[a)] Let $k\not\equiv 0\mod(3)$. Then there are constants 
$\e_k\in \{\pm 1\}$ such that 
$$
\begin{array}{lll}
d_5(\Delta_{k}\widetilde{\a}\b)&=
&\ \e_k \Delta_{k-1}\b^{4}v_1\\
d_5(\Delta_{k}\b^2)&=
&\ \e_k \Delta_{k-1}\a(\GG_2)\b^{4}  \\
d_5(\Sigma^{48}\overline{\Delta}_{k}\widetilde{\a})&=
&\ \e_k\Sigma^{48}\overline{\Delta}_{k-1}\b^3v_1 \\
d_5(\Sigma^{48}\overline{\Delta}_{k}\b)&=
&\ \e_k\Sigma^{48}\overline{\Delta}_{k-1}\a(\GG_2)\b^3 \ . \\
\end{array} 
$$ 
\item[b)] Let $k\equiv 0\mod(3)$. Then 
$$
\begin{array}{lll}
d_5(\Delta_{k}\widetilde{\a}\b)&=&0\\
d_5(\Delta_{k}\b^2)&=&0\\
d_5(\Sigma^{48}\overline{\Delta}_{k}\widetilde{\a})&=&0\\
d_5(\Sigma^{48}\overline{\Delta}_{k}\b^2)&=&0\  .\\
\end{array} 
$$ 
\end{itemize} 
\end{lem}

%Of course, notation like  
%$\Delta_{k-1}\a(\GG_2)\b^{4}$ is abusive; strictly speaking 
%the notation is only meaningful in the $E_1$-term of the algebraic 
%spectral sequence for $\GG_2$ in which 
%$\a(\GG_2)$ is represented by $\a(G_{24})-v_1\zeta$. 
\remarkstyle{Remark}  By identifying vector space generators 
in the appropriate bidegrees it is easy to see that  
there are unique elements $\l_i,\mu_i,\nu_i\in \FF_3$ such that  
$$
\begin{array}{lll}
d_5(\Delta_{k}\widetilde{\a}\b)&=&\ \l_1\Delta_{k-1}\b^{4}v_1+
\mu_1\Sigma^{48}\overline{\Delta}_{k-2}\a(\GG_2^1)\b^{2} 
+\nu_1 \Sigma^{48}\overline{\Delta}_{k-2}\b^{2}v_1\zeta\\

d_5(\Delta_{k}\b^2)&=&\ \l_2\Delta_{k-1}\a(\GG_2^1)\b^{4} +
\mu_2\Delta_{k-1}\b^{4}v_1\zeta 
+\nu_2\Sigma^{48}\overline{\Delta}_{k-2}\a(\GG_2^1)\b^{2}\zeta\\

d_5(\Sigma^{48}\overline{\Delta}_{k}\widetilde{\a})&=
&\ \l_3\Sigma^{48}\overline{\Delta}_{k-1}\b^{3}v_1 \\

d_5(\Sigma^{48}\overline{\Delta}_{k}\b)&=
&\ \l_4\Sigma^{48}\overline{\Delta}_{k-1}\a(\GG_2^1)\b^{3} +
\mu_4\Sigma^{48}\overline{\Delta}_{k-1}\b^{3}v_1\zeta \ .\\
\end{array} 
$$ 
Naturality and the geometric boundary theorem (cf. Theorem 2.3.4 
of \cite{Rav}) applied to the reso{\-}lution (\ref{hom}) 
allow to determine the values of the $\l_i$, i.e. to 
show the lemma modulo elements of lower filtration. The Lemma 
confirms these values and also determines $\mu_i$ and $\nu_i$;  
formally $\mu_i$ and $\nu_i$ can be deduced by simply replacing 
in the differentials for $E_2^{hG_{24}}\wedge V(0)$ the 
elements $\a$ by $\a(\GG_2)$. 

%In fact, the $\nu_i$ are all trivial and the $\mu_i$ 
%determines the values of $\l_i,\mu_i,\nu_i$ (up to a common sign).  

\begin{proof} We start with $\Delta_k\widetilde{\a}\b$. 
This is in the kernel of $v_1$-multiplication and must therefore 
(after $4$-fold suspension)  be in the image of the 
Bockstein homomorphism $\delta^1_{\GG_2}$ in $H^*(\GG_2,-)$ 
and $\delta^1_{\GG_2^1}$ in $H^*(\GG_2^1,-)$ associated to the short exact sequence 
(\ref{delta^1}).  
Similarly with $\Delta_k\b v_1$. 
By Theorem \ref{E_1V(1)} and by degree reasons we must therefore have 
\begin{equation}\nonumber
\delta^1_{\GG_2^1}((\omega^2u^{-4})^{3k+2}\b)= 
\pm  \Sigma^4\Delta_{k}\widetilde{\a}\b, \ \ \ 
\delta^1_{\GG_2^1}((\omega^2u^{-4})^{3k+2}\a(\GG_2^1)) = 
\pm  \Sigma^4\Delta_{k}\b v_1  
\end{equation} 
and, by Corollary \ref{Bockstein}, we even get 
\begin{equation}\label{various-delta1}
\delta^1_{\GG_2}((\omega^2u^{-4})^{3k+2}\b)= 
\pm  \Sigma^4\Delta_{k}\widetilde{\a}\b, \ \ \ 
\delta^1_{\GG_2}((\omega^2u^{-4})^{3k+2}\a(\GG_2^1)) = 
\pm  \Sigma^4\Delta_{k}\b v_1  \ . 
\end{equation} 
and by $\b$-linearity  
$\delta^1((\omega^2u^{-4})^{3k+2}\a\b^3) = 
\pm  \Sigma^4\Delta_{k}\b^4 v_1$. 
Then the geometric boundary theorem %(cf. Theorem 2.3.4 of \cite{Rav})  
and Theorem \ref{LK(2)V(1)}.b show that $\Delta_{6k}\widetilde{\a}\b$ 
and $\Delta_{6k+3}\widetilde{\a}\b$ are permanent cycles and the value 
of the differential in the other cases is as stated (with a suitable 
constant $\e_k$).  

The case of $\Sigma^{48}\overline{\Delta}_{k}\widetilde{\a}\b$ 
can be treated similarly but in this case it would also suffice 
to use the strategy described in the remark above. The sign is 
clearly the same as in the previous case. 

The remaining two cases are deduced from what has already been 
established by using the Bockstein $\delta^0_{\GG_2}$
in $H^*(\GG_2,-)$ associated to the short exact sequence (\ref{delta^0}) 
and the geometric boundary theorem.   
By Corollary \ref{Bockstein} we have 
$$
\begin{array}{ll} 
\delta^0_{\GG_2}(\Delta_{k}\widetilde{\a}\b)=\Delta_{k}\b^{2}, &  
\delta^0_{\GG_2}(\Sigma^{48}\overline{\Delta}_{k}\widetilde{\a})= 
\Sigma^{48}\overline{\Delta}_{k}\b\\
\delta^0_{\GG_2}(\Delta_{k}\b^4 v_1) =\Delta_{k}\a(\GG_2)\b^4, &
\delta^0_{\GG_2}(\Sigma^{48}\overline{\Delta}_{k}\b^3 v_1) = 
\Sigma^{48}\overline{\Delta}_{k}\a(\GG_2)\b^3\ . 
\end{array}
$$  
In fact, by the Corollary we only need to determine 
$\delta^0_{\GG_2^1}$ and this is straightforward in the case of 
$\Sigma^{48}\overline{\Delta}_{k}\b^3 v_1$ and 
$\Sigma^{48}\overline{\Delta}_{k}\widetilde{\a}\b^3$. In the other two cases 
it is straightforward modulo terms of lower filtration and by degree reasons 
there are no error terms of lower filtration.  
\end{proof}

\begin{cor}\label{D5} The $d_5$-differential in the 
Adams-Novikov spectral sequence for $L_{K(2)}V(0)$ 
is linear with respect to $\FF_3[\b,v_1]\otimes\Lambda(\zeta)$   
and is trivial on all $\FF_3[\b,v_1]\otimes\Lambda(\zeta)$-module 
generators of $H^*(\GG_2^1,(E_2)_*/(3))$ of Theorem \ref{H*G21} 
except the following: 

$$
\begin{array}{llll}
d_5(\Delta_m\beta)       &=&\pm \Delta_{m-1}\a(\GG_2)\b^3 & m\not\equiv 0\mod (3) \\
d_5(\Delta_{2m}\widetilde{\alpha}) &= & 
\pm \Delta_{2m-1}\b^3v_1& m\not\equiv 0\mod (3)\\
d_5(\Delta_{6m+1}\widetilde{\alpha})&=&\pm \Delta_{6m}\b^3v_1&  \\
d_5(\Delta_{6m+5}\widetilde{\alpha}\b)&= 
&\pm\Delta_{6m+4}\b^4{v_1} &         \\
&\\
%\end{array}
%$$
%$$
%\begin{array}{lllll}
d_5(\overline{b}_{3^{n+1}(6m+1)})&= 
&\pm\Sigma^{48}\overline{\Delta}_{3^n(6m+1)-3}\b^2v_1 
& n\geq 0 \\ 
d_5(\overline{b}_{3^{n}(6m+5)})&=
&\pm\Sigma^{48}\overline{\Delta}_{3^{n-1}(6m+5)-3}\b^2v_1 
& m\equiv 1\mod(3),\ n\geq 2\\ 
&\\
%\end{array}
%$$
%$$
%\begin{array}{lllll}
d_5(\Sigma^{48}\overline{\Delta}_{2m})&=
&\pm\Sigma^{48}\overline{\Delta}_{2m-1}\a(\GG_2)\b^2    &m\not\equiv 0\mod (3)\\
d_5(\Sigma^{48}\overline{\Delta}_{3^n(6m+1)-2})&=
&\pm\Sigma^{48}\overline{\Delta}_{3^n(6m+1)-3}\a(\GG_2)\b^2&n\geq 0\\
d_5(\Sigma^{48}\overline{\Delta}_{3^n(6m+5)-2})&= 
&\pm\Sigma^{48}\overline{\Delta}_{3^n(6m+5)-3}\a(\GG_2)\b^2&n\geq 1   \\
d_5(\Sigma^{48}\overline{\Delta}_{2m}\widetilde{\alpha})            &=&
\pm \Sigma^{48}\overline{\Delta}_{2m-1}\b^3v_1           &m\not\equiv 0\mod (3)\\
d_5(\Sigma^{48}\overline{\Delta}_{3^n(6m+5)-2}\widetilde{\alpha}) &=&
\pm\Sigma^{48}\overline{\Delta}_{3^{n}(6m+5)-3}\b^3v_1,&  m\not\equiv 1\mod(3), 
\ n\geq 1\ . \\
\end{array}
$$
\end{cor} 

\begin{proof}  
Linearity with respect to $\FF_3[\b,v_1]\otimes\Lambda(\zeta)$ 
is clear. The rest is an immediate consequence of the previous lemma  
together with sparseness, Proposition \ref{extensions}, Theorem \ref{H*G21}, 
and the fact that $\b$-multiplication is injective above 
cohomological degree $3$.  
\end{proof} 

\begin{cor}\label{E6}  
$E_6$ is the quotient of the direct sum of cyclic 
$\FF_3[\b,v_1]\otimes\Lambda(\zeta)$-modules with 
the following generators and annihilating ideals  
$$
\begin{array}{lll}
1=\Delta_0          &      & (\b v_1^2,\b^3v_1) \\
&\\
\Delta_{m}\beta& 0\neq m\equiv 0\mod(3) &(v_1^2,\b^2v_1) \\
\Delta_{6m+1}\b v_1  &                      &(v_1,\b^2) \\
\Delta_{6m+4}\b v_1 &                      &(v_1,\b^3) \\
\Delta_{m}\b v_1&    m\equiv 2\mod(3)     &(v_1) \\ 
%\end{array}
%$$     
%$$            
%\begin{array}{ll}
&\\
\alpha(\GG_2^1)                &          & (\b v_1,\b^3) \\
\Delta_{2m+1}\alpha(\GG_2^1)& m\not\equiv 2\mod(3)       &(v_1,\b^3) \\
\Delta_{2m+1}\alpha(\GG_2^1)& m\equiv 2\mod(3)            &(v_1) \\
\Delta_{2.3^n(3m-1)}\alpha(\GG_2^1)& m\not\equiv 0\mod (3),\ n\geq 1 
& (v_1^{4.3^{n+1}-1},\b v_1,\b^3)  \\
\Delta_{2(3m-1)}\alpha(\GG_2^1)& m\not\equiv 0\mod (3) 
& (v_1^{11},\b v_1,\b^4) \\
\Delta_{6m}\widetilde{\alpha}    & & (v_1)\\
b_{2(9m+2)+1}   & &  (v_1^2,\b)\\
\Delta_{6m+3}\widetilde{\alpha}  & & (v_1^{3},\b v_1) \\
&\\
%\end{array}
%$$                             
%$$
%\begin{array}{ll}
\Delta_{2.3^n(3m+1)}\alpha(\GG_2^1)\b & n\geq 1                   & (v_1,\b^2) \\
\Delta_{2(3m+1)}\alpha(\GG_2^1)\b & m\equiv 0\mod (3)           & (v_1) \\
%&\\
\Delta_{2.3^n(3m-1)}\alpha(\GG_2^1)\b& n\geq 1         & (v_1,\b^2) \\
\Delta_{2(3m-1)}\alpha(\GG_2^1)\b    & m\equiv 0\mod (3) & (v_1,\b^3) \\
&\\
%\end{array}
%$$
%$$
%\begin{array}{ll}
\overline{b}_{3^{n+1}(6m+1)}v_1  & n\geq 0
& (v_1^{6.3^n},\beta) \\ 
\overline{b}_{3^{n}(6m+5)}v_1    & m\equiv 1\mod (3),\ n\geq 2          
& (v_1^{10.3^n},\beta) \\
\overline{b}_{3^{n}(6m+5)}       & m\equiv 1\mod (3),\ n=0,1                
& (v_1^{10.3^n+1},\b v_1) \\ 
&\\
%\end{array}
%$$
%$$
%\begin{array}{ll}
\Sigma^{48}\overline{\Delta}_{3^n(6m+1)-3}&  n\geq 1  
&(v_1^2,\b^2v_1 )\\
\Sigma^{48}\overline{\Delta}_{3^n(6m+5)-3}&  m\not\equiv 1\mod (3),\ n\geq 1 
&(v_1^2,\b^3v_1)\\
\Sigma^{48}\overline{\Delta}_{3^n(6m+5)-3}&  m\equiv 1\mod (3),\ n\geq 1
&(v_1^2,\b^2v_1)\\
%\end{array}
%$$
%$$
%\begin{array}{ll}
%&\\
\Sigma^{48}\overline{\Delta}_{(6m+1)-3}v_1     &&(v_1,\b^2)\\
\Sigma^{48}\overline{\Delta}_{(6m+5)-3}v_1     &&(v_1)\\
&\\
\Sigma^{48}\overline{\Delta}_{3^n(6m\pm 1)-2}v_1& n\geq 1    
&(v_1^{4.3^n-1},\b v_1,\b^3)\\
\Sigma^{48}\overline{\Delta}_{(6m+1)-2}v_1    & 
&(v_1^{3},\b v_1)\\
\Sigma^{48}\overline{\Delta}_{(6m+5)-2}      &
&(v_1^{4},\b v_1^2,\b^3v_1)\\
\end{array}
$$
$$
\begin{array}{lll}
&\\
\Sigma^{48}\overline{\Delta}_{2m-1}\alpha(\GG_2^1)& m\not\equiv 0\mod(3)     
&(v_1,\b^3) \\
\Sigma^{48}\overline{\Delta}_{6m+5}\alpha(\GG_2^1)   &  &(v_1) \\
\Sigma^{48}\overline{\Delta}_{3^n(6m+1)-3}\alpha(\GG_2^1)& n\geq0&(v_1,\b^2) \\ 
\Sigma^{48}\overline{\Delta}_{3^n(6m+5)-3}\alpha(\GG_2^1)&
m\not\equiv 1\mod(3), \ n\geq 1  &(v_1,\b^3) \\ 
\Sigma^{48}\overline{\Delta}_{3^n(6m+5)-3}\alpha(\GG_2^1)&
m\equiv 1\mod(3),\  n\geq 1 &(v_1,\b^2) \\ 
%\end{array}
%$$
%$$
%\begin{array}{ll}
%&\\
\Sigma^{48}\overline{\Delta}_{6m}\widetilde{\alpha}   &   &(v_1)\\
\Sigma^{48}\overline{\Delta}_{(6m+5)-2}\widetilde{\alpha}& 
m\not\equiv 1\mod(3)  &(v_1)\\
\end{array}
$$
\noindent
modulo the following relations in which module generators 
are put into paranthesis (in order to distinguish between module 
multiplications and generators)
$$
\begin{array}{lllll}
\beta^3(\Delta_{k}\a(\GG_2^1))&=
& \b^2\zeta(\Delta_{k}\b v_1),\  & k=2(3m-1)& m\not\equiv 0\mod(3)   \\
\beta^2(\Delta_{k}\a(\GG_2^1)\b)&=
& \b^2\zeta(\Delta_{k}\b v_1),\    & k=2(3m-1)& m\equiv 0\mod(3)  \\
\b^2(\Sigma^{48}\overline{\Delta}_{k}\a(\GG_2^1))&=
& \b^2v_1\zeta(\Sigma^{48}\overline{\Delta}_{k}), &k=2m-1& m\not\equiv 0\mod (3) \\
\b^2(\Sigma^{48}\overline{\Delta}_{k}\a(\GG_2^1)&=
& \b^2v_1\zeta(\Sigma^{48}\overline{\Delta}_{k}) 
& k=3^n(6m+5)-3& m\not\equiv 1\mod(3),\ n\geq 1 \ .\\
\end{array}
$$
\end{cor}

\begin{proof} In principle this is a straightforward consequence 
of the previous corollary and Theorem \ref{H*G21}. Complications arise 
because some of the time integers are distinguished by their residue 
class modulo $(2)$, some of the time by their residue class modulo $(3)$ 
and some of the time the distinction is more involved. 

The most complicated case is perhaps that of the classes 
$\Sigma^{48}\overline{\Delta}_{m}$ for $m$ even.    
There one uses that an even integer can be uniquely written in the 
form $k-3$ with $k$ odd and an odd integer $k$ 
can be uniquely written either as $3^n(6m\pm 1)$
with $n\geq 1$. This together 
with the previous corollary and Theorem \ref{H*G21} leads to the result 
for the first block of generators involving 
$\Sigma^{48}\overline{\Delta}_{m}$. 
The other blocks can be treated similarly.
%An odd integer can be uniquely written in the form $3^n(6m+1)-2$ 
%or $3^n(6m+5)-2$ with $n\geq 0$ and then the previous corollary 
%and Theorem \ref{H*G21} explain the next block of differentials 
%involving $\Sigma^{48}\overline{\Delta}_{m}$.
\end{proof} 

The following Lemma records immediate consequences 
of the knowledge of the $d_9$-differential in the Adams-Novikov 
spectral sequence for $\pi_*(L_{K(2)}V(1))$. 

%\vfil\eject

\begin{lem}\label{d9} \mbox{ }  The following identities hold in  
the $E_9$-term of the Adams-Novikov spectral sequence 
for $\pi_*(L_{K(2)}V(0))$. 
\begin{itemize} 
\item[a)] Let $k\equiv 2\mod(3)$. Then there are constants 
$\e_k'\in \{\pm 1\}$ such that 
$$
\begin{array}{lll}
d_9(\Delta_{k}\a(\GG_2)\b^2)&=&\ \e_k'\Delta_{k-2}\b^7 \\
d_9(\Delta_{k}\b^2 v_1)&=&\ \e_k'\Delta_{k-2}\widetilde{\a}\b^6 \\
d_9(\Sigma^{48}\overline{\Delta}_k\a(\GG_2)\b)&=
&\ \e_k'\Sigma^{48}\overline{\Delta}_{k-2}\b^6 \\
d_9(\Sigma^{48}\overline{\Delta}_{k}\b v_1)&=
&\ \e_k'\Sigma^{48}\overline{\Delta}_{k-2}\widetilde{\a}\b^5\ . \\
\end{array} 
$$
\item[b)] Let $k\equiv 0,1\mod(3)$. Then 
$$
\begin{array}{lll}
\begin{array}{lll}
d_9(\Delta_{k}\a(\GG_2)\b^2)&=& 0 \\
d_9(\Delta_{k}\b^2 v_1)&=& 0 \\
d_9(\Sigma^{48}\overline{\Delta}_k\a(\GG_2)\b)&=& 0  \\
d_9(\Sigma^{48}\overline{\Delta}_{k}\b v_1)&=& 0\ . \\
\end{array} \end{array} 
$$ 
\end{itemize} 
\end{lem}

\remarkstyle{Remark} By identifying vector space generators 
in the appropriate bidegrees it is easy to see that  
there are unique elements $\l_i,\mu_i,\nu_i\in \FF_3$ such that  
$$
\begin{array}{lll}
d_9(\Delta_{k}\a(\GG_2^1)\b^2)&=&\ \l_5\Delta_{k-2}\b^{7}+
\mu_5\Delta_{k-2}\widetilde{\a}\b^{6}\zeta+
\nu_5 \Sigma^{48}\overline{\Delta}_{k-3}\b^{5}\zeta\\
d_9(\Delta_{k}\b^2v_1)&=&\ \l_6\Delta_{k-2}\widetilde{\a}\b^{6} +
\mu_6\Sigma^{48}\overline{\Delta}_{k-3}\b^{5}+
\nu_6 \Sigma^{48}\overline{\Delta}_{k-3}\widetilde{\a}\b^{4}\zeta\\
d_9(\Sigma^{48}\overline{\Delta}_k\a(\GG_2^1)\b)&=
&\ \l_7\Sigma^{48}\overline{\Delta}_{k-2}\b^{6}+
\mu_7\Sigma^{48}\overline{\Delta}_{k-2}\widetilde{\a}\b^{5}\zeta\\
d_9(\Sigma^{48}\overline{\Delta}_{k}\b v_1)&=
&\ \l_8\Sigma^{48}\overline{\Delta}_{k-2}\widetilde{\a}\b^{5}\ . \\
\end{array} 
$$  
As before naturality and the geometric boundary theorem 
applied to the resolution (\ref{hom}) 
allow to show the lemma modulo elements of lower filtration, 
i.e. to determine the values of the $\l_i$. Again the Lemma 
confirms these values and also determines $\mu_i$ and $\nu_i$, and  
formally $\mu_i$ and $\nu_i$ can be deduced by simply replacing 
in the differentials for $E_2^{hG_{24}}\wedge V(0)$ the 
elements $\a$ by $\a(\GG_2)$. 

%, and as before the lemma  
%determines the values of $\l_i,\mu_i,\nu_i$ (up to a common sign). 

\begin{proof} The proof resembles that of Lemma \ref{d5}. 
We start with the class $\Delta_k\b^2 v_1$. From (\ref{various-delta1}) 
and $\b$-linearity of $\delta^1_{\GG_2}$ we get    
$$
\delta^1_{\GG_2}((\omega^2u^{-4})^{3k+2}\a(\GG_2)\b) = 
\pm \Sigma^{4}\Delta_{k}\b^2 v_1, \ \ \ 
\delta^1_{\GG_2}((\omega^2u^{-4})^{3k-4}\b^6) = 
\pm \Sigma^{4}\Delta_{k-2}\widetilde{\a}\b^6 
$$ 
and as before the geometric boundary theorem and Theorem \ref{LK(2)V(1)}.c 
yield the value of the differential (with a suitable constant $\e_k'$).  

The case of $\Sigma^{48}\overline{\Delta}_{k}\b v_1$ 
can be treated similarly but we can also use the strategy 
described in the remark above. Again the sign is 
clearly the same as in the previous case.   

The remaining two cases can once again be deduced by using the Bockstein 
$\delta^0_{\GG_2}$ in $H^*(\GG_2,-)$ associated to the short exact sequence 
(\ref{delta^0}). As in the proof of Lemma \ref{d5} we obtain  
$$ 
\begin{array}{ll} 
\delta^0_{\GG_2}(\Delta_{m}\b^2 v_1)=
\Delta_{m}\a(\GG_2)\b^2, &  
\delta^0_{\GG_2}(\Sigma^{48}\overline{\Delta}_{m}\b v_1) = 
\Sigma^{48}\overline{\Delta}_{m}\a(\GG_2)\b \\
\delta^0_{\GG_2}(\Delta_{m}\widetilde{\a}\b^6) = \Delta_{m}\b^7, &
\delta^0_{\GG_2}(\Sigma^{48}\overline{\Delta}_{m}\widetilde{\a}\b^5) = 
\Sigma^{48}\overline{\Delta}_{m}\b^6  \\
\end{array}
$$  
and the geometric boundary theorem gives the result. 
\end{proof}

\begin{cor}\label{D9} The $d_9$-differential in the Adams-Novikov 
spectral sequence for $L_{K(2)}V(0)$ is linear with respect to 
$\FF_3[\b,v_1]\otimes\Lambda(\zeta)$   
and is trivial on all $\FF_3[\b,v_1]\otimes\Lambda(\zeta)$-module 
generators of $E_6^{*,*}$ given in Corollary \ref{E6} except the following: 
$$
\begin{array}{lllll}
d_9(\Delta_m\beta v_1)        &=&\pm\Delta_{m-2}\widetilde{\a}\b^5  
&\ m\equiv 2                  &\mod (3)\\
d_9(\Delta_{6m+5}\a(\GG_2^1)) &=&\pm(\Delta_{6m+3}\b^5+\Delta_{6m+3}\widetilde{\a}\b^4\zeta) &\\
d_9(\Delta_{2(3m+1)}\a(\GG_2^1)\b)&=
&\pm(\Delta_{6m}\b^6+\Delta_{6m}\widetilde{\a}\b^5\zeta)& \\
%\end{array}
%$$
%$$
%\begin{array}{lllll}
d_9(\overline{b}_{18m+11})    &=&
\pm(\Sigma^{48}\overline{\Delta}_{6m}
\b^4+\Sigma^{48}\overline{\Delta}_{6m}\widetilde{\a}\b^3\zeta) &  \\ 
d_9(\Sigma^{48}\overline{\Delta}_{6m+2}v_1)  &=&
\pm\Sigma^{48}\overline{\Delta}_{6m}\widetilde{\a}\b^4 &\\
d_9(\Sigma^{48}\overline{\Delta}_{6m+5}v_1)  &=&
\pm\Sigma^{48}\overline{\Delta}_{6m+3}\widetilde{\a}\b^4 
&\ m\not\equiv 1&\mod (3)\\
d_9(\Sigma^{48}\overline{\Delta}_{6m+5}v_1)
&=&\pm \overline{b}_{3(6m+5)}\b^5           &\ m \equiv 1&\mod (3)  \\ 
d_9(\Sigma^{48}\overline{\Delta}_{6m+5}\alpha(\GG_2^1))
&=&\pm(\Sigma^{48}\overline{\Delta}_{6m+3}\b^5+
\Sigma^{48}\overline{\Delta}_{6m+3}\widetilde{\a}\b^4\zeta) &\ . \\ 
\end{array}
$$
\end{cor}  

\begin{proof}  
Linearity with respect to $\FF_3[\b,v_1]\otimes\Lambda(\zeta)$ is clear. 
Then we note that the $\b$-torsion classes in the $E_6$-term 
are in too low a cohomological degree to interact via $d_9$ and hence 
$d_9$ is trivial on them.  The rest is an 
immediate consequence of Proposition \ref{extensions}, 
Corollary \ref{E6} and the previous Lemma, and the fact 
that $\b$-multipication is injective in the 
relevant bidegrees.  
\end{proof} 

This finally allows us to calculate the homotopy of $\pi_*(L_{K(2)}V(0))$.  

\noindent {\it Proof of Theorem 1.8.}   
By using the last corollary 
it is straightforward to verify that the $E_{10}$-term has the structure 
described in Theorem \ref{E-infty}. Then we see that $E_{10}^{s,*}=0$ 
for $s>11$ and hence there is no room for higher differentials 
and we get $E_{10}=E_{\infty}$. 

\bigbreak

\appendix\section{The Adams-Novikov spectral sequences converging towards 
$\pi_*(E_2^{hG_{24}}\wedge V(0))$, 
$\pi_*(E_2^{hG_{24}}\wedge V(1))$ and $\pi_*(L_{K(2)}V(1))$} 

The behaviour of the spectral sequence for $\pi_*(E_2^{hG_{24}}\wedge V(0))$ 
can be deduced from that for $\pi_*(E_2^{hG_{24}})$. 
We record this in the following result.

%\vfil\eject

\begin{thm}\label{EO2V(0)} \mbox{ } 
\begin{itemize} 
\item[a)] 
The differentials in the Adams-Novikov spectral sequence
$$
%\begin{array}{llll}
E_2^{s,t}\cong 
\F_3[[v_1^6\Delta^{-1}]][v_1,\Delta^{\pm 1},\b,\a,\widetilde{\a}]/
(\a^2,\widetilde{\a}^2,v_1\a,v_1\widetilde{\a},\a\widetilde{\a}+v_1\b)\\ 
\Longrightarrow \pi_{t-s}(E_2^{hG_{24}}\wedge V(0))
%\end{array}
$$  
are linear with respect to 
$\FF_3[\Delta^{\pm 3},v_1,\b,\a,\widetilde{\a}]$. 
The only non{\-}trivial differentials are 
$d_5$ and $d_9$.  They are (redundantly) determined  by 
$$
d_5(\Delta)=\pm\a\b^2, \ \ 
d_5(\Delta\widetilde{\a})=\pm \widetilde{\a}\a\b^2=\mp v_1\b^3\ 
$$   
and  
$$
d_9(\Delta^2\a)=\pm \b^5, \ \ 
d_9(\Delta^2v_1)=\mp\widetilde{\a}\b^4 \ . 
$$ 
\item[b)] 
We have an inclusion of subalgebras 
$$
E_{\infty}^{0,*}\cong 
\FF_3[[v_1^6\Delta^{-1}]][v_1,v_1\Delta,\Delta^{\pm 3}]\ .
$$
In positive filtration $E_{\infty}^{s,t}$ has an $\FF_3$-vector space basis 
given by the $16$ elements which are represented on $E_2$ by 
$\a$, $\b\a$, $\Delta\a$, $\b\Delta\a$, 
$\widetilde{\a}\a=\b v_1$, $\b\widetilde{\a}\a=\b^2v_1$, 
$\Delta\widetilde{\a}\a=\b\Delta v_1$, 
$\b\Delta\widetilde{\a}\a=\b^2\Delta v_1$, 
$\b^j$, $j=1,2,3,4$, 
$\b^k\widetilde{\a}$, $k=0,1,2,3$  and 
their multiples by powers of $\Delta^{\pm 3}$. 
%(There are two $v_1$-multiplications which are exotic {\bf (Do we need this? 
%And wouldn't it be better to include the differnetials in the statement!)}, 
%i.e. not filtration preserving, namely  
%$v_1(\Delta\a)=\widetilde{\a}\b^2$ and 
%$v_1(\Delta\a\b)=\widetilde{\a}\b^3$. They are  
%necessary to make the result compatible with Theorem 16! 
%One should be able to explain them via Toda brackets!) 
\end{itemize}
\end{thm}

\begin{proof} This is a consequence of the behaviour of the 
spectral sequence converging to $\pi_*(E_2^{hG_{24}})$ 
(cf. \cite{GHM}, \cite{GHMR}).  
First one observes that every element in the 
transfer from $\pi_*(E_2)$ is a permanent cycle. Together with the fact 
that $v_1$ is a permant cycle as well, 
this shows that the only classes on the $0$-line which can carry non-trivial 
differentials are the classes $\Delta^kv_1^{\epsilon}$ with $k\in \Z$ and 
$\epsilon =0$ or $k\neq 0$ and $\epsilon=1$.  
Furthermore we recall that the elements $\a$ and $\b$ are permanent cycles in the 
Adams-Novikov spectral sequence converging to  
$\pi_*(E_2^{hG_{24}})$ and $\widetilde{\a}$ in that for 
$V(0)$ detecting homotopy classes fow which we use the same names. 

Now the spectral sequence for $\pi_*(E_2^{hG_{24}})$   
shows $d_5(\Delta)=\pm \a\b^2$ 
and from the fact that the spectral sequence for $\pi_*(E_2^{hG_{24}}\wedge V(0))$ 
is a module over that for $\pi_*(E_2^{hG_{24}})$ we deduce 
$$
d_5(\Delta\widetilde{\a})=\pm\a\widetilde{\a}\b^2=\mp v_1\b^3, \ \ \
d_5(\Delta^kv_1)=0,\ k\equiv 1,2\mod (3)\ ,
$$
and $d_5$ happens to be a derivation.    
In particular we obtain an isomorphism of algebras 
$$
E_6^{0,*}\cong 
\FF_3[[v_1^6\Delta^{-1}]][\Delta^{\pm 3},v_1,\Delta v_1,\Delta^2 v_1]\ ,
$$
and in positive filtration $E_6^{s,t}$ has an $\FF_3$-vector space basis 
given by the elements  
$\b^k$, $k>0$, $\b^kv_1$, $k=1,2$, $\b^k\a$, $k=0,1$, 
$\b^k\widetilde{\a}$, $k\geq 0$, 
$\b^k\Delta v_1$, $k=1,2$, $\b^k\Delta \a$, $k=0,1$, 
$\b^k\Delta^2v_1$, $k>0$ and $\b^k\Delta^2\a$, $k\geq 0$ and their 
multiples by $\Delta^{\pm 3}$.  

Next the Adams-Novikov spectral sequence for $\pi_*(E_2^{hG_{24}})$ implies  
$d_9(\b^l\Delta^{3k+2}\a)=\pm \b^{l+5}\Delta^{3k}$, hence 
the module structure of the spectral sequence for $\pi_*(E_2^{hG_{24}}\wedge V(0))$ 
with respect to that of $\pi_*(E_2^{hG_{24}})$ gives 
$d_9(\b^{l+1}\Delta^{3k+2}v_1)=d_9(\b^l\Delta^{3k+2}\widetilde{\a}\a)=
\pm \b^{l+5}\Delta^{3k}\widetilde{\a}$ 
and thus $d_9(\b^l\Delta^{3k+2}v_1)=\pm\b^{l+4}\Delta^{3k}\widetilde{\a}$. 
Furthermore sparseness shows that $d_9$ is trivial on all other  
classes of positive cohomological degree. In the resulting $E_{10}$-term 
we have $E_2^{s,*}=0$ for $s>8$ and 
thus there is no more room for higher differentials. 
Hence we get $E_{10}=E_{\infty}$ and the structure of $E_{\infty}$ 
agrees with the stated result.
\end{proof}

The following result has already been proved in \cite{GHM} for the subgroup 
$G_{12}$ instead of $G_{24}$. In fact, part (a) is an immediate consequence 
of the formulae  (\ref{SD16-action}) and (\ref{Frobenius-action}) and 
the fact that the action of $G_{24}$ on $(E_2)_*/(3,u_1)$ factors through 
an action of the quotient $G_{24}/\Z/3=Q_8\subset SD_{16}$. The proof of 
part (b) and (c) 
is analogous to the case of $G_{12}$.
%\footnote{The only minor 
%subtlety in the passage from $G_{12}$ to $G_{24}$ is caused by the 
%fact that  $\omega^2u^{-4}$ does not qualify as a square root of $v_2=u^{-8}$.}  
We leave the details to the reader. 

%\vfil\eject

\begin{thm}\label{homv(1)}(cf.  Theorem 9 of {\cite{GHM}}) \mbox{ } 
\begin{itemize} 
\item[a)] The $E_2$-term of the Adams-Novikov spectral sequence converging to 
$\pi_*(E_2^{hG_{24}}\wedge V(1))$ is given by 
$$
E_2^{*,*}\cong H^*(G_{24},(E_2)_*/(3,u_1))\cong 
\FF_3[\omega^2u^{\pm 4},\b,\a]/(\a^2) \ . 
$$  
\item[b)] The only non-trivial differentials in this 
Adams-Novikov spectral sequence are $d_5$ and $d_9$. They are determined 
by linearity with respect to 
$\FF_3[\omega^2u^{\pm 4},\b,\a]$ and the formulae 
$$
d_5((\omega^2u^{\pm 4})^k)= 
\left\{\begin{array}{lll}
0                     
& \ k\equiv 0,1,2     &  \mod \ (9) \\ 
\pm (\omega^2u^{\pm 4})^{k-3}\a\b^2 
& \ k\equiv 3,4,5,6,7,8 &   \mod \ (9) 
\end{array}\right.
$$
and 
$$
d_9((\omega^2u^{\pm 4})^k\a )=
\left\{\begin{array}{lll}
0                     & \ k\equiv 0,1,2,3,4,5  & \mod \ (9) \\ 
\pm (\omega^2u^{\pm 4})^{k-3}\b^5 & \ k\equiv 6,7,8 & \mod \ (9)  \ .
\end{array}\right.  
$$
\item[c)] 
There are elements in 
$\pi_{8k}(E_2^{hG_{24}}\wedge V(1))$ represented by  
$(\omega^2u^{-4})^k$, $k=0,1,2$,  
and in $\pi_{8k+3}(E_2^{hG_{24}}\wedge V(1))$ 
represented by $(\omega^2u^{-4})^k\a$, $k=0,1,2,3,4,5$,   
such that there is an isomorphism of modules over 
$\FF_9[\Delta^{\pm 3},\b]$ 
$$
\begin{array}{ll} 
\pi_*(E_2^{hG_{24}}\wedge V(1))\cong & 
\FF_9[\Delta^{\pm 3}]\otimes 
\Big(\FF_9[\b]/(\b^5)\{1,\omega^2u^{-4},(\omega^2u^{-4})^2\}\oplus \\
& \oplus \FF_9[\beta]/(\b^2)\{\a,\cdots,(\omega^2u^{-4})^5\a\}\Big)\ .
\end{array} 
$$
\end{itemize} 
\end{thm}

Next we turn towards the algebraic spectral sequence (\ref{algss})
in the case of $M=(E_2)_*/(3,u_1)$ and the Adams-Novikov 
spectral sequence converging towards $\pi_*(L_{K(2)}V(1))$. 
This has already been discussed in \cite{GHM} and \cite{shi}. 
Here we merely translate those results into a form suitable 
for our discussion. As before $v_2$ is defined to be $u^{-8}$.  

%\vfil\eject 
 
\begin{thm}\label{E_1V(1)} (cf. section 4 of \cite{GHM}) 
\mbox{ } 
\begin{itemize} 
\item[a)] As a module over $\F_3[v_2^{\pm 1},\b,\a]/(\a)^2$    
the $E_1$-term of the algebraic spectral sequence 
(\ref{algss}) for $(E_2)_*/(3,u_1)$ is given as follows   
(with $\beta$ and $\a$ acting trivially on $E_1^{s,*,*}$ if $s=1,2$, 
and the elements $e_s$, $s=0,1,2,3$,   
serving as module generators in tridegree $(s,0,0)$. )
$$
E_1^{s,*,*}\cong 
\begin{cases}
\F_3[\omega^2u^{\pm 4}][\beta,\alpha]/(\alpha^2)e_s
& s=0,3 \\
\omega^2u^4\F_3[v_2^{\pm 1}]e_s  & s=1,2 \\ 
0 & s>3\ .
\end{cases}  
$$
%\vfil\eject
\item[b)] The differentials in this spectral sequence are 
$\FF_3[v_2^{\pm 1},\b,\a]$-linear. The only non-trivial differential 
in this spectral sequence is $d_1$ and is determined by 
$$
d_1^{0,0}(\omega^2u^{-12}e_0)=\omega^2u^{-12}e_1, 
\ \ d_1^{1,0}=0, \ \ d_1^{2,0}=0 \ . 
$$  
\item[c)] The following $\b$-extensions 
hold in $H^*(\mathbb{G}_2^1,(E_2)_{*}/(3,u_1))$   
$$ 
\beta \omega^2u^4v_2^ke_2=\pm v_2^k\a e_3\ . 
$$
\item[d)] $H^*(\G_2^1;(E_2)_*/(3,u_1))$ is a free module over 
$\FF_3[v_2^{\pm 1},\b]$ on generators 
$$
e_0,\ \ \a e_0, \ \ \omega^2u^{-4}\a e_0,\ \  
\omega^2u^{-4}\b e_0, \ \ \omega^2u^{-4}e_2, \ \ e_3, 
\ \ \omega^2u^{-4}e_3, \ \  \omega^2u^{-4}\a e_3\ . 
$$  
\item[e)] There is an isomorphism of 
$\FF_3[v_2^{\pm 1},\b]\otimes\Lambda(\zeta)$-modules (even of algebras) 
$$
H^*(\G_2,(E_2)_*/(3,u_1))\cong 
H^*(\G_2^1,(E_2)_*/(3,u_1))\otimes \Lambda(\zeta)\ . 
$$  
%\item[f)] If $\delta^1$ denotes the boundary 
%homomorphism for $H^*(\GG_2,-)$ associated to 
%the short exact sequence 
%$$
%0\to \Sigma^4E_*/(3)\buildrel{.v_1}\over\longrightarrow E_*/(3)
%\longrightarrow E_*/(3,v_1)\to 0 \ . 
%$$ then we have {\bf Behaviour of $\delta^1$!!}
\end{itemize} 
\end{thm} 

\begin{proof} a) The action of $\G_2^1$ on $(E_2)_*/(3,u_1)$ 
is trivial on its Sylow-subgroup $S_2^1$, and on the quotient group 
$\G_2^1/S_2\cong SD_{16}$ the action is given by 
the formulae in (\ref{SD16-action}) and (\ref{Frobenius-action}). 
With this information the calculation of the $E_1$-term is straightforward 
and is left to the reader.  

b) As module over $\FF_3[v_2^{\pm 1},\b,\a]$ the $E_1$-term is generated 
by $e_s$ and $\omega^2u^{-12}e_s$ for $s=0,3$ and $\omega^2u^{-4}e_s$ for $s=1,2$. 
The map of algebraic spectral sequences 
(\ref{algss}) induced by the canonical homomorphism 
$(E_2)_*/(3)\to (E_2)_*/(3,u_1)$ of $\Z_3[[\GG_2^1]]$-modules sends 
$\Delta_{2k}$ to $v_2^{3k}e_0$, $b_{2k+1}$ to $\omega^2 u^{-4}v_2^ke_1$, 
$\overline{b}_{2k+1}$ to $\omega^2 u^{-4}v_2^ke_2$ and 
$\overline{\Delta}_{2k}$ to $v_2^{3k}e_3$. 
This and Theorem \ref{d1} determine the $d_1$-differential and the $E_2$-page. 
The abutment of the spectral sequence is known by Corollary 19 of 
\cite{GHM} and comparing the $E_2$-term with the abutment shows that 
the spectral sequence collapses at its $E_2$-term. 

c) From the same corollary we know that 
$H^*(\s_2,(E_2)_*/(3,u_1))$ is free as module over $\F_9[\b]$,  
hence $H^*(\GG_2,(E_2)_*/(3,u_1))$ and then also 
$H^*(\GG_2^1,(E_2)_*/(3,u_1))$ are free as modules over $\F_3[\b]$. 
This requires nontrivial $\beta$-multiplications on $E_2^{2,0}$ 
and by degree reasons these multiplications must be as claimed.  

d) This is an immediate consequence of (a), (b) and (c). 

e) This is an easy consequence of the isomorphism 
$\G_2\cong \G_2^1\times \Z_3$ and the fact that the central factor $\Z_3$ 
acts trivially on $(E_2)_*/(3,u_1)$. 
\end{proof} 

Next we note that as before the existence of the resolution 
(\ref{hom}) of \cite{GHMR} puts additional 
restrictions on the Adams-Novikov differentials for 
$L_{K(2)}V(1)$. In fact, again by naturality and the geometric boundary 
theorem, these differentials can be easily read off from those 
for $E_2^{hG_{24}}\wedge V(1)$, at least modulo 
the filtration on $H^*(\G_2,(E_2)_*/(3,u_1))$ determined by the 
algebraic resolution (\ref{algss}).      
It turns out that the potential terms of lower filtration 
are always trivial although showing this requires a non-trivial 
effort which has essentially been carried out in \cite{GHM}.  

%\vfil\eject

\begin{thm}\label{LK(2)V(1)}(cf. section 3 of {\cite{GHM}}) 
\begin{itemize} 
\item[a)] The only non-trivial differentials in the 
Adams-Novikov spectral sequence converging to $\pi_*(L_{K(2)}V(1))$ 
are $d_5$ and $d_9$. They are both 
determined by the fact that they are linear with respect to 
$\FF_3[v_2^{\pm 9},\b]\otimes\Lambda(\zeta)$ 
and by the following formulae in which we identify
$v_2^{k+3}e_3$ with $\Sigma^{48}v_2^ke_3$ etc. . 

\item[b)]  The differential $d_5$ is given by 
$$
\begin{array}{lll}  
d_5(v_2^k\a e_0)&=&0\\  
d_5(v_2^k\omega^2u^{-4}\a e_0)&=&0 \\ 
d_5(v_2^k\omega^2u^{-4} e_2)&=&0\\
d_5(\Sigma^{48}v_2^k\omega^2u^{-4}\a e_3)&=&0 \\ 
\end{array} 
$$
for all $k$, and  
$$
\begin{array}{lll}
d_5(v_2^ke_0)&= &
\left\{\begin{array}{lll}
0                                      &\ \ \ \ k\equiv 0,1,5     &\ \mod (9) \\ 
\pm v_2^{k-2}\omega^2u^{-4}\a\b^2e_0   &\ \ \ \ k\equiv 2,3,4,6,7,8 &\ \mod (9) \\
\end{array}\right. \\ 
&\\
d_5(v_2^k\omega^2u^{-4}\b e_0)&=&
\left\{\begin{array}{lll}
0                   & \ \ \ \ \ \ \ \ \ \ \ \ k\equiv 0,4,5     &\ \mod (9) \\ 
\pm v_2^{k-1}\a\b^3e_0   &\ \ \ \ \ \ \ \ \ \ \ \ k\equiv 1,2,3,6,7,8 &\ \mod (9) \\ 
\end{array}\right. \\
%\end{array}
%$$ 
%$$
%\begin{array}{lll}
&\\
d_5(\Sigma^{48}v_2^k e_3)&=&
\left\{\begin{array}{lll}
0                                              &  k\equiv 0,1,5       &\ \mod (9) \\ 
\pm \Sigma^{48}v_2^{k-2}\omega^2u^{-4}\a\b^2e_3&  k\equiv 2,3,4,6,7,8 &\ \mod (9)\  \\ 
\end{array}\right. \\
&\\
d_5(\Sigma^{48}v_2^k\omega^2u^{-4}e_3)&=&
\left\{\begin{array}{lll}
0                                     &\ \ \ \ \ \ \ \ k\equiv 0,4,5       &\ \mod (9) \\ 
\pm \Sigma^{48}v_2^{k-1}\a\b^2e_3     &\ \ \ \ \ \ \ \ k\equiv 1,2,3,6,7,8 &\ \mod (9)\ . \\
\end{array}\right. \\
&\\
\end{array}
$$ 
\item[c)] The differential $d_9$ is given by 
$$
\begin{array}{llllll}
d_9(v_2^k e_0)&=&0                     &&\  k\equiv 0,1,5   &    \ \mod  (9)\\
d_9(v_2^k\omega^2u^{-4}\b e_0)&=&0  &&\  k\equiv 0,4,5   &    \ \mod  (9)\\
d_9(\Sigma^{48}v_2^k e_3)     &=&0          &&\  k\equiv 0,1,5   &    \ \mod (9)\\
d_9(\Sigma^{48}v_2^k\omega^2u^{-4}\b e_3)&=&0 && \ k\equiv 0,4,5&\ \mod (9)    \\ 
&\\
\end{array}
$$
%and 
$$
\begin{array}{lll}
d_9(v_2^{k}\a e_0)&=&
\begin{cases} 
0                           &\   \ \ \ \ \ \ \ \ \ \ \ k\equiv 0,1,2,5,6,7 \mod (9) \\
\pm v_2^{k-3}\b^5 e_0       &\   \ \ \ \ \ \ \ \ \ \ \ k\equiv 3,4,8\ \ \ \ \ \ \ \ \mod (9) 
\end{cases} \\
&\\
d_9(v_2^{k}\omega^2u^{-4}\a e_0)&=&
\begin{cases} 
0                                   &\ \ \ \ \ k\equiv 0,1,2,4,5,6 \mod (9) \\
\pm v_2^{k-3}\omega^2u^{-4}\b^5e_0  &\ \ \ \ \ k\equiv 3,7,8\ \ \ \ \ \ \ \ \mod (9) 
\end{cases} \\
%\end{array}
%$$
%$$\begin{array}{lll}
&\\
d_9(v_2^{k+2}\omega^2u^{-4}e_2)&=& 
\begin{cases} 
0                               &\ \ \ \ \ \ \ \ k\equiv 0,1,2,5,6,7 \mod (9) \\
\pm \Sigma^{48}v_2^{k-5}\b^4e_3 &\ \ \ \ \ \ \ \ k\equiv 3,4,8\ \ \ \ \ \ \ \ \mod (9) 
\end{cases} \\ 
&\\ 
d_9(\Sigma^{48}v_2^{k}\omega^2u^{-4}\a e_3)&=& 
\begin{cases} 
0                                            &  k\equiv 0,1,2,4,5,6 \mod (9)\\
\pm \Sigma^{48}v_2^{k-3}\b^5\omega^2u^{-4}e_3&  k\equiv 3,7,8\ \ \ \ \ \ \ \ \mod (9) \ .
\end{cases} \\
&\\
\end{array} 
$$
\item[d)] As a module over $P=\FF_3[v_2^{\pm 9},\b]\otimes \Lambda(\zeta)$ 
there is an isomorphism
$$
\begin{array}{lll} 
\pi_*(L_{K(2)}V(1)) \cong 
& P/(\b^5)\{v_2^{k}e_0\}_{k=0,1,5}
& \oplus \ P/(\b^3)\{v_2^{k}\a e_0\}_{k=0,1,2,5,6,7}\\
& \oplus \ P/(\b^4)\{v_2^{k}\omega^2u^{-4}\b e_0\}_{k=0,4,5} 
& \oplus \ P/(\b^2)\{v_2^{k}\omega^2u^{-4}\a e_0\}_{k=0,1,2,4,5,6}  \\
& \oplus \ P/(\b^4)\{\Sigma^{48}v_2^{k}e_3\}_{k=0,1,5}
& \oplus \ P/(\b^3)\{v_2^{k+2}\omega^2u^{-4}e_2\}_{k=0,1,2,5,6,7}  \\
& \oplus \ P/(\b^5)\{\Sigma^{48}v_2^{k}\omega^2u^{-4}e_3\}_{k=0,4,5}
& \oplus \ P/(\b^2)\{\Sigma^{48}v_2^{k}\omega^2u^{-4}\a e_3\}_{k=0,1,2,4,5,6}  \ .
\end{array}
$$
\end{itemize} 
\end{thm} 

\begin{proof}  This is an immediate reformulation of the 
main theorem of \cite{GHM}. We just have to use the following 
dictionary which translates between the 
$\FF_9[v_2^{\pm 1},\b]$-module generators used in Corollary 19 of 
\cite{GHM} and those of $H^*(\s_2^1,(E_2)_*/(3,u_1))$ 
of Theorem \ref{E_1V(1)}\footnote{For the translation we pass from 
$H^*(\GG_2^1;(E_2)_*/(3,u_1))$ to 
$H^*(\s_2^1;(E_2)_*/(3,u_1))$ in Theorem \ref{E_1V(1)} and from 
$H^*(\s_2;(E_2)_*/(3,u_1))$ to $H^*(\s_2^1;(E_2)_*/(3,u_1))$  
in \cite{GHM}. This passage  
is straightforward and left to the reader.}. 

By degree reasons the generators $1$, $\a$, $v_2^\frac{1}{2}\a$, 
$v_2^{\frac{1}{2}}\b$, $\a a_{35}$ and $v_2^{\frac{1}{2}}\b\a a_{35}$ 
of \cite{GHMR} must correspond, up to a unit in $\FF_9$, to $e_0$, 
$\a e_0$, $\omega^2u^{-4}\a e_0$, $\omega^2u^{-4}\b e_0$, 
$\omega^2u^{-4}v_2^2e_2$ and $\Sigma^{48}\omega^2u^{-4}\a e_3$ of 
Theorem \ref{E_1V(1)}. 
The generators $\b a_{35}$ and $v_2^{\frac{1}{2}}\b a_{35}$ 
are not determined (not even up to a unit) by their bidegree, 
but if one takes into account that they are in the kernel 
of the restriction map to $H^*(G_{12};(E_2)_*/(3,u_1))$ 
then they are also determined up to a unit and they must 
therefore agree, up to a unit, with the elements 
$\Sigma^{48}e_3$ and $\Sigma^{48}\omega^2u^{-4}e_3$. 
\end{proof}

%\nocite{*}
\bibliographystyle{amsplain}
\bibliography{bibhkm}
\bigbreak
\bigbreak  

%\begin{thebibliography}{GHMR1}
%\end{thebibliography}
\end{document}